\documentclass{article}

\usepackage{PRIMEarxiv}

\usepackage[utf8]{inputenc} 
\usepackage[T1]{fontenc}    
\usepackage{hyperref}       
\usepackage{url}            
\usepackage{booktabs}       
\usepackage{amsfonts}       
\usepackage{nicefrac}       
\usepackage{microtype}      
\usepackage{lipsum}
\usepackage{fancyhdr}       
\usepackage{graphicx}       
\graphicspath{{media/}}     

\usepackage{amsmath}
\allowdisplaybreaks[4]

\usepackage{amssymb}

\newtheorem{theorem}{Theorem}[section]
\newtheorem{corollary}{Corollary}[section]
\newtheorem{lemma}{Lemma}[section]
\newtheorem{proposition}{Proposition}[section]
\newtheorem{definition}{Definition}[section]
\newtheorem{example}{Example}[section]
\newtheorem{remark}{Remark}[section]

\newenvironment{proof} {\textsc{Proof}\quad} {\hfill $\Box$\\}

\title{A Topological Representation of Semantics of First-order Logic and Its Application as a Method in Model Theory
}

\author{
  Yunfei Qin\\
  Peking University\\
  \texttt{qyf@pku.edu.cn} \\
}

\begin{document}
\maketitle

\begin{abstract}
Various topological concepts are often involved in the research of mathematical logic, and almost all of these concepts can be regarded as developing from the Stone representation theorem. In the Stone representation theorem, a Boolean algebra is represented as the algebra of the clopen sets of a Stone space. And based on this, a natural connection is established between the structure of Stone space and the semantics of propositional logic. In other words, models of a propositional theory are represented as points in a Stone space. This enables us to use the concepts of topology to describe many facts in logic. In this paper, we do the same thing for the first-order logic. That is, we organize the basic objects of semantics of first-order logic, such as theories, models, elementary embeddings, and so on, into a kind of topological structure defined abstractly. To be precise, this kind of structure is a kind of enriched-topological space which we call cylindric space in this paper. Furthermore, based on this topological representation of semantics of first-order logic, we systematically introduce a method of point-set topology into the research of model theory. We demonstrate the great advantages of this topological method with an example and provide a general discussion of its features, advantages, and connection to the type space.
\end{abstract}


\section{Introduction}\label{intro}
Logic has been closely related to topology for a long time, and the most important connection is Stone duality. Stone duality is a further development of the Stone representation theorem of Boolean algebra. In Stone's representation theorem, Boolean algebra is represented as a set field algebra on the power set of its ultrafilters, and this set field algebra has a natural topological structure. In Stone duality, this topological structure is characterized as Stone space \cite{johnstone1982stone}. The connection between Stone duality and logic is that Stone duality can be regarded as a syntax-semantic duality of propositional logic: on the one hand, it is well known that the Lindenbaum-Tarski algebra of the syntactic structure of propositional logic is a Boolean algebra; On the other hand, since the Boolean homomorphisms from a Boolean algebra to 0-1 algebra correspond to the ultrafilters on the Boolean algebra, the points in the dual Stone space of this Boolean algebra can be regarded as such Boolean homomorphisms, and these homomorphisms can be regarded as an abstraction of the model of propositional logic theory (0-1 assignment to propositional variables) \cite{givant2008introduction}. Based on this, a Stone space can be regarded as a topological space consisting of all the models of a propositional theory. In this way, Stone duality establishes an algebraic-geometric duality between the syntactic structure of propositional logic theory and its semantic interpretation and gives us a systematic topological perspective on the semantics of propositional logic.

As a generalization of propositional logic, first-order logic naturally inherits the topological intuition . Various topological concepts are common in the study of first-order model theory. Even in many early works, the topology method has a place (two obvious examples can be seen in \cite{feferman1952h} and \cite{suzuki1970orbits}). Thus, a question arises: to what extent can we develop these topological concepts and intuitions in first-order logic? Can we organize the models of first-order theory 
into a certain topological structure, just as organizing the models of propositional logic as a Stone space? This paper gives a positive answer to these questions. That is, the development in this direction is not only possible but also has considerable potential to be applied as a systematic method to the further study of first-order model theory.

First, we review the existing works on topologizing first-order logic.

Unlike propositional logic, a first-order structure has a natural way of constructing a set field representation algebra of it, that is, cylindric set algebra. Each of these algebras is a set field algebra obtained by adding some non-Boolean operators to the set field Boolean algebra consisting of sequences of length $\alpha$($\alpha$-sequence, in short) on a set. For a first-order structure, consider the set $S$ consisting of all $\omega$-sequences on the domain of this first-order structure. Each definable set of this structure can naturally correspond to a subset of $S$, and all such subsets of $S$ form a cylindric set algebra. Obviously, this algebra generated by the first-order structure is a Boolean algebra and thus naturally has a zero-dimensional topological structure.

Cylindric set algebras are related to cylindric algebra proposed by Henkin, Monk, and Tarski in \cite{henkin1981cylindric}. Chapter 1 of \cite{henkin1971cylindric} mentions that the cylindric algebra of local-finite dimension is the algebraic abstraction of Lindenbaum-Tarski algebras of theories of the first-order language with equality and without functions and constants. Generally speaking, some but not every cylindric algebra can be represented as a cylindric set algebra. The Resek-Thompson theorem tells us that the cylindric set algebra is a set field representation of a modified version of cylindric algebra (see \cite{andreka1988stone}). However, considering the soundness and completeness of first-order logic, it is easy to see that each cylindric set algebra generated by a first-order structure is the set field representation of a local-finite dimensional cylindric algebra.



In addition to the topology on the cylindric set algebra of the first-order structure mentioned above, there has been much research on constructing the corresponding topological space of cylindric algebra in the following two early works.

In \cite{comer1972sheaf}, as a generalization of Stone duality, Comer Stephen dualized a cylindric algebra to a class of sheaves derived from it, which are, of course, based on topological space. However, he did not relate this work to the semantics of first-order logic.

In \cite{pinter1980topological}, Charles Pinter investigated the Stone dual space of the underlying Boolean algebra of the local-finite dimensional cylindric algebra (in fact, it is not a standard cylindric algebra but an equivalent form of cylindric algebra proposed in \cite{pinter1973simple}). He showed that a cylindrification $c_{\kappa}$ of a cylindric algebra could be dualized as the following equivalent relation $\sim_{\kappa}$ on the dual space:
$$p\sim_{\kappa}q\Leftrightarrow \{c_{\kappa}a:a\in p\}\subseteq q.$$
On this basis, he showed how to construct a first-order structure based on an ultrafilter with Henkin property of algebraic elements. However, his discussion was limited to the case of countable-dimension, so that only countable models can be constructed in this way. Moreover, the continuous mappings between spaces were not mentioned.

There are also works related to the topologization of semantics of first-order logic that does not involve cylindric algebra. Based on the results of Butz and Moerdijk in \cite{butz1996representing}, Awodey and Forssell constructed a first-order logical duality based on category logic in \cite{awodey2013first}. They established a duality between the Boolean coherent category representing the first-order theory and a kind of topological groupoids, in which the dual topological groupoid of the Boolean coherent category representing a theory $T$ can be regarded as consisting of enough $T$-models. This duality, although providing a topological representation of first-order structures, did not consider how to represent the elementary embeddings between the structures, which are clearly an important part of the semantics of first-order logic.


The work in this paper can be seen as a development of Pinter's investigation of the Stone dual space of cylindric algebra. However, this paper will not deal with cylindric algebra in technical detail. It is easy to see that the ultrafilters on the Lindenbaum algebra of a first-order theory are essentially the algebraic counterpart of the maximal consistent sets in logic, so by replacing the ultrafilters in the Stone dual space with the MCSs, we obtain a space consisting of first-order logic objects, which we call model space in this paper (For a strict definition see \ref{sglz}). Then a direct connection between topology and first-order logic is established without the need to go through the cylindric algebra as an intermediary.

As the title of this paper implies, this paper consists of two parts, one on the topological representation of the semantics of first-order logic, which forms the main body of the paper, and the other on the application of this topological representation as a method in model theory.

As mentioned earlier, the semantic models of propositional logic can be represented one-to-one as points in Stone spaces, and we want to do something similar for first-order logic. That is, we want to find a kind of abstractly defined topological structure such that the construction of such topological structures can completely and systematically represent the classical semantics of first-order logic. This work constitutes the first part of this paper, which consists of Sections \ref{yzkj}, \ref{sbw}, and \ref{tph}.

It is easy to see that the equivalence relation $\sim_{\kappa}$ on the Stone dual space of a Lindenbaum algebra is transformed into the following equivalence relation when we turn the Stone dual space into a model space:
$$p\sim_{\kappa}q\Leftrightarrow \{\exists v_{\kappa}\phi:\phi\in p\}\subseteq q.$$
In fact, we can define a similar equivalence relation $\sim'_{\kappa}$ on the cylindric set algebra:
$$s\sim'_{\kappa}t\Leftrightarrow s\backslash s(\kappa)=t\backslash t(\kappa).$$
The above discussion implies that we can transform the models and theories of first-order logic into the same class of structures, which should be topological spaces with a family of additional equivalence relations. In Section \ref{bco}, we will define these structures strictly and call them cylindric spaces. In addition, we will define several mappings between cylindric spaces (Strongly continuous mapping, cylindric mapping, basis-preserving mapping) and discuss some basic properties of these spaces and mappings. In Section \ref{sglz}, as examples of the cylindric space, we will strictly define the model space and verify that the cylindric set algebra constructed from a first-order structure naturally constitutes a cylindric space (called topologization space). Further, we will also show how to represent the relations between theories and their models as mappings between spaces.

For a complete topological representation of the semantics of first-order logic, two issues must be addressed:
\begin{enumerate}
    \item For a theory with infinite models, all of its models constitute a proper class, so we cannot represent all of them as points in the same space as propositional logic does.
    \item Unlike the case of propositional logic, there are formalized connections between models of first-order logic theories like elementary embeddings, and a complete topological representation should also consider such connections.
\end{enumerate}
In Section \ref{sbw}, we develop some techniques based on the cylindric space to solve these problems. The idea is that, for the first problem, there may not be enough points ``in" a space $\mathcal{C}$, But we can construct for each ordinal $\alpha$ an unique space that is obtained by expanding $\mathcal{C}$ ($\alpha$-expansion space of $\mathcal{C}$) so that we have enough points ``associated with" $\mathcal{C}$. For the second problem, we will define a kind of set-topological relation between points in a cylindric space (factor, permutation on point) to be used in the next work to represent the elementary embedding.

In Section \ref{tph}, based on the previous work, we finally give a topological representation of the first-order logical semantics. Technically, we build this representation using the topologization spaces of the first-order structures mentioned above as mediators. Intuitively, for a theory $T$, all $T$-models are precisely all first-order structures whose topologization spaces can be mapped onto the model space of $T$ . Calling the sequences in the topologization space that list all elements of the first-order structure as domain points, the points in the model space (or its expansion space) to which the domain points are mapped can be regarded as representing the first-order structures. These points representing first-order structures naturally have a set-topological 
definition (that is, the set-topological correspondence property of Henkin set). Based on this, the elementary embeddings between first-order structures can then be represented as mappings between the topologization spaces of these first-order structures and thus as relations between points of the kind we defined in Section \ref{sbw}.

The first part can be seen as part of an overall work: constructing a first-order analog of the Stone duality. In a forthcoming paper, we will prove that there is a categorical duality between the category of the cylindric algebra and the category consisting of the Stone cylindric space and the cylindric mapping, which will then naturally carry the meaning of the Syntax-Semantics dual of the first-order logic, based on the work of this paper.

In the second part, Section \ref{yy}, we introduce a topological method to model theory systematically based on the first part's work. In this part, we first discuss the connection between the model space and the type space, the classical set-topology concept in model theory. And then, take one theorem in model theory (which can be regarded as a critical intermediate theorem in the classical proof of Morley's theorem) as a example to show the application of the topological method in model theory. Moreover, we also compare the topological proof with the corresponding classical proof. As we will see, this new method greatly reduces the complexity of the proof.

In fact, the features and advantages shown in the example are not limited to the example. We have already had a complete topological proof of Morley's theorem, and the advantages presented in the example are general in this complete proof (although it is much more compact than the classical proof, it is still too long to attach here).

It is worth noting that the second part of this paper does not rely very much on the discussion of the abstract structure of the cylindric space in the first part. What I mean is that when considering the application of topological methods to model theory, the only cylindric spaces that need to be considered are concrete model spaces. Moreover, most of the tools and properties we introduce in the first part and will use in the application examples can be constructed or proved on these concrete model spaces by non-topological and sometimes relatively less complex processes. However, to simplify the discussion, we will not discuss these non-topological construction processes in this paper.
\section{Cylindric space}\label{yzkj}
\subsection{Basic concept}\label{bco}
This section will define some fundamental concepts, including cylindric space, and discuss their properties. As mentioned in Section \ref{intro}, cylindric space is a topological space with additional components. Since the direct definition of this structure will be very complex, and its property partly depends on the properties of its components, we first define and discuss these components, then define the whole structure.

\begin{definition}\label{cs}
Let $(S,\tau)$ be a topological space, $\{\sim_{i}:i\in I\}$ be an infinite family of equivalence relations on $S$,
\begin{itemize}
    \item If for any $a,b\in S$, any $i,j\in I$, $a\sim_{i}c\sim_{j}b$ for some $c\Leftrightarrow a\sim_{j}c'\sim_{i}b$ for some $c'$, we say that $\{\sim_{i}:i\in I\}$ is \textbf{commutative};
    \item For a basis $B$ of $\tau$, we call $B$ a \textbf{cylindric basis} about $\{\sim_{i}:i\in I\}$ if $B$ is closed under taking union, intersection, complement and $\sim_{i}$-saturation\footnote{For a subset $T\subseteq S$, $\sim_{i}$-saturation of $T$ writen as $[T]_{i}$ is the set $\{a\in S:\exists b\in T,a\sim_{i} b\}$; If $T=[T]_{i}$, we say $T$ is $\sim_{i}$-saturated.}.
    \item We call $\{\sim_{i}:i\in I\}$ a \textbf{cylindric system} if it is commutative and has a cylindric basis.
\end{itemize}
\end{definition}
\begin{proposition}
Let $(S,\tau)$ be a topological space, $\{\sim_{i}:i\in I\}$ be a cylindric system on it, we have:
\begin{enumerate}
    \item $(S,\tau)$ is a zero-dimensional space.
    \item If $(S,\tau)$ is compact, then $\{u\in\tau:u\mbox{ is a clopen set}\}$ is the only cylindric basis about $\{\sim_{i}:i\in I\}$.
\end{enumerate}
\end{proposition}
\begin{proof}$ $

\begin{enumerate}
    \item Since the cylindric basis is closed under taking the complement, the first item holds.
    \item By definition of the cylindric system, $\{\sim_{i}:i\in I\}$ has a cylindric basis. Choosing an arbitrary cylindric basis $B$ about it, by definition of cylindric basis, $B\subseteq\{u\in\tau:u\mbox{ is clopen}\}$. Conversely, for an arbitrary clopen set $u\in\tau$, by compactness, there is finite $B'\subseteq B$ s.t. $u=\bigcup B'$. By definition of cylindric basis, $u\in B$.
\end{enumerate}
\end{proof}
\begin{proposition}\label{clp}
For a topological space $(S,\tau)$, a cylindric system $\{\sim_{i}:i\in I\}$ on it and any $U\subseteq\tau$, $i\in I$, we have: $[\bigcup U]_{i}=\bigcup\{[u]_{i}:u\in U\}$.
\end{proposition}
\begin{proof}
$a\in [\bigcup U]_{i}\Leftrightarrow\exists b\in\bigcup U, a\sim_{i}b\Leftrightarrow\exists u\in U,b\in u, a\sim_{i}b\Leftrightarrow a\in\bigcup\{[u]_{i}:u\in U\}$.
\end{proof}

It is easy to see that the cylindric system is the topological counterpart of quantifiers. Now we discuss the topological counterpart of equality.
\begin{definition}\label{df}
Let $(S,\tau)$ be a topological space, $\mathcal{E}=\{\sim_{i}:i\in I\}$ be a cylindric system on it, a \textbf{diagonal family} on $\mathcal{E}$ is a family of clopen sets $\{D_{ij}:i,j\in I\}$ satisfying the following conditions: for any $i,j,k\in I$,
\begin{itemize}
    \item If $k\in I\backslash\{i,j\}$, then $D_{ij}$ is $\sim_{k}$-saturated;
    \item If $i\neq j$, then for any $a\in S$, $\|[a]_{i}\cap D_{ij}\|=\|[a]_{j}\cap D_{ij}\|=1$\footnote{For a set $X$, by $\|X\|$ we denote the cardinality of $X$}; this implies $S=[D_{ij}]_{i}=[D_{ij}]_{j}$;
    \item $D_{ij}\cap D_{jk}\subseteq D_{ik}$, $D_{ij}=D_{ji}$.
\end{itemize}
We call elements in the diagonal family diagonals.
\end{definition}

\begin{lemma}
Let $\{D_{ij}:i,j\in I\}$ be a diagonal family on some cylindric system on some topological space, for any $i,j,k\in I$,
\begin{enumerate}
    \item For any $u$, if $k\not\in\{i,j\}$, then $[D_{ij}\cap u]_{k}=D_{ij}\cap[u]_{k}$;
    \item $D_{ii}=S$; (and therefore, if $i=j$ or $k=j$, then $D_{ik}=D_{ij}\cap D_{jk}$)
    \item If $j\not\in\{i,k\}$, then $D_{ik}=[D_{ij}\cap D_{jk}]_{j}$.
\end{enumerate}
\end{lemma}
\begin{proof}
\
\begin{enumerate}

    \item 
    
    \begin{itemize}
        \item $\subseteq$: If $a\in[D_{ij}\cap u]_{k}$, we know that there is $b\sim_{k}a$ s.t. $b\in D_{ij}\cap u$, by definition of $D_{ij}$, $a\in D_{ij}$, then $a\in D_{ij}\cap[D_{ij}\cap u]_{k}\subseteq D_{ij}\cap[u]_{k}$;
        \item $\supseteq$: If $a\in D_{ij}\cap[u]_{k}$, then $a\in D_{ij}$ and there is $b\sim_{k}a$ s.t. $b\in u$, by definition of $D_{ij}$, $b\in D_{ij}$, then $a\in [D_{ij}\cap u]_{k}$.
    \end{itemize}
    \item Choose an arbitrary $i'\neq i$, we have $D_{ii}$ is $\sim_{i'}$-saturated, then $S=[D_{ii'}]_{i'}=[D_{ii'}\cap D_{i'i}]_{i'}\subseteq[D_{ii}]_{i'}=D_{ii}$. Hence $D_{ii}=S$.
    \item\begin{itemize}
        \item\textbf{Case 1:} $i\neq k$.

Clearly we have $[D_{ij}\cap D_{jk}]_{j}\subseteq[D_{ik}]_{j}=D_{ik}$.

For any $a\in D_{ij}$, $[a]_{k}\subseteq D_{ij}$, then there is $b\sim_{k}a$ s.t. $b\in D_{ij}\cap D_{jk}$. Hence $D_{ij}\subseteq[D_{ij}\cap D_{jk}]_{k}$. By commutativity of cylindric system, $S=[D_{ij}]_{j}\subseteq[[D_{ij}\cap D_{jk}]_{k}]_{j}=[[D_{ij}\cap D_{jk}]_{j}]_{k}$ i.e. $[[D_{ij}\cap D_{jk}]_{j}]_{k}=S$. This means that for any $a, [D_{ij}\cap D_{jk}]_{j}\cap [a]_{k}\neq\emptyset$, then $[D_{ij}\cap D_{jk}]_{j}\cap [a]_{k}=D_{ik}\cap [a]_{k}$. Hence $[D_{ij}\cap D_{jk}]_{j}=D_{ik}$.

\item\textbf{Case 2:} $i=k$.

$[D_{ij}\cap D_{jk}]_{j}=[D_{ij}]_{j}=S=D_{ii}=D_{ik}$.
    \end{itemize}
\end{enumerate}
\end{proof}

Now, let's introduce the primary object that the subsequent work will focus on: cylindric space.

\begin{definition}\label{cysp}
 A \textbf{cylindric space} is a triad $(\mathcal{S},\mathcal{E}_{\alpha},\mathcal{D})$, where $\mathcal{S}$ is a topological space, $\mathcal{E}_{\alpha}$ is a cylindric system on $\mathcal{S}$ with infinite ordinal $\alpha$ as index, $\mathcal{D}$ is a diagonal family on $\mathcal{E}_{\alpha}$.
\end{definition}

For a cylindric space $\mathcal{C}$, we write its underlying set as $|\mathcal{C}|$, underlying topology as $\tau^{\mathcal{C}}$, cylindric system as $\mathcal{E}^{\mathcal{C}}$ and equivalence relations in it as $\sim^{\mathcal{C}}_{i}$, diagonal family as $\mathcal{D}^{\mathcal{C}}$ and diagonals in it as $D^{\mathcal{C}}_{ij}$.

\begin{definition}
\
\begin{itemize}
    \item For an arbitrary cylindric space $\mathcal{C}$, we call the index ordinal $\alpha$ of $\mathcal{E}^{\mathcal{C}}$ the \textbf{dimension of $\mathcal{C}$}. For convenience, we sometimes write $\mathcal{C}$ as $\mathcal{C}_{\alpha}$ to indicate that its dimension is $\alpha$;
    \item For any arbitrary subset $X\subseteq|\mathcal{C}|$, define $\Delta(X):=\{i\in \alpha:X$ is not $\sim_{i}$-saturated$\}$, we call $\Delta(X)$ the \textbf{dimension of $X$};
    \item If there is a cylindric basis $B$ of $\mathcal{C}$ s.t. for any $u\in B$, $\|\Delta(u)\|<\omega$, then we say $\mathcal{C}$ is \textbf{basis-finite}.
\end{itemize}
\end{definition}

Now, we define the mapping between cylindric spaces.
\begin{definition}\label{cm}
\
\begin{itemize}
    \item Let $\alpha\ge\beta$ be two ordinals, $\mathcal{C}_{\alpha},\mathcal{C'}_{\beta}$ be two cylindric space, $f:|\mathcal{C}|\rightarrow|\mathcal{C'}|$ be a mapping between them, we call $f$ a \textbf{strongly continuous mapping} (\textbf{S-mapping} for short) and denote it as $f:\mathcal{C}\rightarrow\mathcal{C'}$ if the following two conditions are satisfied:
    \begin{itemize}
        \item $f:(|\mathcal{C}|,\tau^{\mathcal{C}})\rightarrow(|\mathcal{C'}|,\tau^{\mathcal{C'}})$ is continuous where $f^{-1}[D^{\mathcal{C'}}_{ij}]=D^{\mathcal{C}}_{ij}$ for any $i,j\in\beta$;
        \item $f:(|\mathcal{C}|,\sim_{i})_{i\in\beta}\rightarrow(|\mathcal{C'}|,\sim_{i})_{i\in\beta}$ is a structure homomorphism;
    \end{itemize}
    \item For an S-mapping $f$, if $f^{-1}$ is also an S-mapping, we call $f$ an \textbf{S-homeomorphism}, and denote it as $f:\mathcal{C}_{\alpha}\cong\mathcal{C'}_{\beta}$;
    \item We call an S-mapping $f:\mathcal{C}_{\alpha}\rightarrow\mathcal{C'}_{\beta}$ a \textbf{cylindric mapping} (\textbf{C-mapping} for short), if for any $i\in\alpha$, any $a\in|\mathcal{C}|$, $f[[a]_{i}]$ is a dense subset of $[f(a)]_{i}$. 
    \item Let $f:\mathcal{C}_{\alpha}\rightarrow\mathcal{C'}_{\beta}$ be an S-mapping, we say $f$ is \textbf{basis-preserving}, if for any $u\in\tau^{\mathcal{C}}$ with $\Delta(u)\subseteq\beta$, there is a $u'\in\tau^{\mathcal{C'}}$ s.t. $f^{-1}[u']=u$. 
  \end{itemize}
\end{definition}
\begin{proposition}\label{kjdw}
Let $f:\mathcal{C}_{\alpha}\rightarrow\mathcal{C'}_{\beta}$ be an S-mapping.
\begin{enumerate}
    \item For any $u'\in\tau^{\mathcal{C'}}$, $\Delta(f^{-1}[u'])\subseteq\Delta(u')$;
    \item The following statements are equivalent:
    
    \begin{itemize}
        \item $f$ is a C-mapping.
        \item For any $u'\in\tau^{\mathcal{C'}}$, $i\in\alpha$, $[f^{-1}[u']]_{i}=f^{-1}[[u']_{i}]$.
        \item For any clopen set $u'$ in a cylindric basis, $i\in\alpha$, $[f^{-1}[u']]_{i}=f^{-1}[[u']_{i}]$.
    \end{itemize}
\end{enumerate}
\end{proposition}
\begin{proof}$ $

\begin{enumerate}
    \item For any $i\in\beta$,

\begin{tabular}{rl}
     &$i\not\in\Delta(u')$\\
     $\Leftrightarrow$& for any $a'\in u',b'\in -u',a\not\sim_{i}b$\\
     $\Rightarrow$& for any $a\in f^{-1}[u'],b'\in f^{-1}[-u']=-f^{-1}[u'],a\not\sim_{i}b$\\
     $\Leftrightarrow$&$i\not\in\Delta(f^{-1}[u'])$\\
\end{tabular}

For any $i\in\alpha\backslash\beta$, clearly $i\not\in\Delta(f^{-1}[u'])$ since the image of any $\sim_{i}$-equivalence class is a singleton. Hence, $\Delta(f^{-1}[u'])\subseteq\Delta(u')$.

\item Since for any collection $U$ of open sets, $[\bigcup U]_{i}=\bigcup\{[u]_{i}:u\in U\}$, the second equivalence holds. By Proposition \ref{kjdw}.(1), if $i\in\alpha\backslash\beta$, the first equivalence obviously holds. If $i\in\beta$, it is also easy to see $[f^{-1}[u']]_{i}\subseteq f^{-1}[[u']_{i}]$ holds for any S-mapping $f$. Then we have

\begin{tabular}{rl}
     &$f$ is not a C-mapping\\
     $\Leftrightarrow$& there exists $i\in\alpha$, $u'\in\tau^{\mathcal{C'}}$ and $a\in|\mathcal{C}|$, $a\in f^{-1}[[u']_{i}]$ but $[a]_{i}\cap f^{-1}[u']=\emptyset$\\
     $\Leftrightarrow$& there exists $i\in\alpha$, $u'\in\tau^{\mathcal{C'}}$ and $a\in|\mathcal{C}|$, $a\in f^{-1}[[u']_{i}]$ but $a\not\in[f^{-1}[u']]_{i}$\\
     $\Leftrightarrow$& there exists $i\in\alpha$, $u'\in\tau^{\mathcal{C'}}$, $[f^{-1}[u']]_{i}\neq f^{-1}[[u']_{i}]$
\end{tabular}
\end{enumerate}
\end{proof}

\begin{proposition}\label{mp}
Let $f:\mathcal{C}_{\alpha}\rightarrow\mathcal{C'}_{\alpha'},g:\mathcal{C'}_{\alpha'}\rightarrow\mathcal{C''}_{\alpha''}$ be two S-mapping, then $g\circ f$ is an S-mapping. This statement also holds for C-mapping and basis-preserving C-mapping.
\end{proposition}
\begin{proof}
Clearly $g\circ f$ is continuous and for any $i,j\in\alpha''$, $(g\circ f)^{-1}[D^{\mathcal{C''}}_{ij}]=D^{\mathcal{C}}_{ij}$. By the property of structural homomorphism, $g\circ f$ also forms structural homomorphism for corresponding equivalence relations. Hence, $g\circ f$ is an S-mapping.

\begin{itemize}
    \item \textbf{C-mapping:} For any $u\in\tau^{\mathcal{C''}}$, $i\in\alpha$,  By Proposition \ref{kjdw}.(2), we have $(g\circ f)^{-1}[[u]_{i}]=f^{-1}[g^{-1}[[u]_{i}]]=f^{-1}[[g^{-1}[u]]_{i}]=[f^{-1}[g^{-1}[u]]]_{i}=[(g\circ f)^{-1}[u]]_{i}$ and then $g\circ f$ is a C-mapping. Hence, the case about C-mapping holds.
    \item \textbf{basis-preserving C-mapping:} For any $u\in\tau^{\mathcal{C}}$ with $\Delta(u)\subseteq\alpha''$, there is $u'\in\tau^{\mathcal{C'}}$ s.t. $f^{-1}[u']=u$.
    For any finite $s=\{s_{1},...,s_{n}\}\subseteq\Delta(u')\backslash\Delta(u)$, abbreviating $[...[u']_{s_{1}}...]_{s_{n}}$ to $[u']_{s}$, by Proposition \ref{kjdw} (2), we know $f^{-1}[[u']_{s}]=[f^{-1}[u']]_{s}=[u]_{s}=u$. Let $X:=\Delta(u')\backslash\alpha''$, $v':=\bigcup\{[u']_{s}:s\subseteq X$ is finite$\}$, we know $v'$ is open, $\Delta(v')\subseteq\alpha''$ and $f^{-1}[v']=u$. Then there is $u''\in\tau^{\mathcal{C''}}$ s.t. $g^{-1}[u'']=u'$ then $(g\circ f)^{-1}[u'']=u$. Hence, the case about basis-preserving C-mapping holds.
\end{itemize}
\end{proof}

\begin{lemma}\label{bh}
Basis-preserving S-bijection between two cylindric spaces with the same dimension is S-homeomorphism.
\end{lemma}
\begin{proof}
Let $f:\mathcal{C}_{\alpha}\rightarrow\mathcal{C'}_{\alpha}$ be a basis-preserving S-bijection, We show $f^{-1}$ is also a S-mapping:

It's easy to see that $(f^{-1})^{-1}[D^{\mathcal{C}}_{ij}]=D^{\mathcal{C'}}_{ij}$ and since $f$ is basis-preserving, $f^{-1}$ is continuous. 

For any $a',b'\in|\mathcal{C'}|$, any $i\in\alpha$, if $a'\sim_{i}b'$, then there is an unique $c'\sim_{i}a',b'$ s.t. $c'\in D^{\mathcal{C'}}_{ij}$. Assume $f^{-1}(a')\not\sim_{i}f^{-1}(b')$, then there are different $c\sim_{i}f^{-1}(a'),d\sim_{i}f^{-1}(b')$ s.t. $c,d\in D^{\mathcal{C}}_{ij}$. By definition of S-mapping, $f(c)\in [a']_{i}\cap D^{\mathcal{C'}}_{ij}=[b']_{i}\cap D^{\mathcal{C'}}_{ij}\ni f(d)$, which means $f(c)=c'=f(d)$. This contradicts with bijectivity and therefore $f^{-1}(a')\sim_{i}f^{-1}(b')$.
\end{proof}

\begin{proposition}\label{cstc}
Let $\mathcal{C}_{\alpha}$ be a $T_{2}$ cylindric space, $i\in\alpha$, then $[a]_{i}$ is a closed set for any $a\in|\mathcal{C}|$ and if $\mathcal{C}$ is compact, $[u]_{i}$ is closed for any closed set $u\subseteq|\mathcal{C}|$.
\end{proposition}
\begin{proof}
Let $U$ be the family of all clopen neighborhoods of $[u]_{i}$. We know that $[u]_{i}\subseteq\bigcap U$. For any $v\in U$, we know that $-v\cap [u]_{i}=\emptyset$, then $[-v]_{i}\cap [u]_{i}=\emptyset$, then $[u]_{i}\subseteq-([-v]_{i})\subseteq v$ and $-([-v]_{i})\subseteq U$ is $\sim_{i}$-saturated. Write $U':=\{v\in U:v$ is $\sim_{i}$-saturated$\}$, we know $\bigcap U=\bigcap U'$. So we just need to show $[u]_{i}\not\subset\bigcap U'$ to prove $[u]_{i}$ is closed:
\begin{itemize}
    \item Firstly, consider the case that $u$ is a singleton. Write $u=\{a\}$:
    
    Assume there exists $b\in\bigcap U'\backslash[u]_{i}$, we know $[b]_{i}\subseteq\bigcap U'$ and $[b]_{i}\cap[u]_{i}=\emptyset$. Choose an arbitrary $j\neq i$, we know $\|D_{ij}\cap[a]_{i}\|=\|D_{ij}\cap[b]_{i}\|=1$. Let $a'\in D_{ij}\cap[a]_{i},b'\in D_{ij}\cap[b]_{i}$.

For any clopen $v$, if $a'\in v$, then $a\in [v\cap D_{ij}]_{i}$. Since $[v\cap D_{ij}]_{i}\in U'$, we know $b'\in[v\cap D_{ij}]_{i}$, then $b'\in[v\cap D_{ij}]_{i}\cap D_{ij}$ i.e., $\exists b''\sim_{i}b'$, $b''\in v\cap D_{ij}$ and $b'\in D_{ij}$. By the fact $\|D_{ij}\cap[b']_{i}\|=1$, we know $b'=b''$, then $b'\in v$, vice versa. Since we choose $v$ arbitrarily, $\{w\in\tau:w\mbox{ is clopen and }a'\in w\}=\{w\in\tau:w\mbox{ is clopen and }b'\in w\}$. By property of $T_{2}$ space, $a'=b'$ which contradicts the assumption. Hence there isn't $b\in\bigcap U'\backslash[u]_{i}$, then $[u]_{i}\not\subset\bigcap U'$.

\item Now, consider the general $u$:

Assume there exists $b\in\bigcap U'\backslash[u]_{i}$, we know $[b]_{i}\subseteq\bigcap U'$ and $[b]_{i}\cap[u]_{i}=\emptyset$. For any clopen $v\supseteq u$, we know $[v]_{i}\in U'$ and therefore $[b]_{i}\subseteq [v]_{i}$, then $v\cap [b]_{i}\neq\emptyset$. Since $\{w\in \tau^{\mathcal{C}}:w\mbox{ is clopen and }u\subseteq w\}$ is closed under taking intersect and $[b]_{i}$ is closed set, by compactness we know $\bigcap\{w\in \tau^{\mathcal{C}}:w\mbox{ is clopen and }u\subseteq w\}\cap [b]_{i}=u\cap[b]_{i}\neq\emptyset$, then $b\in[u]_{i}$ which contradicts assumption. Hence $[u]_{i}\not\subset\bigcap U'$.
\end{itemize}
\end{proof}
\subsection{Four examples}\label{sglz}
In the previous subsection, we have defined cylindric spaces and the mappings between them. As introduced in the Introduction, we are trying to describe the structure of first-order logical semantics in the topological language through cylindric space. In this subsection, we respectively use the first-order structure and formulas to construct two specific examples of cylindric space and two specific examples of S-mapping.

\subsubsection*{Example 1:}

Let $\mathcal{L}$ be a first-order language, $\Lambda_{\mathcal{L}}$ be the set of $\mathcal{L}$-formulas, $\{R_{i}:i\in I\}$ be all predicates in $\mathcal{L}$ (as mentioned in Introduction, the only non-logical symbols in $\mathcal{L}$ are predicates), $\mathfrak{A}=(A,R^{\mathfrak{A}}_{i})_{i\in I}$ be an $\mathcal{L}$-structure.

Let $A^{\omega}$ be the $\omega$ power of $A$, for any formula $\phi\in\Lambda_{\mathcal{L}}$, we define $A_{\phi}\subseteq A^{\omega}$ as: $A_{\phi}=\{a:$ for every free variables $x_{i_{1}},...,x_{i_{n}}\mbox{ in }\phi,\langle a({i_{1}})...a({i_{n}})\rangle\in \phi(\mathfrak{A})\}$\footnote{$\phi(\mathfrak{A})$ denotes the subset defined by $\phi$ on $\mathfrak{A}$}. Let $B:=\{A_{\phi}:\phi\in\Lambda_{\mathcal{L}}\}$, it is easy to see that both empty set, and universal set are in $B$ and $B$ is closed under taking intersection, union and complement, empty set, and universal set. Hence $B$ forms a basis of a zero-dimensional topology on $A^{\omega}$. We denote this topology as $\tau^{A}$.

For each $i\in\omega$, we define equivalence relation $\sim^{\mathfrak{A}}_{i}$ on $A^{\omega}$ as: $a\sim^{\mathfrak{A}}_{i}b\Leftrightarrow$ for any $j\neq i$, $ a(j)=b(j)$. Let $\mathcal{S^{\mathfrak{A}}}:=(A^{\omega},\tau^{A})$, for any $i,j\in\omega$, $D^{\mathfrak{A}}_{ij}:=A_{v_{i}=v_{j}}$, $\mathcal{C^{\mathfrak{A}}_{\omega}}:=(\mathcal{S^{\mathfrak{A}}},\{\sim^{\mathfrak{A}}_{i}:i\in\omega\},\{D^{\mathfrak{A}}_{ij}:i,j\in\omega\})$. We have:

\begin{proposition}
$\mathcal{C^{\mathfrak{A}}_{\omega}}$ is a basis-finite cylindric space.
\end{proposition}
\begin{proof}
\
\begin{itemize}
    \item \textbf{$\{\sim^{A}_{i}:i\in\omega\}$ is a cylindric system:}
    
    $[[a]_{i}]_{j}=\{b:$ for $k\not\in\{i,j\},a(k)=b(k)\}=[[a]_{j}]_{i}$, so $\{\sim^{\mathfrak{A}}_{i}:i\in\omega\}$ is commutative. For any $u,v\in B$, we know there exist formulas $\phi,\psi$ s.t. $u=A_{\phi},v=A_{\psi}$. It is easy to see $-A_{\phi}=A_{\neg\phi},A_{\phi}\cap A_{\psi}=A_{\phi\wedge\psi},A_{\phi}\cup A_{\psi}=A_{\phi\vee\psi}$ and for any $i\in\omega$, $[A_{\phi}]_{i}=A_{\exists v_{i}\phi}$. Hence $B$ is closed under taking intersect, union, complement and $\sim^{\mathfrak{A}}_{i}$-saturation. Hence $\{\sim^{\mathfrak{A}}_{i}:i\in\omega\}$ forms a cylindric system.
    \item \textbf{$\{D^{\mathfrak{A}}_{ij}:i,j\in\omega\}$ is a diagonal family on $\{\sim^{A}_{i}:i\in\omega\}$:}
    
    For any pair $i\neq j\in\omega$, any $a\in A^{\omega}$, we define $a_{ij}$ as: $a_{ij}(i)=a(j)$ and for any $k\neq i$, $a_{ij}(k)=a(k)$. By definition of $D^{\mathfrak{A}}_{ij}$, we know $a_{ij}$ is the only point in $D^{\mathfrak{A}}_{ij}\cap[a]_{i}=D^{\mathfrak{A}}_{ji}\cap[a]_{i}$. Hence the intersection of $D^{\mathfrak{A}}_{ij}=D^{\mathfrak{A}}_{ji}$ and any $\sim^{\mathfrak{A}}_{i}$-equivalence class is singleton. For any $k\not\in\{i,j\}$ and any $a'\in [a]_{k}$, we know $a'(i)=a(i)=a(j)=a'(j)$ and therefore $a'\in D^{\mathfrak{A}}_{ij}$. Hence $[a]_{k}\subseteq D^{\mathfrak{A}}_{ij}$. Hence $D^{\mathfrak{A}}_{ij}$ is $\sim_{k}$-saturated.
    
    For any $i,j,k\in\omega$, $a\in D^{\mathfrak{A}}_{ij}\cap D^{\mathfrak{A}}_{jk}\Rightarrow a(i)=a(j)=a(k)\Rightarrow a\in D^{\mathfrak{A}}_{ik}$. Hence $D^{\mathfrak{A}}_{ij}\cap D^{\mathfrak{A}}_{jk}\subseteq D^{\mathfrak{A}}_{ik}$.
    \item \textbf{For any $u\in B$, $|\Delta(u)|<\omega$:}
    
    For any $u\in B$, we know there is a formula $\phi$ s.t. $u=A_{\phi}$. For any $a\in u$ and variable $v_{i}$ that doesn't occur freely in $\phi$, it's easy to see that $[a]_{i}\subseteq A_{\phi}=u$ i.e. $u$ is $\sim^{\mathfrak{A}}_{i}$-saturated. Since there are only finitely many free variables in $\phi$, we know $\Delta(u)<\omega$.
\end{itemize}
\end{proof}

We say that $\mathcal{C^{\mathfrak{A}}_{\omega}}$ is the \textbf{$\omega$-topologization} of $\mathfrak{A}$. We call a sequence $a\in A^{\omega}$ which lists all the elements of $\mathfrak{A}$ a \textbf{domain point} of $\mathcal{C^{\mathfrak{A}}_{\omega}}$. For an assignment $\sigma$ and a sequence $a\in A^{\omega}$, if for each $i\in\omega$, $a(i)=\sigma(v_{i})$, then we say that $a$ \textbf{represent} $\sigma$.

\subsubsection*{Example 2:}

For any langage $\mathcal{L}$, any $\mathcal{L}$-theory $T$ and any $\Phi\subseteq\Lambda_{\mathcal{L}}$, if there is $T$-model $\mathfrak{A}$ and assignment $\sigma$ on $\mathfrak{A}$ s.t. $\Phi=\{\phi\in\Lambda_{\mathcal{L}}:(\mathfrak{A},\sigma)\vDash\phi\}$, we call $\Phi$ a maximal satisfiable set (MSS) of $T$.

Let $S_{T}$ be the family of all MSSs of $T$. For each formula $\phi$, we define $S_{\phi}:=\{a\in S_{T}:\phi\in a\}$. It is easy to see that $B=\{S_{\phi}:\phi\in\Lambda_{\mathcal{L}}\}$ is a topological basis, we write the induced topology as $\tau^{T}$.

For each $i\in\omega$, we define equivalence relation $\sim^{T}_{i}$ as: $a\sim_{i}b\Leftrightarrow\{\phi\in a:v_{i}$ doesn't occur in $\phi$ freely$\}=\{\phi\in b:v_{i}$ doesn't occur in $\phi$ freely$\}$. For each $i,j\in\omega$, we define $D^{T}_{ij}:=S_{v_{i}=v_{j}}$ and write $\mathcal{C}^{T}=((S_{T},\tau^{T}),\{\sim^{T}_{i}:i\in\omega\},\{D^{T}_{ij}:i,j\in\omega\})$. We have:
\begin{proposition}
$\mathcal{C}^{T}$ is a $T_{2}$ basis-finite cylindric space.
\end{proposition}
\begin{proof}
\
\begin{itemize}
    \item \textbf{$\{\sim^{T}_{i}:i\in\omega\}$ is a cylindric system:}
    
    For any $a,b\in S_{T}$,
    
    \begin{tabular}{rl}
         &$b\in[[a]_{i}]_{j}$\\
         $\Leftrightarrow$& for any formula $\phi$ satisfying $v_{i},v_{j}$ doesn't occur in it freely, we have $\phi\in a$ if and only if $\phi\in b$\\
         $\Leftrightarrow$& $b\in[[a]_{j}]_{i}$
    \end{tabular}
    
    Hence $\{\sim^{T}_{i}:i\in\omega\}$ is commutative.
    
    For any $u,v\in B$, there are formulas $\phi,\psi$ s.t. $u=S_{\phi},v=S_{\psi}$. It is easy to see $-S_{\phi}=S_{\neg\phi},S_{\phi}\cap S_{\psi}=S_{\phi\wedge\psi},S_{\phi}\cup S_{\psi}=S_{\phi\vee\psi}$ and for any $i\in\omega$, $[S_{\phi}]_{i}=S_{\exists v_{i}\phi}$. Hence $B$ is closed under intersection, union, complement and $\sim^{T}_{i}$-saturation.
    
    Hence, $\{\sim^{T}_{i}:i\in\omega\}$ is a cylindric system.
    \item \textbf{$\{D^{T}_{ij}:i,j\in\omega\}$ is a diagonal family:}
    
    For any pair $i\neq j\in\omega$ and any $a\in S_{T}$, there is a $T$-model $\mathfrak{A}$ and an assignment $\sigma$ on $\mathfrak{A}$ s.t. $(\mathfrak{A},\sigma)\vDash a$, we define assignment $\sigma'$ as: $\sigma'(v_{i})=\sigma(v_{j})$ and for any $k\neq i$, $\sigma'(v_{k})=\sigma(v_{k})$. Let $a_{\sigma'}:=\{\phi:(\mathfrak{A},\sigma')\vDash\phi\}$, we know $a_{\sigma'}\in[a]_{i}$ and is the only point containing $v_{i}=v_{j}$ in $[a]_{i}$. Hence $\|D^{T}_{ij}\cap[a]_{i}\|=1$. For any $k\not\in\{i,j\}$, any $a\in D^{T}_{ij}$ and any $a'\in [a]_{k}$, by definition, $v_{i}=v_{j}\in a'$, then $a'\in D^{T}_{ij}$. Hence $[a]_{k}\subseteq D^{T}_{ij}$. Hence $D^{\mathfrak{A}}_{ij}$ is $\sim^{T}_{k}$-saturated.
    
    For any $i,j,k\in\omega$, $a\in D^{T}_{ij}\cap D^{T}_{jk}\Rightarrow \{v_{i}=v_{j},v_{j}=v_{k}\}\subseteq a\Rightarrow \{v_{i}=v_{k}\}\subseteq a\Rightarrow a\in D^{T}_{ik}$. Hence $D^{T}_{ij}\cap D^{T}_{jk}\subseteq D^{T}_{ik}$.
    \item \textbf{For any $u\in B$, $\Delta(u)<\omega$:}
    
    For any $u\in B$, there is a formula $\phi$ s.t. $u=S_{\phi}$. For any $a\in u$ and variable $v_{i}$ which does not occur freely in $\phi$, by definition we have $[a]_{i}\subseteq S_{\phi}=u$ i.e. $u$ is $\sim^{T}_{i}$-saturated. Since there are only finitely many free variables in $\phi$, we know $\Delta(u)<\omega$.
    \item \textbf{$T_{2}$ property:}
    
    For any pair $a\neq b\in S_{T}$, we know there exists a formula $\phi$ s.t. $\phi\in a,\neg\phi\in b$, then $a\in S_{\phi},b\in S_{\neg\phi}$. Since $S_{\phi}\cap S_{\neg\phi}=\emptyset$, this means the property of $T_{2}$.
\end{itemize}
\end{proof}

We call $\mathcal{C}^ {T}$ the \textbf{model space} of $T$. It is easy to see from the compactness theorem of first-order logic that any clopen covering in the model space of $T$ must have a finite sub-covering, which means that the model space of $T$ is compact and thus is an FOL space (Definition \ref{dofol} below). We are not in a hurry to write this property into the above proposition here because we will give a topological proof of this property in Theorem \ref{cop}.

\begin{remark}
As mentioned in the Introduction, the work of this paper is inspired by the structure of cylindric set algebra and Charles Pinter's research on Stone dual space of cylindric algebra. This inspiration is embodied in the above two examples. Among them, Example 1 can be regarded as the topological structure naturally derived from the cylindric set algebra on $\mathfrak{A}$; the second example is essentially the ``cylindric dual space” in Charles Pinter's article.
\end{remark}

\subsubsection*{Example 3}

Let $\mathcal{L},\mathcal{L'}$ be two first-order languages s.t. $\mathcal{L'}\le\mathcal{L}$, $T$ be an $\mathcal{L}$-theory, $T'$ be an $\mathcal{L'}$-theory s.t. $T'\subseteq\{\phi\in\Lambda_{\mathcal{L'}}:\phi\in T\}$. Let $f:|\mathcal{C}^{T}|\rightarrow|\mathcal{C}^{T'}|$ be $f(a)=a|_{\mathcal{L'}}$ ($a|_{\mathcal{L'}}:=\{\phi\in\Lambda_{\mathcal{L'}}:\phi\in a\}$), then we have:
\begin{proposition}
$f:\mathcal{C}^{T}\rightarrow\mathcal{C}^{T'}$ is a C-mapping and if $\mathcal{L}=\mathcal{L'}$, then $f$ is basis-preserving.
\end{proposition}
\begin{proof}$ $

\begin{itemize}
        \item \textbf{continuity and preserving diagonal family:}
        
        Here we only consider the case about the clopen sets in the cylindric basis (i.e., $\{S'_{\phi}:\phi\in\Lambda_{\mathcal{L'}}\}$), the case about general open sets is an obvious corollary.
        
        For any formula $\phi\in\Lambda_{\mathcal{L'}}$,
    
    \begin{tabular}{rl}
         &$a\in f^{-1}[S'_{\phi}]$\\
         $\Leftrightarrow$&$\phi\in f(a)$\\
         $\Leftrightarrow$&$\phi\in a$\\
         $\Leftrightarrow$&$a\in S_{\phi}$
    \end{tabular}
    
    Hence $f^{-1}[S'_{\phi}]=S_{\phi}\in\tau^{\mathcal{C}^{T}}$ and particularly $f^{-1}[D^{\mathcal{C}^{T'}}_{ij}]=D^{\mathcal{C}^{T}}_{ij}$.
    \item \textbf{structure homomorphism:}
    
    For any $a,b\in S_{T}$, $i\in\omega$, we have
    
    \begin{tabular}{rl}
         &$a\sim_{i}b$\\
         $\Leftrightarrow$& for any formula $\phi\in\Lambda_{\mathcal{L}}$ in which $v_{i}$ doesn't occur freely, $\phi\in a$ if and only if $\phi\in b$\\
         $\Rightarrow$& for any formula $\phi\in\Lambda_{\mathcal{L'}}$ in which $v_{i}$ doesn't occur freely, $\phi\in f(a)$ if and only if $\phi\in f(b)$\\
         $\Leftrightarrow$&$f(a)\sim_{i}f(b)$
    \end{tabular}
    
    \item \textbf{C-mapping:}
    
    For any $a\in|\mathcal{C}|$ and any $S'_{\phi}\cap[f(a)]_{i}\neq\emptyset$, we know $f(a)\in[S'_{\phi}]_{i}=S'_{\exists v_{i}\phi}$, then $a\in S_{\exists v_{i}\phi}=[S_{\phi}]_{i}$, then there is $b\in[a]_{i}\cap S_{\phi}$, then $f(b)\in S'_{\phi}\cap[f(a)]_{i}$. hence, $f[[a]_{i}]$ is dense in $[f(a)]_{i}$.
    \item \textbf{If $\mathcal{L}=\mathcal{L'}$}
    
    Then for any $\phi\in\Lambda_{\mathcal{L}}$, $S_{\phi}=f^{-1}[S'_{\phi}]$, then $f$ is basis-preserving.
    \end{itemize}
\end{proof}

\subsubsection*{Example 4}

Let $\mathcal{L}$ be a first-order language, $T$ be an $\mathcal{L}$-theory, $\mathfrak{A}$ be a $T$-model. Define $f:|\mathcal{C}^{\mathfrak{A}}_{\omega}|\rightarrow\mathcal|{C}_{T}|$ as $f(a)=\{\phi:(\mathfrak{A},\sigma_{a})\vDash\phi\}$ ($\sigma_{a}$ is the assignment represented by $a$), we have:

\begin{proposition}
$f:\mathcal{C}^{\mathfrak{A}}_{\omega}\rightarrow\mathcal{C}_{T}$ is basis-preserving C-mapping.
\end{proposition}
\begin{proof}$ $

\begin{itemize}
    \item \textbf{continuity and diagonal-preserving:}
    
    Here we only consider the case about the clopen sets in cylindric basis (i.e. $\{S_{\phi}:\phi\in\Lambda_{\mathcal{L}}\}$ and $\{A_{\phi}:\phi\in\Lambda_{\mathcal{L}}\}$), the case about general open sets is an obvious corollary.
    
    For any $\phi\in\Lambda_{\mathcal{L}}$,
    
    \begin{tabular}{rl}
         &$a\in f^{-1}[S_{\phi}]$\\
         $\Leftrightarrow$&$\phi\in f(a)$\\
         $\Leftrightarrow$&$(\mathfrak{A},\sigma_{a})\vDash\phi$\\
         $\Leftrightarrow$&$a\in A_{\phi}$
    \end{tabular}
    
    Hence $f^{-1}[S_{\phi}]=A_{\phi}$ and particularly, $f^{-1}[D^{T}_{ij}]=D^{\mathfrak{A}}_{ij}$.
    
    \item \textbf{structure homomorphism, C-mapping and basis-preserving:} the proof is the same as the corresponding part of the proof of Example 3, except that $S_{\phi}$ needs to be changed to $A_{\phi}$, and $S'_{\phi}$ to $S_{\phi}$.
\end{itemize}
\end{proof}
\section{Some Basic Works}\label{sbw}
As mentioned in Introduction, the topological representation of semantics of first-order logic based on the cylindric space faces two requirements: we need a point-to-point relation on the space to represent elementary embeddings between models and we need enough points to represent all models of a theory.

This section will devote to providing the basic topological tools to meet these requirements.
\subsection{Permutation on sets and points}
In this subsection, we introduce a kind of point-to-point, set-to-set relations on the basis-finite cylindric space, namely \textbf{permutation} and \textbf{factor}. Intuitively, this kind of relationship can be regarded as an abstract characterization of the endo-functions on the domain of first-order structure.

For convenience, we agree on the following notation:
\begin{definition}
For any cylindric space $\mathcal{C}_{\alpha}$, $u\subseteq|\mathcal{C}|$ and $i,j\in\alpha$, if $i\neq j$, we write $u(\frac{i}{j}):=[u\cap D^{\mathcal{C}}_{ij}]_{j}$, if $i=j$, we write $u(\frac{i}{j}):=u$.
\end{definition}

Now, let us do some technical preparations:
\begin{proposition}\label{pop}
Let $\mathcal{C}_{\alpha}$ be a cylindric space, $u,u'\in\tau^{\mathcal{C}}$ be clopen sets on it, then
\begin{enumerate}
    \item If $j\not\in\Delta(u)$, then $u(\frac{i}{j})=u$;
    \item If $i\neq j$, then $(-u)(\frac{i}{j})=-(u(\frac{i}{j}))$ and $(u\cap u')(\frac{i}{j})=u(\frac{i}{j})\cap u'(\frac{i}{j})$;
    \item  If $j_{1}\neq j_{2}$ and $\{i_{i},i_{2}\}\cap\{j_{1},j_{2}\}=\emptyset$, then $u(\frac{i_{1}}{j_{1}})(\frac{i_{2}}{j_{2}})=u(\frac{i_{2}}{j_{2}})(\frac{i_{1}}{j_{1}})$;
    \item If $k\not\in\Delta(u),i\not\in\{j,k\}$, then $u(\frac{k}{j})(\frac{i}{k})=u(\frac{i}{j})$;
    \item If $j\not\in\Delta(u),i\neq j$, then $(u\cap D_{ij})(\frac{i}{j})=u(\frac{j}{i})(\frac{i}{j})=u$;
    \item If no duplicate element in $\{j_{i}:1\ge i\ge n\}$, for any $r,s\in\alpha$, $k_{r}=k_{s}\Leftrightarrow k'_{r}=k'_{s}\Rightarrow i_{r}=i_{s}$ and $(\{k_{1},...k_{n}\}\cup\{k'_{1},...k'_{n}\})\cap(\{i_{1},...,i_{n}\}\cup\{j_{1},...,j_{n}\})=\emptyset$ and $\Delta(u)\subseteq\{j_{1},...,j_{n}\}$, then $u(\frac{k_{1}}{j_{1}})...(\frac{k_{n}}{j_{n}})(\frac{i_{1}}{k_{1}})...(\frac{i_{n}}{k_{n}})=u(\frac{k'_{1}}{j_{1}})...(\frac{k'_{n}}{j_{n}})(\frac{i_{1}}{k'_{1}})...(\frac{i_{n}}{k'_{n}})$.
\end{enumerate}
\end{proposition}
\begin{proof}$ $

\begin{enumerate}
    \item 
    For any $a\in|\mathcal{C}|$,
    
    \begin{tabular}{rl}
         &$a\in u(\frac{i}{j})$\\
         $\Leftrightarrow$&$\exists b\sim_{j}a$ s.t. $b\in u\cap D^{\mathcal{C}}_{ij}$\\
         $\Leftrightarrow$&$\exists b\sim_{j}a$ s.t. $b\in [u]_{j}\cap D^{\mathcal{C}}_{ij}$\\
         $\Leftrightarrow$&$a\in [u]_{j}$ and $\exists b\sim_{j}a$ s.t. $b\in [u]_{j}\cap D^{\mathcal{C}}_{ij}$\\
         $\Leftrightarrow$&$a\in u$ and $\exists b\sim_{j}a$ s.t. $b\in [u]_{j}\cap D^{\mathcal{C}}_{ij}$\\
         $\Leftrightarrow$&$a\in u$ ($[u]_{j}\cap D^{\mathcal{C}}_{ij}\supseteq[a]_{j}\cap D^{\mathcal{C}}_{ij}\neq\emptyset$)
    \end{tabular}
    \item 
    For any $a\in|\mathcal{C}|$,
    
    \begin{tabular}{rl}
         &$a\in(-u)(\frac{i}{j})$\\
         $\Leftrightarrow$&$[a]_{j}\cap(-u\cap D^{\mathcal{C}}_{ij})\neq\emptyset$\\
         $\Leftrightarrow$&$[a]_{j}\cap(u\cap D_{ij})=\emptyset\quad$(because $\|[a]_{j}\cap D^{\mathcal{C}}_{ij}\|=1$)\\
         $\Leftrightarrow$&$[a]_{j}\subseteq-(u\cap D^{\mathcal{C}}_{ij})$\\
         $\Leftrightarrow$&$a\in -(u(\frac{i}{j}))$
    \end{tabular}
    
    \begin{tabular}{rl}
         &$a\in u(\frac{j}{i})\cap u'(\frac{j}{i})$\\
         $\Leftrightarrow$&$\exists b,c\sim_{i}a$ s.t. $b\in u\cap D^{\mathcal{C}}_{ij}$ and $c\in u'\cap D^{\mathcal{C}}_{ij}$\\
         $\Leftrightarrow$&$\exists b,c\sim_{i}a$ s.t. $b\in u\cap D^{\mathcal{C}}_{ij}$ and $c\in u'\cap D^{\mathcal{C}}_{ij}$ and $c=b$ (because $\|[b]_{i}\cap D^{\mathcal{C}}_{ij}\|=1$)\\
         $\Leftrightarrow$&$\exists b\sim_{i}a$ s.t. $b\in u\cap u'\cap D^{\mathcal{C}}_{ij}$\\
         $\Leftrightarrow$&$a\in(u\cap u')(\frac{j}{i})$
    \end{tabular}
    \item 
    For any $a\in|\mathcal{C}|$,
    
    \begin{tabular}{rl}
         &$a\in u(\frac{i_{1}}{j_{1}})(\frac{i_{2}}{j_{2}})$\\
         $\Leftrightarrow$&$\exists b,c$ s.t. $c\sim_{j_{1}}b\sim_{j_{2}}a$ and $c\in u\cap D^{\mathcal{C}}_{i_{1}j_{1}},b\in D^{\mathcal{C}}_{i_{2}j_{2}}$\\
         $\Leftrightarrow$&$\exists b,c$ s.t. $c\sim_{j_{1}}b\sim_{j_{2}}a$ and $c\in u\cap D^{\mathcal{C}}_{i_{1}j_{1}},b,c\in D^{\mathcal{C}}_{i_{2}j_{2}}$ (because $D^{\mathcal{C}}_{i_{2}j_{2}}$ is $\sim_{j_{1}}$-saturated)\\
         $\Leftrightarrow$&$\exists b,c$ s.t. $c\sim_{j_{1}}b\sim_{j_{2}}a$ and $c\in u\cap D^{\mathcal{C}}_{i_{1}j_{1}}\cap D^{\mathcal{C}}_{i_{2}j_{2}}\quad$(similarly)\\
         $\Leftrightarrow$&$a\in[[u\cap D^{\mathcal{C}}_{i_{1}j_{1}}\cap D^{\mathcal{C}}_{i_{2}j_{2}}]_{j_{1}}]_{j_{2}}$\\
         $\Leftrightarrow$&$a\in[[u\cap D^{\mathcal{C}}_{i_{1}j_{1}}\cap D^{\mathcal{C}}_{i_{2}j_{2}}]_{j_{2}}]_{j_{1}}$\\
         $\Leftrightarrow$&$a\in u(\frac{i_{2}}{j_{2}})(\frac{i_{1}}{j_{1}})$
    \end{tabular}
    \item
    It is obviously true when $j=k$. Consider the case when $j\neq k$: for any $a\in|\mathcal{C}|$,
    
    \begin{tabular}{rl}
         &$a\in u(\frac{k}{j})(\frac{i}{k})$\\
         $\Leftrightarrow$&$\exists b,c$ s.t. $c\sim_{j}b\sim_{k}a$ and $c\in u\cap D^{\mathcal{C}}_{kj},b\in D^{\mathcal{C}}_{ik}$\\
         $\Leftrightarrow$&$\exists b,c$ s.t. $c\sim_{j}b\sim_{k}a$ and $c\in u\cap D^{\mathcal{C}}_{kj},b,c\in D^{\mathcal{C}}_{ik}$ (because $D^{\mathcal{C}}_{ik}$ is $\sim_{j}$-saturated)\\
         $\Leftrightarrow$&$\exists b,c$ s.t. $c\sim_{j}b\sim_{k}a$ and $c\in u\cap D^{\mathcal{C}}_{kj}\cap D^{\mathcal{C}}_{ik}\quad$(similarly)\\
         $\Leftrightarrow$&$a\in[[u\cap D^{\mathcal{C}}_{kj}\cap D^{\mathcal{C}}_{ik}]_{j}]_{k}$\\
         $\Leftrightarrow$&$a\in[[u\cap D^{\mathcal{C}}_{kj}\cap D^{\mathcal{C}}_{ik}]_{k}]_{j}$\\
         $\Leftrightarrow$&$\exists b,c$ s.t. $c\sim_{k}b\sim_{j}a$ and $c\in u\cap D^{\mathcal{C}}_{kj}\cap D^{\mathcal{C}}_{ik}$\\
         $\Leftrightarrow$&$\exists b,c$ s.t. $c\sim_{k}b\sim_{j}a$ and $c\in u\cap D^{\mathcal{C}}_{kj}\cap D^{\mathcal{C}}_{ik}$ and $b\in u$ (because $u$ is $\sim_{k}$-saturated)\\
         $\Leftrightarrow$&$\exists b,c$ s.t. $c\sim_{k}b\sim_{j}a$ and $c\in D^{\mathcal{C}}_{kj}\cap D^{\mathcal{C}}_{ik}$ and $b\in u\quad$(similarly)\\
         $\Leftrightarrow$&$\exists b\sim_{j}a$ s.t. $b\in D^{\mathcal{C}}_{ij}=[D^{\mathcal{C}}_{kj}\cap D^{\mathcal{C}}_{ik}]_{k}$ and $b\in u$\\
         $\Leftrightarrow$&$a\in u(\frac{i}{j})$
    \end{tabular}
    \item
    For any $a\in|\mathcal{C}|$,
    
    \begin{tabular}{rl}
         &$a\in u(\frac{j}{i})(\frac{i}{j})$\\
         $\Leftrightarrow$&$a\in [[u\cap D^{\mathcal{C}}_{ij}]_{i}\cap D^{\mathcal{C}}_{ij}]_{j}$\\
         $\Leftrightarrow$&$\exists b\sim_{j}a$ s.t. $b\in D^{\mathcal{C}}_{ij}$ and $\exists c\sim_{i}b$ s.t. $c\in u\cap D^{\mathcal{C}}_{ij}$\\
         $\Leftrightarrow$&$\exists b\sim_{j}a$ s.t. $b\in D^{\mathcal{C}}_{ij}$ and $\exists c\sim_{i}b$ s.t. $c\in u\cap D^{\mathcal{C}}_{ij}$ and $c=b$ (because $\|[b]_{i}\cap D^{\mathcal{C}}_{ij}\|=1$)\\
         $\Leftrightarrow$&$a\in [u\cap D^{\mathcal{C}}_{ij}]_{j}=[(u\cap D^{\mathcal{C}}_{ij})\cap D^{\mathcal{C}}_{ij}]_{j}=(u\cap D^{\mathcal{C}}_{ij})(\frac{i}{j})$\\
         $\Leftrightarrow$&$[a]_{j}\cap (u\cap D^{\mathcal{C}}_{ij})\neq\emptyset$\\
         $\Leftrightarrow$&$[a]_{j}\subseteq u$ and $[a]_{j}\cap D^{\mathcal{C}}_{ij}\neq\emptyset\quad$(because $u$ is $\sim_{j}$-saturated)\\
         $\Leftrightarrow$&$a\in u$
    \end{tabular}
    \item WLOG, let $\{k_{i}:1\ge i\ge n\}\cap\{k'_{i}:1\ge i\ge n\}=\emptyset$, we have: (the following proof will frequently use the antecedents of the state but will not be indicated in the corresponding place, please note when reading)
    
    \begin{tabular}{rl}
         &$u(\frac{k_{1}}{j_{1}})...(\frac{k_{n}}{j_{n}})(\frac{i_{1}}{k_{1}})...(\frac{i_{n}}{k_{n}})$\\
         $=$&$u(\frac{k_{1}}{j_{1}})...(\frac{k_{n}}{j_{n}})(\frac{k'_{1}}{k_{1}})(\frac{k_{1}}{k'_{1}})(\frac{i_{1}}{k_{1}})...(\frac{i_{n}}{k_{n}})$(by Proposition \ref{pop}.(5))\\
         &\dots\dots\\
         $=$&$u(\frac{k_{1}}{j_{1}})...(\frac{k_{n}}{j_{n}})(\frac{k'_{1}}{k_{1}})...(\frac{k'_{s}}{k_{s}})(\frac{k_{s}}{k'_{s}})...(\frac{k_{1}}{k'_{1}})(\frac{i_{1}}{k_{1}})...(\frac{i_{n}}{k_{n}})$ \\&(if there is $r<s,k_{r}=k_{s}$, by Proposition \ref{pop}.(1), $(\frac{k'_{s}}{k_{s}})$ can be added \\&after $(\frac{k'_{r}}{k_{r}})$, and $(\frac{k_{r}}{k'_{r}})$ can be recorded as $(\frac{k_{s}}{k'_{s}})$ and $(\frac{k_{r}}{k'_{r}})$ can be added \\&after it, and then $(\frac{k'_{s}}{k_{s}})$ and $(\frac{k_{s}}{k'_{s}})$ are moved to the desired position by \\&the Proposition \ref{pop}.(3); if there is no $r<s$, using the Proposition \ref{pop}.(5))\\
         &$\dots\dots$\\
         $=$&$u(\frac{k_{1}}{j_{1}})...(\frac{k_{n}}{j_{n}})(\frac{k'_{1}}{k_{1}})...(\frac{k'_{n}}{k_{n}})(\frac{k_{n}}{k'_{n}})...(\frac{k_{1}}{k'_{1}})(\frac{i_{1}}{k_{1}})...(\frac{i_{n}}{k_{n}})$ (A)\\
         $=$&$u(\frac{k_{1}}{j_{1}})...(\frac{k_{n}}{j_{n}})(\frac{k'_{s_{11}}}{k_{s_{11}}})(\frac{k'_{s_{12}}}{k_{s_{12}}})...(\frac{k'_{s_{1t_{1}}}}{k_{s_{1t_{1}}}})...(\frac{k'_{ht_{h}}}{k_{ht_{h}}})(\frac{k_{n}}{k'_{n}})...(\frac{k_{2}}{k'_{2}})(\frac{i_{1}}{k'_{1}})(\frac{i_{2}}{k_{2}})...(\frac{i_{n}}{k_{n}})$ \\&(here $(\frac{k'_{s_{11}}}{k_{s_{11}}}),...,(\frac{k'_{ht_{h}}}{k_{ht_{h}}})$ lists all the $(\frac{k'_{s}}{k_{s}})$ in the previous equation, and \\&for each $l\le h$, $(\frac{k'_{l2}}{k_{l2}}),...,(\frac{k'_{lt_{l}}}{k_{lt_{l}}})$ lists all the $(\frac{k'_{s}}{k_{s}})$ with the same shape as \\&$(\frac{k'_{l1}}{k_{l1}})$, by Proposition\ref{pop}.(3), two equations are equal)\\
         $=$&$u(\frac{k_{1}}{j_{1}})...(\frac{k_{n-1}}{j_{n-1}})(\frac{k'_{s_{11}}}{k_{s_{11}}})...(\frac{k'_{s_{1t_{1}}}}{k_{s_{1t'_{1}}}})(\frac{k^{*}_{n}}{j_{n}})(\frac{k'_{s_{21}}}{k_{s_{21}}})...$ \\&(if $k_{n}\neq k_{s_{11}}$, then repeated use of Proposition \ref{pop}.(3) yields this equ-\\&-ation, and we know $t_{1}=t'_{1},k^{*}_{n}=k_{n}$; if $k_{n}=k_{s_{11}}$, then by Proposition \\&\ref{pop}.(4), $(\frac{k_{n}}{j_{n}})(\frac{k'_{s_{11}}}{k_{s_{11}}})$ can be synthesized into $(\frac{k'_{n}}{j_{n}})$
        then by Proposition \\&\ref{pop}.(3) $(\frac{k'_{s_{12}}}{k_{s_{12}}})...(\frac{k'_{s_{1t_{1}}}}{k_{s_{1t'_{1}}}})$ can be moved before $(\frac{k'_{n}}{j_{n}})$ and therefore this eq-\\&-uation can be obtained, and we know $t'_{1}=t_{1}-1$, $k^{*}_{n}=k'_{n}$)$\quad$(B)\\
         &$\dots\dots$\\
         $=$&$u(\frac{k'_{1}}{j_{1}})...(\frac{k'_{n}}{j_{n}})(\frac{k_{n}}{k'_{n}})...(\frac{k_{2}}{k'_{2}})(\frac{i_{2}}{k_{2}})...(\frac{i_{n}}{k_{n}})\quad$(C) \\&(In the same way as steps A-B, repeated several times to obtain)\\
         &$\dots\dots$\\
         $=$&$u(\frac{k'_{1}}{j_{1}})...(\frac{k'_{n}}{j_{n}})(\frac{i_{1}}{k'_{1}})...(\frac{i_{n}}{k'_{n}})\quad$(In the same way as steps A-C)
    \end{tabular}
\end{enumerate}
\end{proof}

The following definition is well-defined by Proposition \ref{pop}.

\begin{definition}\label{pmn}
Let $\mathcal{C}_{\alpha}$ be a basis-finite cylindric space, $\rho:\alpha^{\alpha}$ be a (partial) mapping,
\begin{itemize}
    \item Let $B$ be a cylindric basis of $\mathcal{C}$, for any $u\in B$ with $\Delta(u)\subseteq dom(\rho)$, we write $\rho u:=u(\frac{k_{1}}{j_{1}})...(\frac{k_{n}}{j_{n}})(\frac{\rho(j_{1})}{k_{1}})...(\frac{\rho(j_{n})}{k_{n}})$ s.t. no duplicate elements in both $\{k_{i}:1\ge i\ge n\}$ and $\{j_{i}:1\ge i\ge n\}$, $\Delta(u)\subseteq\{j_{1},...,j_{n}\}$ and $\{k_{1},...k_{n}\}\cap(\{j_{1},...,j_{n}\}\cup\{\rho(j_{1},...,\rho(j_{n}))\})=\emptyset$. We call $\rho u$ a \textbf{permutation} of $u$; 
    \item For a closed set $u$ with $\Delta(u)\subseteq dom(\rho)$, we write $\rho u:=\bigcap\{\rho v:v\in B\mbox{ and }u\subseteq v\}$, call $\rho u$ a \textbf{permutation} of $u$; the permutation of open sets also can be defined in the same way.
\end{itemize}
\end{definition}

Now we define a cylindric space that is very important in the following text:
\begin{definition}\label{dofol}
We call a Stone basis-finite cylindric space an \textbf{FOL space}.
\end{definition}

For FOL space, the following proposition holds:
\begin{theorem}\label{up}
Let $\mathcal{C}_{\alpha}$ be an FOL space, $\rho:\alpha^{\alpha}$ be a mapping, $a\in|\mathcal{C}_{\alpha}|$, we have
\begin{enumerate}
    \item If for any $k\in ran(\rho)$ and $i,j\in\rho^{-1}[k]$, we have $a\in D_{ij}$, then $\rho\{a\}\neq\emptyset$;
    \item If $\rho$ is a surjection, then $\|\rho\{a\}\|\le 1$;
    \item $\|\{a':a\in\rho\{a'\}\}\|=1$;
    \item For closed sets $u,v,w$, if $\rho u=v,\rho' v=w$, then $(\rho'\circ\rho) u=w$.
\end{enumerate}
\end{theorem}
\begin{proof}
Let $B$ be a cylindric basis of $\mathcal{C}$,
\begin{enumerate}
    \item 
    We know $\rho\{a\}=\bigcap\{\rho u:u$ is a clopen neighborhood of $a\}$.
    
    For any clopen neighborhood $u_{1},...,u_{m}\ni a$, we know there are $k_{1},...,k_{n},j_{i},...,j_{n}$, $\rho u_{i}=u_{i}(\frac{k_{1}}{j_{1}})...(\frac{k_{n}}{j_{n}})$
    $(\frac{\rho(j_{1})}{k_{1}})...(\frac{\rho(j_{n})}{k_{n}})$ satisfying the conditions in Definition \ref{pmn}. Let $(j_{l_{1}},j_{l'_{1}}),...,(j_{l_{s}},j_{l'_{s}})$ be all pairs in $\{j_{1},...,j_{n}\}$ satisfying $\rho(j_{l_{r}})=\rho(j_{l_{r'}})$, we know $\bigcap_{r=1}^{s} D^{\mathcal{C}}_{j_{l_{r}},j_{l'_{r}}}\cap\bigcap_{i=1}^{n}u_{i}\neq\emptyset$ because $a$ is in it by the assumption.
    
    \begin{tabular}{rl}
         &$\bigcap_{r=1}^{s} \rho D^{\mathcal{C}}_{j_{l_{r}},j_{l'_{r}}}\cap\bigcap_{i=1}^{n}\rho u_{i}$\\
         $=$&$\bigcap_{r=1}^{s} (D^{\mathcal{C}}_{j_{l_{r}},j_{l'_{r}}}(\frac{k_{1}}{j_{1}})...(\frac{k_{n}}{j_{n}})(\frac{\rho(j_{1})}{k_{1}})...(\frac{\rho(j_{n})}{k_{n}}))\cap\bigcap_{i=1}^{n}(u_{i}(\frac{k_{1}}{j_{1}})...(\frac{k_{n}}{j_{n}})(\frac{\rho(j_{1})}{k_{1}})...(\frac{\rho(j_{n})}{k_{n}}))$\\
         $=$&$(\bigcap_{r=1}^{s} D^{\mathcal{C}}_{k_{l_{r}},k_{l'_{r}}}\cap(\bigcap_{i=1}^{n}u_{i})(\frac{k_{1}}{j_{1}})...(\frac{k_{n}}{j_{n}}))(\frac{\rho(j_{1})}{k_{1}})...(\frac{\rho(j_{n})}{k_{n}})$ (use Proposition \ref{pop}.(2) repeatedly)
    \end{tabular}
    
    Since $\bigcap_{r=1}^{s} D^{\mathcal{C}}_{j_{l_{r}},j_{l'_{r}}}$
    $\cap$
    $\bigcap_{i=1}^{n}u_{i}\neq\emptyset$, we know $\bigcap_{r=1}^{s} D^{\mathcal{C}}_{k_{l_{r}},k_{l'_{r}}}$
    $\cap(\bigcap_{i=1}^{n}u_{i})(\frac{k_{1}}{j_{1}})$
    $...$
    $(\frac{k_{n}}{j_{n}})\neq\emptyset$. For each $i\in\{0,...,n\}$, let $X_{i}:=\bigcap_{r=1}^{s} D^{\mathcal{C}}_{k_{l_{r}},k_{l'_{r}}}$
    $\cap$
    $(\bigcap_{i=1}^{n}u_{i})(\frac{k_{1}}{j_{1}})$
    $...$
    $(\frac{k_{n}}{j_{n}})(\frac{\rho(j_{1})}{k_{1}})$
    $...$
    $(\frac{\rho(j_{i})}{k_{i}})$. We use induction on $i\in\{0,...,n\}$ as follow to show $X_{i}\neq\emptyset$:
    
    \begin{itemize}
        \item \textbf{IH:} For $i'<i$, $X_{i'}=\bigcap\{D^{\mathcal{C}}_{\rho(j_{l_{r}}),\rho(j_{l'_{r}})}:l_{r},l'_{r}\le i'\}\cap\bigcap\{D^{\mathcal{C}}_{\rho(j_{l_{r}}),k_{l'_{r}}}:l_{r}\le i,l'_{r}>i'\}\cap\bigcap\{D^{\mathcal{C}}_{k_{l_{r}},k_{l'_{r}}}:l_{r},l'_{r}>i'\}\cap(\bigcap_{t=1}^{n}u_{t})(\frac{k_{1}}{j_{1}})...(\frac{k_{n}}{j_{n}})(\frac{\rho(j_{1})}{k_{1}})...(\frac{\rho(j_{i'})}{k_{i'}})\neq\emptyset$.
        \item \textbf{IS:}$ $
        
        \begin{itemize}
        \item If $\rho(j_{i})$ does not in the dimension of $X_{i-1}$, then $X_{i}\neq\emptyset$. By Proposition \ref{pop}.(2), $X_{i}=\bigcap\{D^{\mathcal{C}}_{\rho(j_{l_{r}}),\rho(j_{l'_{r}})}:l_{r},l'_{r}\le i\}\cap\bigcap\{D^{\mathcal{C}}_{\rho(j_{l_{r}}),k_{l'_{r}}}:l_{r}\le i,l'_{r}>i\}\cap\bigcap\{D^{\mathcal{C}}_{k_{l_{r}},k_{l'_{r}}}:l_{r},l'_{r}>i\}\cap(\bigcap_{t=1}^{n}u_{t})(\frac{k_{1}}{j_{1}})...(\frac{k_{n}}{j_{n}})(\frac{\rho(j_{1})}{k_{1}})...(\frac{\rho(j_{i})}{k_{i}})$;
        \item If $\rho(j_{i})$ does in the dimension of $X_{i-1}$, then there is $i'<i$ s.t. $\rho(j_{i})=\rho(j_{i'})$. By IH, we know $X_{i-1}\subseteq D^{\mathcal{C}}_{\rho(j_{i'}),k_{i}}$, then $X_{i}=[X_{i-1}\cap D^{\mathcal{C}}_{\rho(j_{i}),k_{i}}]_{k_{i}}\neq\emptyset$ and by Proposition \ref{pop}.(2), $X_{i}=\bigcap\{D^{\mathcal{C}}_{\rho(j_{l_{r}}),\rho(j_{l'_{r}})}:l_{r},l'_{r}\le i\}\cap\bigcap\{D^{\mathcal{C}}_{\rho(j_{l_{r}}),k_{l'_{r}}}:l_{r}\le i,l'_{r}>i\}\cap\bigcap\{D^{\mathcal{C}}_{k_{l_{r}},k_{l'_{r}}}:l_{r},l'_{r}>i\}\cap(\bigcap_{t=1}^{n}u_{t})(\frac{k_{1}}{j_{1}})...(\frac{k_{n}}{j_{n}})(\frac{\rho(j_{1})}{k_{1}})...(\frac{\rho(j_{i})}{k_{i}})$.
    \end{itemize}
    \end{itemize}
    
    Hence, $\bigcap_{r=1}^{s} \rho D^{\mathcal{C}}_{j_{l_{r}},j_{l'_{r}}}\cap\bigcap_{i=1}^{n}\rho u_{i}\neq\emptyset$, then we know $\{\rho u:u$ is a clopen neighborhood of $a\}$ has finite intersection property, then by compactness, $\rho\{a\}=\bigcap\{\rho u:u$ is clopen neighborhood of $a\}\neq\emptyset$.
    
    \item Assume there are $b,b'\in\rho\{a\}$ s.t. $b\neq b'$, we know there is a $u\in B$ s.t. $u\ni b,-u\ni b'$. Let $\Delta(u)=\Delta(-u)=\{x_{1},...,x_{n}\}$. Since $\rho$ is a surjection, we know there are $y_{1},..,y_{n}\in\alpha$ s.t. $\rho(y_{i})=x_{i}$. Choose arbitrary $k_{1},...,k_{n}\in\alpha\backslash\{x_{1},...,x_{n},y_{1},...,y_{n}\}$ and let $v=u(\frac{k_{1}}{x_{1}})...(\frac{k_{n}}{x_{n}})(\frac{y_{1}}{k_{1}})...(\frac{y_{n}}{k_{n}})$. By Proposition \ref{pop}, we know $\rho v=u(\frac{k_{1}}{x_{1}})...(\frac{k_{n}}{x_{n}})(\frac{y_{1}}{k_{1}})...(\frac{y_{n}}{k_{n}})(\frac{k_{1}}{y_{1}})...(\frac{k_{n}}{y_{n}})(\frac{x_{1}}{k_{1}})...(\frac{x_{n}}{k_{n}})=u,\rho -v=-u(\frac{k_{1}}{x_{1}})...(\frac{k_{n}}{x_{n}})(\frac{y_{1}}{k_{1}})...(\frac{y_{n}}{k_{n}})$
    $(\frac{k_{1}}{y_{1}})...(\frac{k_{n}}{y_{n}})$
    $(\frac{x_{1}}{k_{1}})...(\frac{x_{n}}{k_{n}})=-u$. Clearly $a\in v$ or $a\in -v$, then $\rho\{a\}\subseteq u$ or $\rho\{a\}\subseteq -u$, which contradicts the requirement for $u$.
   
    \item Let $U=\{u\in B:a\in\rho u\}$, we know $\{a':a\in\rho\{a'\}\}=\bigcap U$. Choose an arbitrary finite $V\subseteq U$, by Proposition \ref{pop}.(2), $\rho(\bigcap V)=\bigcap\{\rho v:v\in V\}\ni a$, then $\bigcap V\neq\emptyset$. By compactness, $\bigcap U\neq\emptyset$.
    
    For any $u\in B$, if $u\not\in U$, then $a\not\in\rho u$, then by Proposition \ref{pop}.(2), $a\in-\rho u=\rho-u$, then $-u\in U$, then by $T_{2}$ property, $\|\bigcap U\|=1$.
    \item$ $
    
    \begin{tabular}{rl}
         &$(\rho'\circ\rho) u$\\
         $=$&$(\frac{k_{1}}{j_{1}})...(\frac{k_{n}}{j_{n}})(\frac{\rho'\circ\rho(j_{1}))}{k_{1}})...(\frac{\rho'\circ\rho(j_{n}))}{k_{n}})$\\
         $=$&$(\frac{k_{1}}{j_{1}})...(\frac{k_{n}}{j_{n}})(\frac{k'_{n}}{k_{n}})...(\frac{k'_{1}}{k_{1}})(\frac{\rho'\circ\rho(j_{1}))}{k'_{1}})...(\frac{\rho'\circ\rho(j_{n}))}{k'_{n}})$\\&(In the same way as steps A-C of the proof of the Proposition \ref{pop}.(6))\\
         $=$&$(\frac{k_{1}}{j_{1}})...(\frac{k_{n}}{j_{n}})(\frac{k'_{1}}{k_{1}})...(\frac{k'_{n}}{k_{n}})(\frac{\rho'\circ\rho(j_{1}))}{k'_{1}})...(\frac{\rho'\circ\rho(j_{n}))}{k'_{n}})$ (Repeated use of the \\&Proposition \ref{pop}.(3))\\
         $=$&$(\frac{k_{1}}{j_{1}})...(\frac{k_{n}}{j_{n}})(\frac{\rho(j_{1})}{k_{1}})...(\frac{\rho(j_{n})}{k_{n}})(\frac{k'_{n}}{\rho(j_{n})})...(\frac{k'_{1}}{\rho(j_{1})})(\frac{\rho'(\rho(j_{1}))}{k'_{1}})...(\frac{\rho'(\rho(j_{n}))}{k'_{n}})$\\&(In the same way as steps A-C of the proof of the Proposition \ref{pop}.(6))\\
         $=$&$(\frac{k_{1}}{j_{1}})...(\frac{k_{n}}{j_{n}})(\frac{\rho(j_{1})}{k_{1}})...(\frac{\rho(j_{n})}{k_{n}})(\frac{k'_{1}}{\rho(j_{1})})...(\frac{k'_{n}}{\rho(j_{n})})(\frac{\rho'(\rho(j_{1}))}{k'_{1}})...(\frac{\rho'(\rho(j_{n}))}{k'_{n}})$\\&(Repeated use of the Proposition \ref{pop}.(3))\\
         $=$&$\rho'(\rho u)\quad$\\&(By Proposition \ref{pop}.(1), we can add or delete operations $(\frac{k}{i}),(\frac{\rho'(i)}{k})$ \\&to a format consistent with the definition)\\
         $=$&$w$
    \end{tabular}
\end{enumerate}
\end{proof}

\begin{definition}
let $\mathcal{C}_{\alpha}$ be a basis-finite cylindric space, $s\subseteq\alpha$, $u$ be a closed set in $\mathcal{C}_{\alpha}$ with dimension $s$ s.t. for any closed set $v$, $\Delta(v)\subseteq s\Rightarrow u\subseteq v$ or $u\subseteq-v$. We call $u$ a \textbf{complete closed set of dimension $s$}. For $a\in|\mathcal{C}|$, write the complete closed set of dimension $s$ containing $a$ as $a|_{s}$.
\end{definition}

We can now define the permutation relation between points based on the above theorem.

\begin{definition}
Let $\mathcal{C}_{\alpha}$ be an FOL space.
\begin{itemize}
    \item For $a,b\in|\mathcal{C}_{\alpha}|$ and (partial) mapping $\rho:\alpha^{\alpha}$, if $b\in\rho\{a|_{dom(\rho)}\}$, then we say that $a$ is a \textbf{(partial) factor} of $b$, and denote as $\rho^{-1}b=a$, $a\prec_{\rho} b$ or $a\prec b$;
    \item For $a,b\in|\mathcal{C}_{\alpha}|$ and $\rho:\alpha^{\alpha}$, if $\rho$ is a surjection and $\rho^{-1}b=a$ (knowing that at this time, $\rho\{a\}=\{b\}$), we say that $b$ is a \textbf{permutation} of $a$, and denote as $\rho a=b$.
\end{itemize}
\end{definition}

\begin{example}\label{examp}
Let $\mathfrak{A},\mathfrak{B}$ be two countable first-order structures s.t. there is a (partial) elementary embedding $f:\mathfrak{A}\rightarrow\mathfrak{B}$. Then for $a\in\mathcal{C}^{\mathfrak{A}}_{\omega}$, $b\in\mathcal{C}^{\mathfrak{B}}_{\omega}$ being domain points and $a',b'\in\mathcal{C}_{T}$ realized by $a,b$ respectively, $a'$ is a (partial) factor of $b'$. If $f$ is a isomorphism and sequence $b$ has no duplicates, then $b'$ is a permutation of $a'$.
\end{example}

\begin{definition}
Let $\mathcal{C}_{\alpha}$ be an FOL space, $a,b\in|\mathcal{C}|$. If there is a $\rho:a\prec b$ s.t. for any $j\in\alpha\backslash ran(\rho)$, there is an $i\in ran(\rho)$ s.t. $b\in D_{ij}$, then we say $b$ is an equivalent point of $a$ and denote this as $\rho:a\asymp b$.
\end{definition}

\begin{proposition}
Let $\mathcal{C}_{\alpha}$ be an FOL space, the relation $\asymp$ is an equivalence relation on $|\mathcal{C}|$:

\begin{itemize}
    \item $a\asymp b\Leftrightarrow$ there is a $\rho:a\prec b$ s.t. for any $j\in\alpha\backslash ran(\rho)$, there is an $i\in ran(\rho)$ s.t. $b\in D_{ij}$ (denote this as $\rho:a\asymp b$).
\end{itemize}
\end{proposition}
\begin{proof}$ $

\begin{itemize}
    \item\textbf{Reflexivity: }Let $\rho$ be the identity, clearly $\rho:a\asymp a$.
    \item\textbf{Symmetry: }Assume $\rho:a\asymp b$, we define $\rho':\alpha^{\alpha}$ as $i\mapsto min\{i':b\in D_{\rho(i')i}\}$, it's easy to see $\rho'^{-1}a=b$.
    
     For any $j\in\alpha\backslash ran(\rho')$, consider $j'=min\{k:\rho(k)=\rho(j)\}$, we know $j'=\rho'(\rho(j))\in ran(\rho')$ and $a\in D_{jj'}$. Hence $\rho':b\asymp a$.
     
     \item\textbf{Transitivity: }Assume $\rho:a\asymp b$,, $\rho':b\asymp c$, then by Theorem \ref{up}.(4), we know $\rho'\circ\rho\{a\}\supseteq\{c\}$ i.e. $\rho'\circ\rho: a\prec c$. For any $i\in\alpha\backslash ran(\rho'\circ\rho)$, we know there is an $i'\in ran(\rho')$ s.t. $c\in D_{ii'}$. If $i'\not\in ran(\rho'\circ\rho)$, then $\rho'^{-1}[i']\subseteq\alpha\backslash ran(\rho)$, then for any $i^{+}\in\rho'^{-1}[i']$ there is an $i^{*}\in ran(\rho)$ s.t. $b\in D_{i^{+}i^{*}}$, then $c\in D_{i'\rho'(i^{*})}$, then $c\in D_{i\rho'(i^{*})}$, then clearly we have $\rho'(i^{*})\in ran(\rho'\circ\rho)$. Hence $\rho'\circ\rho: a\asymp c$.
\end{itemize}
\end{proof}

\begin{example}
Let $a,b,a,'b'$ be as in Example \ref{examp}, $f$ be a isomorphism, then $a'\asymp b'$ regardless of whether $b'$ has duplicates.
\end{example}
\subsection{$\alpha$-expansion}\label{alpex}
In this subsection, we will study how to ``expand" a basis-finite cylindric space into an FOL space of a specific dimension and further construct a one-to-one correspondence between FOL spaces of different dimensions. First, for a basis-finite cylindric space $\mathcal{C}_{\beta}$ and an ordinal $\alpha\ge\beta$, consider the following construction:

Let $B$ be the family of all clopen sets in $\mathcal{C}_{\alpha}$, $M$ be the collection of all mappings from $\beta$ to $\alpha$. We call $x\subseteq M\times B$ an atom if it satisfies the following conditions:
\begin{enumerate}
    \item For any $\rho\in M,u\in B$, $(\rho,u)\in x\Leftrightarrow(\rho,-u)\not\in x$;
    \item For any $\rho\in M,u,v\in B$, if $u\subseteq v$ and $(\rho,u)\in x$, then $(\rho,v)\in x$;
    \item For any $\rho\in M,u,v\in B$, if $u\cap v\neq\emptyset$ and $(\rho,u),(\rho,v)\in x$, then $(\rho,u\cap v)\in x$;
    \item For any $\rho,\mu,\sigma\in M,u,v\in B$, if $\mu|_{\Delta(u)}=\sigma|_{\Delta(u)}$, $\rho[\Delta(u)]\subseteq\beta$ and $\rho v=u$, then $(\mu\circ \rho,v)\in x\Leftrightarrow (\sigma,u)\in x$;
    \footnote{For any mapping $\rho:\alpha^{\beta}$, any $u\in B$ satisfying $\rho[\Delta(u)]\subseteq\beta$, clearly there is $\rho':\beta^{\beta}$ s.t. $\rho'|_{\Delta(u)}=\rho|_{\Delta(u)}$, we write $\rho u:=\rho' u$.}
\end{enumerate}

Before starting the formal construction, we still need to verify one property.
\begin{proposition}\label{ate}
Define $X_{\mathcal{C}}:=\{a\subseteq M\times B:a$ is atom$\}$. For any $(\rho,u)\in M\times B$, there is an atom $x\in X_{\mathcal{C}}$ s.t. $(\rho,u)\in x$ and therefore $X_{\mathcal{C}}\neq\emptyset$.
\end{proposition}
\begin{proof}
Well-order $M\times B$ with $(\rho,u)$ as the least element and let its isomorphic ordinal be $\gamma$. For each $i\in\gamma$, write the $i$-th element as $(\rho_{i},u_{i})$. We use transfinite induction to define the set $x_{i}$ for each $i\in\gamma$ and verify that they satisfy the following conditions:

\begin{enumerate}
    \item For any $\mu\in M$, any finite $U\subseteq B$, if $\{(\mu,v):v\in U\}\subseteq x_{i}$, then for any $w\subseteq-\bigcap U$, $(\mu,w)\not\in x_{i}$;
    \item It satisfies the fourth clause of the definition of the atom.
\end{enumerate}
The following proof will refer to these two conditions as inductive conditions.

\textbf{BS:} $x_{0}:=\{(\rho\circ\mu,v):\mu[\Delta(v)]=\Delta(u)\mbox{ and }\mu v=u\}$. clearly, $x_{0}$ satisfies the inductive conditions;

\textbf{IH:} For $j<i$, $x_{j}$ has been constructed and satisfies the induction conditions, and for $j'<j<i$, $x_{j'}\subseteq x_{j}$; here we prove a corollary of IH, namely $\bigcup_{j<i}x_{j}$ satisfies the induction conditions:

It is easy to see that $\bigcup_{j<i}x_{j}$ satisfies the fourth clause of the definition of the atom. If there is $\mu\in M$ and finite $U^{+}\subseteq B$ s.t. $\{(\mu,v):v\in U^{+}\}\subseteq\bigcup_{j<i}x_{j}$ constitutes a counterexample to the first clause of the inductive conditions, then since $U^{+}$ is finite, we know there is $j<i$ s.t. $\{(\mu,v):v\in U^{+}\}\subseteq x_{j}$, then $x_{j}$ does not satisfy the inductive conditions, which contradicts IH.

\textbf{IS:} If $(\rho_{i},u_{i})\in \bigcup_{j<i}x_{j}$, then let $x_{i}:=\bigcup_{j<i}x_{j}$. We know $x_{i}$ satisfies the inductive conditions by IH.

If $(\rho_{i},u_{i})\not\in \bigcup_{j<i}x_{j}$,
\begin{itemize}
    \item If there is a finite $U\subseteq B$ s.t. $\{(\rho_{i},v):v\in U\}\subseteq \bigcup_{j<i}x_{j}$ and $u_{i}\subseteq-\bigcap U$, then let $x_{i}:=\bigcup_{j<i}x_{j}\cup\{(\rho_{i}\circ\mu,v):\mu[\Delta(v)]=\Delta(u_{i})\mbox{ and }\mu v=-u_{i}\}$, It's easy to see that it satisfies the fourth clause of definition of atom. Assume there is a $\mu\in M$, finite $U^{+}\subseteq B$ s.t. $\{(\mu,v):v\in U^{+}\}\subseteq x_{i}$ constituting a counterexample to the first clause of the inductive conditions. By the corollary of IH, $\{(\mu,v):v\in U^{+}\}\subseteq x_{i}\not\subseteq\bigcup_{j<i}x_{j}$. By the fourth clause of definition of atom, there is a finite $V^{+}\subseteq B$ s.t. $\{(\rho_{i},v):v\in V^{+}\}\subseteq x_{i}$ constitutes a counterexample to the first clause of the inductive conditions, and $(\rho_{i},-u_{i})\in\{(\rho_{i},v):v\in V^{+}\}\subseteq\{(\rho_{i},-u_{i})\}\cup\bigcup_{j<i}x_{j}$. By the construction method, $\{(\rho_{i},w):w\in U\cup V^{+}\backslash\{-u_{i}\}\}$ constitutes a counterexample to the first clause of the inductive conditions and clearly $\{(\rho_{i},w):w\in U\cup V^{+}\backslash\{-u_{i}\}\}\subseteq\bigcup_{j<i}x_{j}$ which contradicts the corollary of IH;
    \item If there is no such $U$, let $x_{i}:=\bigcup_{j<i}x_{j}\cup\{(\rho_{i}\circ\mu,v):\mu[\Delta(v)]=\Delta(u_{i})\mbox{ and }\mu v=u_{i}\}$, then the inductive conditions is satisfied in the same way as the previous case.
\end{itemize}

Let $x=\bigcup_{i<\gamma}x_{i}$. By the same reasoning as the corollary of IH, it can be obtained that $x$ satisfies the induction conditions. Then we have
\begin{itemize}
    \item $x$ obviously satisfies the first and the fourth clauses of the definition of the atom;
    \item For any $\rho\in M,u,v,w\in B$ satisfying $u\cap v\neq\emptyset$, $(\rho,u),(\rho,v)\in x$ and $u\cap v\subseteq w$, assume $(\rho,w)\not\in x$, we know $(\rho,-w)\in x$, then $\{(\rho,u),(\rho,v),(\rho,-w)\}$ constitutes a counterexample to the inductive conditions, contradiction. Hence $(\rho,w)\in x$. Hence $x$ satisfies the second and the third clauses of the definition of the atom.
\end{itemize}
Hence, $x$ is an atom.
\end{proof}

As the reader may have realized, we want to construct an ``ultrafilter space" of $\mathcal{C}$ and make its dimension $\beta$, each atom being an ``ultrafilter" point in this space. Then, the clopen sets in this space are naturally defined as follows: for each $(\rho,u)\in M\times B$, define $X_{(\rho,u)}=\{x\in X_{\mathcal{C}}:(\rho,u)\in x\}$.

\begin{lemma}\label{2t1}
For any $(\rho,u),(\mu,v)\in M\times B$, there is a $\rho'\in M$, $u',v'\in B$ s.t. $X_{(\rho,u)}=X_{(\rho',u')}$, $X_{(\mu,v)}=X_{(\rho',v')}$.
\end{lemma}
\begin{proof}
obviously there is $\eta\in \beta^{\beta},u'\in B$ s.t. $\eta u'=u$ and $\Delta(u')\cap\Delta(v)=\emptyset$. We know $X_{(\rho\circ\eta,u')}=X_{(\rho,u)}$. Let $\rho'\in M$ contain $(\rho\circ\eta)|_{\Delta(u')}$ and $\mu|_{\Delta(v)}$, we know $X_{(\rho,u)}=X_{(\rho',u')},X_{(\mu,v)}=X_{(\rho',v})$.
\end{proof}
\begin{proposition}\label{ast}
$B^{X}:=\{X_{(\rho,u)}:(\rho,u)\in M\times B\}$ is a basis of Stone topology on $X_{\mathcal{C}}$. We write the induced Stone topology as $\tau^{X}$.
\end{proposition}
\begin{proof}
For any $(\rho,u),(\mu,v)\in M\times B$ s.t. $X_{(\rho,u)}\cap X_{(\mu,v)}\neq\emptyset$, by Lemma \ref{2t1}, there is $\rho'\in M$, $u',v'\in B$ s.t. $X_{(\rho,u)}=X_{(\rho',u')}$, $X_{(\mu,v)}=X_{(\rho',v')}$. Then by the second clause of definition of atom, $X_{(\rho,u)}\cap X_{(\mu,v)}=X_{(\rho',u'\cap v')}$. Hence $B^{X}$ is closed for taking finite intersection. Apparently, both $X_{(1,|\mathcal{C}|)}=X_{\mathcal{C}},X_{(1,\emptyset)}=\emptyset$ are in $B^{X}$\footnote{Here $1:\beta\rightarrow\alpha$ in $X_{(1,|\mathcal{C}|)}$ 
denotes the inclusion mapping, which is the same later in the text.}. Hence $B^{X}$ is a topological basis.

By the definition, $-X_{(\rho,u)}=X_{(\rho,-u)}$, then sets in $B^{X}$ are all clopen sets. Hence $(X_{\mathcal{C}},\tau^{X})$ is zero-dimensional.

For any $x\neq y\in X_{\mathcal{C}}$, there must be $(\rho,u)\in M\times B$ s.t. $x\in X_{(\rho,u)}, y\not\in X_{(\rho,u)}$. Hence $(X_{\mathcal{C}},\tau^{X})$ is a $T_{2}$ space.

For any $U\subseteq B^{X}$ with finite intersection property, define $x_{0}=\{(\rho,u):X_{(\rho,u)}\in U\}$. It is easy to see that $x_{0}$ satisfies the inductive conditions in the proof of Proposition \ref{ate}. Then, in the same way as the proof of Proposition \ref{ate}, $x_{0}$ can be expanded to an atom $x$ and $x\in\bigcap U$. By the arbitrariness of $U$, we know $(X_{\mathcal{C}},\tau^{X})$ is compact.

$ $
\end{proof}

For any $i,j\in\alpha$, we know there are $i',j'\in\beta,\rho\in M$ s.t. $\rho(i')=i,\rho(j')=j$, then define $D_{ij}=X_{(\rho,D^{\mathcal{C}}_{i'j'})}$ (It is easy to see this is well-defined). For any $i\in\alpha$, define relation $\backsim_{i}$ on $X_{\mathcal{C}}$ as: $x\backsim_{i}y\Leftrightarrow \{(\rho,u)\in x:i\not\in \rho[\Delta(u)]\}=\{(\rho,u)\in y:i\not\in \rho[\Delta(u)]\}$.

\begin{lemma}\label{el}
For $(\rho,u),(\mu,v)\in M\times B$, $i\in\beta$ satisfying $\mu(i)\not\in \rho[\Delta(u)]$, we have $X_{(\rho,u)}\cap X_{(\mu,[v]_{i})}\neq\emptyset\Rightarrow X_{(\rho,u)}\cap X_{(\mu,v)}\neq\emptyset$.
\end{lemma}
\begin{proof}
By Lemma \ref{2t1}, w.l.o.g. let $\rho=\mu$, then we have $i\not\in\Delta(u)$. If $X_{(\rho,u)}\cap X_{(\mu,v)}=X_{(\rho,u\cap v)}=\emptyset$, then by definition, $u\cap v=\emptyset$. then $u\cap[v]_{i}=[u\cap v]_{i}=[\emptyset]_{i}=\emptyset$, then $X_{(\rho,u)}\cap X_{(\mu,[v]_{i})}=X_{(\rho,u\cap [v]_{i})}=\emptyset$.
\end{proof}

\begin{theorem}\label{alexpan}
Let $\mathcal{C}_{\beta}$ be a basis-finite cylindric space. For $\alpha\ge\beta$, define $\mathcal{C}^{\alpha}=((X_{\mathcal{C}},\tau^{X}),\{\backsim_{i}:i\in\alpha\},\{D_{ij}:i,j\in\alpha\})$, then
    \begin{enumerate}
        \item $\mathcal{C}^{\alpha}$ forms an FOL space;
        \item If $\alpha=\beta$, then there is a basis-preserving C-mapping $f:\mathcal{C}_{\beta}\rightarrow\mathcal{C}^{\alpha}$ which is an injection if $\mathcal{C}_{\beta}$ is a $T_{2}$ space.
    \end{enumerate}
\end{theorem}
\begin{proof}
To prove that $\mathcal{C}^{\alpha}$ is an FOL space, by Proposition \ref{ast}, we just need to prove that it is a cylindric space:
\begin{itemize}
    \item \textbf{The commutativity of $\{\backsim_{i}:i\in\alpha\}$}
    
    For $x,y,z\in X_{\mathcal{C}}$, if $x\backsim_{i}z\backsim_{j}y$, let $x^{ij}:=\{(\rho,u)\in x:i,j\not\in \rho[\Delta(u)]\},y^{ij}:=\{(\rho,u)\in y:i,j\not\in \rho[\Delta(u)]\}$, then $x^{ij}=y^{ij}$. Let $x^{j}=\{(\rho,u)\in x:j\not\in \rho[\Delta(u)]\}$, $y^{i}=\{(\rho,u)\in y:i\not\in \rho[\Delta(u)]\}$. Write $j':=\rho^{-1}(j),i':=\mu^{-1}(i)$, then for any $(\rho,u)\in x^{j}$ and $(\mu,v)\in y^{i}$, $(\rho,[u]_{i'})\in x^{ij}=y^{ij}\ni (\mu,[v]_{j'})$, then $x,y\in X_{(\rho,[u]_{i'})}\cap X_{(\mu,[v]_{j'})}\neq\emptyset$. Using Lemma \ref{el} twice, we get $X_{(\rho,u)}\cap X_{(\mu,v)}\neq\emptyset$. By compactness, we know $\bigcap\{X_{(\rho,u)}:(\rho,u)\in x^{j}\cup y^{i}\}\neq\emptyset$ and for any $z'$ in this set, $x\sim_{j}z'\sim_{i}y$.
    \item \textbf{$B^{X}$ is a cylindric basis about $\{\backsim_{i}:i\in\alpha\}$}
    
    By definition, $B^{X}$ is closed under taking intersection, union, complement and  for any $i\in\beta$, $(\rho,u)\in B^{X}$. write $i':=\rho^{-1}(i)$, we show $[X_{(\rho,u)}]_{i}=X_{(\rho,[u]_{i'})}$:
    
    \begin{tabular}{rl}
         &$x\in[X_{(\rho,u)}]_{i}$\\
         $\Leftrightarrow$&$\exists y\backsim_{i}x$ s.t. $(\rho,u)\in y$\\
         $\Leftrightarrow$&$x^{i}\cup\{(\rho,u)\}$ can be expanded to a atom\\
         $\Leftrightarrow$&$\{X_{\rho,u}\}\cup\{X_{(\mu,v)}:(\mu,v)\in x^{i}\}$ has finite intersection property \\&(by compactness)\\
         $\Leftrightarrow$& For any $(\mu,v)\in x^{i},X_{(\rho,u)}\cap X_{(\mu,v)}\neq\emptyset$\\
         $\Leftrightarrow$& For any $(\mu,v)\in x^{i},X_{(\rho,[u]_{i'})}\cap X_{(\mu,v)}\neq\emptyset\quad$\\&(by definition of $x^{i}$, $\rho(i')=i\not\in\mu(\Delta(v))$, then by Lemma \ref{el} we get this)\\
         $\Leftrightarrow$&$(\rho,[u]_{i'})\in x^{i}\subseteq x\quad$\\&($\Rightarrow$: otherwise $(\rho,-[u]_{i'})\in x^{i}$; $\Leftarrow$: any $X_{(\mu,v)}$ in the above equation is $\sim_{i}$-saturated)\\
         $\Leftrightarrow$&$x\in X_{(\rho,[u]_{i'})}$
    \end{tabular}
    
    Hence, for any $i\in\alpha$, $B^{X}$ is closed under taking $\backsim_{i}$-saturation, and then is a cylindric basis about $\{\backsim_{i}:i\in\alpha\}$.
    \item \textbf{$\{D_{ij}:i,j\in\alpha\}$ is diagonal family}
    
    For any $i\neq j\in\alpha$, let $x,y\in D_{ij}$ be different and $x\sim_{i}y$, we know there is $(\rho,u)$ s.t. $x\in X_{(\rho,u)},y\in X_{(\rho,-u)}$, then for $x_{\rho}:=\{u:(\rho,u)\in x\}$ and $y_{\rho}:=\{u:(\rho,u)\in y\}$, we have $x_{\rho}\neq y_{\rho}$ and $D^{\mathcal{C}}_{\rho^{-1}(i)\rho^{-1}(j)}\in x_{\rho}\cap y_{\rho}$. By definition of atom, both $x_{\rho},y_{\rho}$ form ultrafilters on $B$ which is the basis of $\mathcal{C}$, then by compactness and $T_{2}$ property, both $\bigcap x_{\rho},\bigcap y_{\rho}$ are singletons. We denote points in $\bigcap x_{\rho},\bigcap y_{\rho}$ as $a_{x},a_{y}$. By $x\sim_{i}y$, we know $\{u\in x_{\rho}:\rho^{-1}(i)\not\in\Delta(u)\}=\{u\in y_{\rho}:\rho^{-1}(i)\not\in\Delta(u)\}$, then $a_{x}\sim_{\rho^{-1}(i)}a_{y}$ which contradicts $D^{\mathcal{C}}_{\rho^{-1}(i)\rho^{-1}(j)}\in x_{\rho}\cap y_{\rho}$.
\end{itemize}

For the case of $\alpha=\beta$, let $f:|\mathcal{C}|\rightarrow X_{\mathcal{C}}$ satisfy: for any $a\in|\mathcal{C}|,f(a)\in\bigcap\{X_{(1,u)}:u\mbox{ is clopen neighborhood of }a\}$ (It is easy to see that this set is a singleton. hence we actually give a precise definition of $f$). Now we prove $f$ is a basis-preserving C-mapping generally and is an injection if $\mathcal{C}_{\beta}$ is a $T_{2}$ space.
\begin{itemize}
    \item \textbf{Continuity, basis-preserving, and diagonal keeping.}
    
    Here it is only proved that for the case of clopen sets in the cylindric basis, the case of general open sets is an obvious corollary.
    
    For any $X_{(\rho,u)}\in B^{X}$, we know there is $u'\in B$, $X_{(1,u')}=X_{(\rho,u)}$. Now we prove, for any $u\in B$, $f^{-1}[X_{(1,u)}]=u$:
    
    By the definition of $f$, $f[u]\subseteq X_{(1,u)}$ i.e. $u\subseteq f^{-1}[X_{(1,u)}]$. If there is $a\in f^{-1}[X_{(1,u)}]\backslash u$, we know $a\in-u$, then $a\in f^{-1}[X_{(1,-u)}]=f^{-1}[-X_{(1,u)}]$, contradiction.
    \item\textbf{The case of $T_{2}$ space}
    
    For any $a,b\in|\mathcal{C}|$ s.t. $f(a)=f(b)$, we know $\bigcap\{X_{(1,u)}:u$ is a clopen neighborhood of $a\}=\bigcap\{X_{(1,u)}:u$ is clopen neighborhood of $b\}$, then $\{X_{(1,u)}:u$ is clopen neighborhood of $a\}=\{X_{(1,u)}:u$ is clopen neighborhood of $b\}$, then $\{u:u$ is clopen neighborhood of $a\}=\{u:u$ is clopen neighborhood of $b\}$, then $a=b$.
    \item\textbf{structure homomorphism and C-mapping}
    
    For any $i\in\beta$ and any $a\in|\mathcal{C}|$, if $b\sim_{i}a$, then $\{u:u$ is $ i$-saturated clopen neighborhood of $a\}=\{u:u$ is $i$-saturated clopen neighborhood of $b\}$, then $\{(1,u):u$ is $i$-saturated clopen neighborhood of $a\}=\{(1,u):u$ is $i$-saturated clopen neighborhood of $b\}$, then $f(a)\sim_{i}f(b)$.
    
    For any $X_{(1,u)}\in B^{X}$ s.t. $[f(a)]_{i}\cap X_{(1,u)}\neq\emptyset$, we know $f(a)\in[X_{(1,u)}]_{i}=X_{(1,[u]_{i})}$, then $a\in[u]_{i}$, then there is $b\sim_{i}a$ s.t. $b\in u$, and then $f(b)\in X_{(1,u)}$.
\end{itemize}
\end{proof}

We call the $\mathcal{C}^{\alpha}$ in the above theorem the \textbf{$\alpha$-expansion space} of $\mathcal{C}$.

\begin{lemma}
If $\mathcal{C}$ is an FOL space, then for any $x\in|\mathcal{C}^{\alpha}|$, $\bigcap\{u|(1,u)\in x\}$ is a singleton.
\end{lemma}
\begin{proof}
By the definition of atom, $\{u|(1,u)\in x\}$ has finite intersection property, then by compactness, $\bigcap\{u|(1,u)\in x\}\neq\emptyset$.

If there are two different points $a,b\in\bigcap\{u|(1,u)\in x\}$, then there are clopen set $v\in\tau^{\mathcal{C}}$ s.t. $a\in v,b\in-v$, then neither $(1,v)$ nor $(1,-v)$ is in $x$, which contradicts the definition of atom. Hence $\bigcap\{u|(1,u)\in x\}$ is singleton.
\end{proof}

\begin{theorem}\label{kex}
Let $\mathcal{C}_{\beta}$ be an FOL space, $\alpha\ge\beta$, then up to S-homemorphism there is a unique $\alpha$-dimensional FOL space $\mathcal{C'}_{\alpha}$ and basis-preserving C-surjection $f$ s.t. $f:\mathcal{C'}_{\alpha}\rightarrow\mathcal{C}_{\beta}$.
\end{theorem}
\begin{proof}
$ $
\begin{itemize}
    \item\textbf{Existence:} Define $f:|\mathcal{C}^{\beta}|\rightarrow|\mathcal{C}|$ as $x\mapsto\bigcap\{u|(1,u)\in x\}$, we verify that $f$ is a basis-preserving C-surjection.
    \begin{itemize}
        \item \textbf{Continuity, basis-preserving and diagonal keeping}
    
    Here it is only proved that for the case of clopen sets in the cylindric basis, the case of general open sets is an obvious corollary.
    
    For any $X_{(\rho,u)}\in B^{X}$ s.t. $\Delta(X_{(\rho,u)})\subseteq\beta$, we know $\rho[\Delta(u)]\subseteq\beta$, then there is $u'\in B$ s.t. $X_{(1,u')}=X_{(\rho,u)}$ ($X_{1,u'}=D_{ij}$ when $u'=D^{\mathcal{C}}_{ij}$). Now we just need to prove for any $u\in B$, $f^{-1}[u]=X_{(1,u)}$:
        
        $x\in f^{-1}[u]\Leftrightarrow f(x)\in u\Leftrightarrow (1,u)\in x\Leftrightarrow x\in X_{(1,u)}$.
        \item \textbf{surjection}
        
        For any $a\in|\mathcal{C}|$, consider $x'=\{X_{(1,u)}:a\in u\}$. It's easy to see $x'$ can be expanded to an atom $x$, then by definition of $f$, $f(x)=a$.
        \item \textbf{structure homomorphism and C-mapping}
        
        For any $x,y\in|\mathcal{C}^{\alpha}|$, and any $i\in\alpha$ s.t. $x\sim_{i}y$, if $i\in\alpha\backslash\beta$ we know $f(x)=f(y)$; if $i\in\beta$, then $\{(1,u)\in x:i\not\in\Delta(u)\}=\{(1,u)\in y:i\not\in\Delta(u)\}$, then $\{u\in B^{\mathcal{C}}:f(x)\in u\mbox{ and }i\not\in\Delta(u)\}=\{u\in B^{\mathcal{C}}:f(y)\in u\mbox{ and }i\not\in\Delta(u)\}$, then by Proposition \ref{cstc}, $f(x)\sim_{i}f(y)$.
        
        For any $u\in B,i\in\beta$, $f^{-1}[[u]_{i}]=X_{(1,[u]_{i})}=[X_{(1,u)}]_{i}=[f^{-1}[u]]_{i}$.
    \end{itemize}
    
    \item\textbf{Uniqueness:} 
    
    Let $\mathcal{C}'$ be a FOL space of $\alpha$ dimension, $f':\mathcal{C}'\rightarrow\mathcal{C}$ be a basis-preserving C-surjection. denotes the cylindric basis of $\mathcal{C}'$ as $B^{\mathcal{C'}}$, we know $\{f'^{-1}[u]:u\in B\}=\{u\in B^{\mathcal{C'}}:\Delta(u)\subseteq\beta\}$, then $B^{\mathcal{C}'}=\{\rho (f'^{-1}[u]):u\in B,\rho:\alpha^{\alpha}\}=\{\rho (f'^{-1}[u]):u\in B,\rho:\alpha^{\beta}\mbox{ is injection }\}$. For each $a\in\mathcal{C'}$, $a_{at}:=\{(\rho,u)\in M\times B:a\in\rho f'^{-1}[u]\}$. By the definition of atom, $a_{at}$ is an atom in $M\times B$ and from each atom in $M\times B$ we can get an ultrafilter on $B^{\mathcal{C'}}$, and then, by compactness, get a point $a$ in $|\mathcal{C'}|$ s.t. $a_{at}=x$. We define $g:|\mathcal{C'}|\rightarrow|\mathcal{C}^{\alpha}|$ as $a\mapsto a_{at}$. Clearly, it is a bijection, and we have $f\circ g=f'$.
   
   Now, by Proposition \ref{bh}, to prove $g$ is an S-homeomorphism, we just need to prove that $g$ is a basis-preserving S-mapping:
   
   \begin{itemize}
        \item \textbf{Continuity, basis-preserving, and diagonal keeping}
    
    Here it is only proved that for the case of clopen sets in the cylindric basis, the case of general open sets is an obvious corollary.
    
    For $D_{ij}\in B^{X}$, we know there is $\rho\in M,i'j'\in\alpha,D_{ij}=X_{(\rho,D^{\mathcal{C}}_{i'j'})}$ and $D^{\mathcal{C'}}_{ij}=\rho D^{\mathcal{C}}_{i'j'}$ (here $\rho(i')=i,\rho(j')=j$).
    
    Hence, we just need to prove, for any $X_{(\rho,u)}\in B^{X}$, $g^{-1}[X_{(\rho,u)}]=\rho f'^{-1}[u]$:
        
        $a\in\rho  f'^{-1}[u]\Leftrightarrow (\rho,u)\in a_{at}=g(a)\Leftrightarrow g(a)\in X_{(\rho,u)}$.
        \item \textbf{Structure homomophism:}
        
        For any $a,b\in|\mathcal{C'}|$, 
        
        \begin{tabular}{rl}
             &$i\in\alpha$, $a\sim_{i}b$\\
             $\Rightarrow$&there is no $u'\in B^{\mathcal{C'}}$ s.t. $a\in u'\not\ni b$ and $i\not\in\Delta(u')$\\
             $\Rightarrow$&there is no $u\in B$, $\rho:\alpha^{\alpha}$ s.t. $a\in \rho u\not\ni b$ and $i\not\in\rho[\Delta(u)]$\\
             $\Rightarrow$&there is no $u\in B$, $\rho:\alpha^{\alpha}$ s.t. $f(a)\in X_{(\rho,u)}\not\ni f(b)$ and $i\not\in\Delta(X_{(\rho,u)})$\\
             $\Rightarrow$&$f(a)\sim_{i}f(b)$
        \end{tabular}
    \end{itemize}
\end{itemize}
\end{proof}

Call the mapping constructed in the above proof the \textbf{$\alpha$-expansion mapping} of $\mathcal{C}_{\beta}$.

\begin{remark}
In fact, by constructing quotient spaces, we can easily prove the opposite direction of this proposition, namely, for $\alpha\ge\beta$, each $\alpha$-dimensional FOL space can be mapped by a basis-preserving C-surjection to a $\beta$-dimensional FOL space, and this basis-preserving C-surjection and this $\beta$-dimensional FOL space are unique up to S-homeomorphism. This way, we get a one-to-one correspondence between FOL spaces with different dimensions.
\end{remark}
\section{Topologization of first-order logic semantics}\label{tph}
Based on the purely topological discussion in Section \ref{sbw}, this section establishes a connection between first-order logic and cylindric space and finally obtains a systematic topological representation of semantics of first-order logic. The basic pattern of this connection has been shown by examples in Section \ref{sglz}. To facilitate a rigorous discussion, in \ref{4.1}, we generalize the correspondence between clopen sets in the topologization space or model space and formulas in these examples as a mapping (called $\mathcal{L}$-formation) from formulas of a language to clopen sets in a cylindric space. In \ref{4.2}, we initially establish the connection between first-order structures and points in FOL spaces. And on this basis, a topological proof of the first-order logic compactness theorem is given. This means that all model spaces are FOL spaces. In \ref{4.3}, for each infinite ordinal $\alpha$, each theory $T$ is associated with a unique $\alpha$-dimensional FOL space called $\alpha$-model space of $T$; and the class of $\alpha$-dimensional FOL spaces is shown to be the class of S-homeomorphic spaces of $\alpha$-model spaces. In this way, we provide enough points for representing uncountable models of $T$. In \ref{4.4}, the topological representations of first-order structures and elementary embeddings are finally given in two theorems, respectively.
\subsection{$\alpha$-topologization and $\mathcal{L}$-formation}
In subsection \ref{sglz}, we define the $\omega$-topologization of first-order structure which can be easily generalized to any infinite ordinal $\alpha$: for a first-order language $\mathcal{L}$ and an infinite ordinal $\alpha$, let $\mathcal{L}_{\alpha}$ be the language obtained by expanding the set of variables $\{v_{i}:i\in\omega\}$ in $\mathcal{L}$ to $\{v_{i}:i\in\alpha\}$. The definition of $\mathcal{L}_{\alpha}$-formulas and the definition of substitution for these formulas are then identical to those in $\mathcal{L}$. For a $\mathcal{L}$-structure $\mathfrak{A}$, we consider the $\alpha$ power of its domain $A$, written as $A^{\alpha}$, and the clopen set $A_{\phi}$ corresponding to $\phi\in\Lambda_{\mathcal{L}_{\alpha}}$. Based on this, applying exactly the same method as that used to handle the $\omega$-topologization, we can obtain the \textbf{$\alpha$-topologization} of $\mathfrak{A}$ (denoted as $\mathcal{C}_{\alpha}^{\mathfrak{A}}$) and prove that $\mathcal{C}_ {\alpha}^{\mathfrak{A}}$ is also a basis-finite cylindric space.

From the construction of $\mathcal{C}_{\alpha}^{\mathfrak{A}}$, we naturally obtain a mapping $\phi\mapsto A_{\phi}$ from $\Lambda_{\mathcal{L}_{\alpha}}$ to the cylindric basis of $\mathcal{C}_{\alpha}^{\mathfrak{A}}$, which can be defined abstractly on a general basis-finite cylindric space as follows.

\begin{definition}
For an infinite ordinal $\alpha$, a basis-finite cylindric space $\mathcal{C}_{\alpha}$, if the mapping $l:\Lambda_{\mathcal{L}_{\alpha}}\rightarrow \tau_{\mathcal{C}}$ satisfies the following conditions, then we call $l$ an \textbf{$\mathcal{L}$-formation} on $\mathcal{C}_{\alpha}$:
\begin{itemize}
    \item For $n$-ary predicate $R$, $l(Rv_{0}...v_{n-1})$ is a clopen set s.t. $\Delta(l(Rv_{0}...v_{n-1}))\subseteq n$;
    \item For $i,j\in\alpha$, $l(v_{i}=v_{j})=D^{\mathcal{C}}_{ij}$;
    \item For $i,j\in\alpha$, $l(\phi(v_{i}/v_{j}))=[l(\phi)\cap D^{\mathcal{C}}_{ij}]_{j}$;
    \item $l(\neg\phi)=|\mathcal{C}|\backslash l(\phi)$;
    \item $l(\phi\wedge\psi)=l(\phi)\cap l(\psi)$;
    \item $l(\exists v_{i}\phi)=[l(\phi)]_{i}$;
    \item $ran(l)$ is a topological basis of $\tau_{\mathcal{C}}$.
\end{itemize}
\end{definition}

For any $\mathcal{L}$-structure $\mathfrak{A}$ and its $\alpha$-topologization $\mathcal{C}^{\mathfrak{A}}_{\alpha}$, define $l_{\mathfrak{A}}:\Lambda_{\mathcal{L}_{\alpha}}\rightarrow\tau_{\mathcal{C}^{\mathfrak{A}}_{\alpha}}$ as $\phi\mapsto A_{[\phi]}$. It is easy to verify that $l_{\mathfrak{A}}$ is an $\mathcal{L}$-formation. Call $l_{\mathfrak{A}}$ the \textbf{natural $\mathcal{L}$-formation} on $\mathcal{C}^{\mathfrak{A}}_{\alpha}$.

For any $\mathcal{L}$-theory $T$ and its model space $\mathcal{C}^{T}$, let $l_{T}:\Lambda_{\mathcal{L}}\rightarrow\tau_{\mathcal{C}^{T}}$ be $\phi\mapsto S_{[\phi]}$. It is easy to verify that $l_{T}$ is also an $\mathcal{L}$-formation. Call $l_{T}$ the \textbf{natural $\mathcal{L}$-formation} on $\mathcal{C}^{T}$.

\begin{definition}
Let $f:\mathcal{C}_{\alpha}\rightarrow\mathcal{C'}_{\beta}$ be an S-mapping, $\mathcal{L}\ge\mathcal{L'}$ be two languages, $l,l'$ be $\mathcal{L}$-formation on $\mathcal{C}$ and $\mathcal{L'}$-formation on $\mathcal{C'}$ respectively, if for any $\phi\in\Lambda_{\mathcal{L'}_{\beta}},f^{-1}[l'(\phi)]=l(\phi)$, then we say \textbf{$f$ maps $l$ to $l'$} and denote this as $f:(\mathcal{C},l)\rightarrow(\mathcal{C}',l')$.
\end{definition}

\begin{proposition}\label{ltl}
Let $\mathcal{C}_{\alpha},\mathcal{C}'_{\beta}$ be basis-finite cylindric space, $\mathcal{L}$ be a language, $f:\mathcal{C}\rightarrow\mathcal{C}'$ be basis-preserving C-mapping, then for any $\mathcal{L}$-formation $l'$ on $\mathcal{C}'_{\beta}$, there is a unique $\mathcal{L}$-formation $l$ on $\mathcal{C}_{\alpha}$ s.t. $f:(\mathcal{C},l)\rightarrow(\mathcal{C}',l')$.
\end{proposition}
\begin{proof}
$ $

\begin{itemize}
    \item\textbf{Existence: }Define $l:\Lambda_{\mathcal{L}_{\alpha}}\rightarrow 2^{\mathcal{C}}$ as: for any $n$-ary predicate $R$, $l(Rv_{0}...v_{n-1}):=f^{-1}[l'(Rv_{0}...v_{n-1})]$, for any other formula $\phi$, $l(\phi)$ is determined by the inductive conditions in the definition of $\mathcal{L}$-formation. By Proposition \ref{kjdw}, it's easy to see that for any $\phi\in\Lambda_{\mathcal{L}_{\beta}}$, $l(\phi)=f^{-1}[l'(\phi)]$. Now we just need to prove $ran(l)$ forms a basis of $\tau^{\mathcal{C}}$:
    \begin{itemize}
        \item For any $n$-ary predicate $R$, $l(Rv_{0}...v_{n-1})=f^{-1}[l'(Rv_{0}...v_{n-1})]$ has finite dimension, then by definition it is easy to see for any $\phi\in\Lambda_{\mathcal{L}_{\alpha}}$, $l(\phi)$ is a clopen set with finite dimension;
        \item By the inductive conditions in the definition, $ran(l)$ is closed under taking intersection, union, complement and $\sim_{i}$-saturation for any $i\in\alpha$.
        \item For any clopen set $u\in\tau^{\mathcal{C}}$ with finite dimension, we know there is a $\rho\in\alpha^{\beta}$ and a finite-dimensional clopen set $v$ s.t. $\Delta(v)\subseteq\beta$ and $\rho v=u$, then there is a $v'\in\tau^{\mathcal{C'}}$ s.t. $f^{-1}[v']=v$, by the definition of $\mathcal{L}$-formation, there is a $U'\subseteq ran(l')$ such that $v'=\bigcup U'$, then there is a $U\subseteq ran(l)$ such that $v=\bigcup U$, then $u=\rho\bigcup U=\bigcup\{\rho w:w\in U\}$. By the third clause of the inductive conditions of definition of $\mathcal{L}$-formation, $\{\rho w:w\in U\}\subseteq ran(l)$.
        
        For any open set $u\in\tau^{\mathcal{C}}$, we know there is a family of finite dimensional clopen sets $V\subseteq\tau^{\mathcal{C}}$ s.t. $u=\bigcup V$. As proved above, for each $v\in V$, there is $U\subseteq ran(l)$, $v=\bigcup U$, then there is a $U^{*}\subseteq ran(l)$ such that $u=\bigcup U^{*}$.
    \end{itemize}
    \item\textbf{Uniqueness: } For any $\mathcal{L}$-formation $l^{*}$ on $\mathcal{C}$, since the definition of $\mathcal{L}$-formation is inductive, if $l\neq l^{*}$, then there must be predicate $R$ s.t. $l^{*}(Rv_{0}...v_{n-1})\neq l(Rv_{0}...v_{n-1})=f^{-1}l'(Rv_{0}...v_{n-1})$.
\end{itemize}
\end{proof}
\begin{proposition}\label{ubk}
Let $\mathcal{C}_{\alpha},\mathcal{C}'_{\beta}$ be basis-finite cylindric spaces, where $\mathcal{C}'_{\beta}$ is a $T_{2}$ space, $\mathcal{L}$ be a language, $l,l'$ be $\mathcal{L}$-formations on $\mathcal{C}_{\alpha},\mathcal{C}'_{\beta}$ respectively. Then there is at most one C-mapping $f:(\mathcal{C},l)\rightarrow(\mathcal{C}',l')$.
\end{proposition}
\begin{proof}
Assume there are $f,f':(\mathcal{C},l)\rightarrow(\mathcal{C}',l')$ such that there is an $a\in|\mathcal{C}|$ satisfying $f(a)\neq f'(a)$, then since $ran(l')$ is a topological basis and $\mathcal{C}$ is $T_{2}$, we know there is a $\phi\in\Lambda_{\mathcal{L}_{\beta}}\subseteq\Lambda_{\mathcal{L}_{\alpha}}$ s.t. $f(a)\in l'(\phi)$, $f'(a)\in l'(\neg\phi)$, then $a\in l(\phi)\cap l(\neg\phi)=\emptyset$, contradiction.
\end{proof}
\begin{proposition}\label{blbk}
Let $\mathcal{C}_{\alpha},\mathcal{C}'_{\beta}$ be basis-finite cylindric spaces, $\mathcal{L},\mathcal{L'}$ be two languages, $l,l'$ be $\mathcal{L}$-formation on $\mathcal{C}_{\alpha}$ and $\mathcal{L'}$-formation on $\mathcal{C}'_{\beta}$ respectively. For a structure homomophism $f:(|\mathcal{C}|,\sim_{i})_{i\in\alpha}\rightarrow(|\mathcal{C'}|,\sim_{i})_{i\in\alpha}$ s.t. for any $\phi\in\Lambda_{\mathcal{L}_{\beta}}$, $f^{-1}[l'(\phi)]=l(\phi)$, we have $f:(\mathcal{C},l)\rightarrow(\mathcal{C}',l')$ is a C-mapping and if $\mathcal{L}=\mathcal{L'}$, then $f$ is basis-preserving.
\end{proposition}
\begin{proof}
Both $ran(l),ran(l')$ are cylindric bases of corresponding spaces, then $f$ is continuous and clearly preserves diagonal. Then by Proposition \ref{kjdw}.(2), the case of C-mapping clearly holds. If $\mathcal{L}=\mathcal{L'}$, then for each clopen set $l(\phi)$ in the cylindric basis of $\mathcal{C}$ with dimension in $\beta$, we can assume the $\phi$ is in $\Lambda_{\mathcal{L}_{\beta}}$ w.l.o.g., then $f^{-1}[l'(\phi)]=l(\phi)$. hence $f$ is basis-preserving.
\end{proof}\label{4.1}
\subsection{Compactness of model space}
\begin{definition} Let $\mathcal{C}_{\alpha}$ be an FOL space, $b\in|\mathcal{C}|$
\begin{itemize}
    \item If for any $i\in\alpha$ and any clopen set $u$, $b\in [u]_{i}\Rightarrow b\in\bigcup_{j\in\alpha} u(\frac{j}{i})$, then we call $b$ a \textbf{model point};
    \item If $b$ is a model point and $\mathcal{D}^{\mathcal{C}}$ doesn't cover the set of all factors of $b$, then we call $b$ a \textbf{big model point}.
\end{itemize}
\end{definition}

\begin{lemma}\label{mtm}
Let $\mathcal{C}_{\alpha}$ be an FOL space, $b\in|\mathcal{C}|$ is a model point, then for any $\mathcal{L}$-formation $l$ on $\mathcal{C}$, there is a $\mathcal{L}$-structure $\mathfrak{A}$ and basis-preserving C-mapping $f:(\mathcal{C}^{\mathfrak{A}}_{\alpha},l_{\mathfrak{A}})\rightarrow(\mathcal{C}_{\alpha},l)$ s.t. $b\in ran(f)$ and: 1), $\|\mathfrak{A}\|\le\alpha$; 2), there is a domain point in $f^{-1}[b]$; 3), for any $c\prec b$, $c\in ran(f)$; 4), for any $c\in ran(f)$, if there is a domain point in $f^{-1}[c]$, then $c$ is a model point.
\end{lemma}
\begin{proof}

Let $\alpha':=\{i\in\alpha:\forall j<i,b\not\in D^{\mathcal{C}}_{ij}\}$.
\begin{itemize}
    \item We construct an $\mathcal{L}$-structure $\mathfrak{A}=(A,R^{\mathfrak{A}}_{i})_{i\in I}$ as: $A:=\{x_{i}:i\in\alpha'\}$, we know $\|A\|\le\alpha$; for an $n$-ary predicate $R$ and $x_{i_{1}},...,x_{i_{n}}\in A$, $\langle x_{i_{1}}...x_{i_{n}}\rangle\in R^{\mathfrak{A}}\Leftrightarrow b\in l(Rv_{i_{1}}...v_{i_{n}})$.
    \item For any sequence $a\in A^{\alpha}$, we know for any $i\in\alpha$, there is an $i'\in\alpha'$ s.t. $a(i)=x_{i'}$, let $\rho_{a}:\alpha^{\alpha}$ be defined as $\rho_{a}(i)=i'$. Define mapping $f:A^{\alpha}\rightarrow|\mathcal{C}|$ as $a\mapsto\rho_{a}^{-1}b$.
    \item We show $f$ maps $l_{\mathfrak{A}}$ to $l$:
    
    Use induction on $\phi\in\Lambda_{\mathcal{L}_{\alpha}}$:
    
    \textbf{BS: }For any $n$-ary predicate $R$, $i_{1},...,i_{n}\in\alpha$,
    
    \begin{tabular}{rl}
         &$a\in A_{Rv_{i_{1}}v_{i_{n}}}$\\
         $\Leftrightarrow$&$\langle a(i_{1})...a(i_{n})\rangle\in R^{\mathfrak{A}}$\\
         $\Leftrightarrow$&$\langle x_{\rho_{a}(i_{1})}...x_{\rho_{a}(i_{n})}\rangle\in R^{\mathfrak{A}}$\\
         $\Leftrightarrow$&$b\in l(Rv_{\rho_{a}(i_{1})}...v_{\rho_{a}(i_{n})})$\\
         $\Leftrightarrow$&$f(a)\in\rho_{a}^{-1}l(Rv_{\rho_{a}(i_{1})}...v_{\rho_{a}(i_{n})})=l(Rv_{i_{1}}...v_{i_{n}})$\\
         $\Leftrightarrow$&$a\in f^{-1}[l(Rv_{i_{1}}...v_{i_{n}})]$
    \end{tabular}
    
    Hence $l_{\mathfrak{A}}(Rv_{i_{1}}...v_{i_{n}})=A_{Rv_{i_{1}}v_{i_{n}}}=f^{-1}[l(Rv_{i_{1}}...v_{i_{n}})]$.
    
    \textbf{IH: }For $\phi\in\Lambda_{\mathcal{L}_{\alpha}}$ with complexity less than $l$, $l_{\mathfrak{A}}(\phi)=A_{\phi}=f^{-1}[l(\phi)]$.
    
    \textbf{IS: }The complexity of $\phi\in\Lambda_{\mathcal{L}_{\alpha}}$ is $l$:
    
    The cases where $\phi$ has the form $\neg\psi,\psi\wedge\chi$ are trivial. Here we consider the case $\phi$ has the form $\exists v_{i}\psi$:
    
    \begin{tabular}{rl}
     &$a\in f^{-1}[l(\phi)]$\\
     $\Leftrightarrow$& For each $j\not\in\rho_{a}^{-1}[\rho_{a}(\Delta(l(\psi)))],b\in \rho_{a} l(\phi)=\rho_{a} [l(\psi)]_{i}=[\rho_{a}l(\psi(v_{j}/v_{i}))]_{\rho_{a}(j)}$\\
     $\Leftrightarrow$& there is a $ j'\in\alpha$, for each $j\not\in\rho_{a}^{-1}[\rho_{a}(\Delta(l(\psi)))]$, $b\in (\rho_{a}l(\psi(v_{j}/v_{i})))(\frac{j'}{\rho_{a}(j)})$\\&$=\rho_{a}(j'/j)(l(\psi)(\frac{j}{i}))=\rho_{a}(j'/i)l(\psi)$\footnotemark (By the property of model point)\\
     $\Leftrightarrow$& there is an $a'$ s.t. $f(a')\in l(\psi)$ and for some $j'$, $\rho_{a'}=\rho_{a}(j'/i)$\\
     &(The existence of $a'$: Let $j'$ be the least number satisfying $b\in \rho_{a}(j'/i)l(\psi)$, we can obtain $a'$)\\
     $\Leftrightarrow$& there is an $a'$ s.t. $a'\sim_{i}a$ and $a'\in A_{\psi}$\\
     $\Leftrightarrow$&$a\in A_{\phi}$
\end{tabular}
\footnotetext{Here $\rho(j'/i)$ denotes such a mapping: its only difference from $\rho$ is that $\rho(j'/i)$ shoots $i$ to $j'$.}

Hence, for any $\phi\in\Lambda_{\mathcal{L}_{\alpha}}$, $l_{\mathfrak{A}}(\phi)=A_{\phi}=f^{-1}[l(\phi)]$.
\item\textbf{$f$ is a basis-preserving C-mapping}

For any $i\in\alpha$, 

\begin{tabular}{rl}
     &$a\sim_{i}a'\in A^{\alpha}$\\
     $\Rightarrow$& for any $j\neq i,a(j)=a'(j)$\\
     $\Rightarrow$&for any $j\neq i,\rho_{a}(j)=\rho_{a'}(j)$\\
     $\Rightarrow$&$\{\rho_{a}^{-1}u:u$ is a clopen neighbourhood of $b$ and $\rho_{a}^{-1}u$ is $i$-saturated$\}=$\\&$\{\rho_{a'}^{-1}u:u$ is a clopen neighbourhood of $b$ and $\rho_{a}^{-1}u$ is $i$-saturated$\}$\\
     $\Rightarrow$&$\{u:u$ is $i$-saturated clopen neighbourhood of $f(a)\}=$\\&$\{u:u$ is an $i$-saturated clopen neighbourhood of $f(a')\}$\\
     $\Rightarrow$&$f(a)\sim_{i}f(a')$
\end{tabular}

As proven above, for any $\phi$, $l_{\mathfrak{A}}(\phi)=A_{\phi}=f^{-1}[l(\phi)]$. Then by Proposition \ref{blbk}, $f$ is a basis-preserving C-mapping.

\item \textbf{$b\in ran(f)$ and there is a domain point in $f^{-1}[b]$}

For any $i\in\alpha$, if $i\not\in\alpha'$, then by definition of $\alpha'$ we know there is a $j<i$ s.t. $j\in\alpha'$ and $b\in D^{\mathcal{C}}_{ij}$. Write $i':=argmin_{j}\{x_{j}:j\in\alpha',b\in D^{\mathcal{C}}_{ij}\}$, we know there is a sequence $a\in A^{\alpha}$ s.t. $a(i)=i'$, then it's easy to see $f(a)=\rho_{a}^{-1}b=b$ and $a$ is a domain point.

\item \textbf{There is a $\rho:c\prec b\Rightarrow c\in ran(f)$}

Let $b'$ be a domain point in $f^{-1}[b]$, we define $c'\in|\mathcal{C}^{\mathfrak{A}}_{\alpha}|$ as $c'(i)=b'(\rho(i))$. It's easy to see for any $\phi$, $c'\in l_{\mathfrak{A}}(\phi)\Leftrightarrow b'\in l_{\mathfrak{A}}(\phi(\overline{v_{\rho(i)}}/\overline{v_{i}}))$, then $f(c')\in l(\phi)\Leftrightarrow b\in l(\phi(\overline{v_{\rho(i)}}/\overline{v_{i}}))=\rho l(\phi)$, then $f(c')=\rho^{-1}b=c$.

\item \textbf{There is a domain point in $f^{-1}[c]\Rightarrow c$ is a model point}

Let $a$ be this domain point. It is easy to see for any $i\in\alpha$ and any formula $\phi$,

\begin{tabular}{rl}
     &$a\in [l_{\mathfrak{A}}(\phi)]_{i}=l_{\mathfrak{A}}(\exists v_{i}\phi)$\\
     $\Rightarrow$&$a\in\bigcup_{j\in\alpha}\{l_{\mathfrak{A}}(\phi(v_{j}/v_{i}))\}=\bigcup_{j\in\alpha} l_{\mathfrak{A}}(\phi)(\frac{j}{i})$ 
\end{tabular}

hence if $c\in[l(\phi)]_{i}$, then $c\in \bigcup_{j\in\alpha} l(\phi)(\frac{j}{i})$. hence $c$ is a model point.
\end{itemize}
\end{proof}

\begin{lemma}\label{mex}
Let $\mathcal{C}_{\alpha}$ be an FOL space, $v$ be an $\emptyset$-dimensional complete closed set on it s.t. the number of clopen sets in the subspace induced by $v$ is not greater than $\alpha$, then for any $a\in v$, there is a model point $b\in v$ s.t. $a\prec b$. In particular, such a big model point exists if the diagonal family has no finite subset covering $v$.
\end{lemma}
\begin{proof}
Let $\lambda:\omega\times\alpha\rightarrow\alpha$ be a bijection, $\rho:\alpha\rightarrow\alpha$ be an injection s.t. $ran(\rho)=\lambda[1\times\alpha]$, then $\rho\{a\}$ is a non-empty complete closed set of dimension $\lambda[1\times\alpha]$. Use induction on $n\in\omega$, we construct closed sets $X_{n}$ and families $U_{n}$ of clopen sets as follow:
\begin{itemize}
    \item\textbf{BS: } $X_{0}:=\rho\{a\}$, $U_{0}:=\{[u]_{i}:u\mbox{ is clopen set},\rho\{a\}\subseteq [u]_{i}\mbox{ but for any }j\in\lambda[1\times\alpha],\rho\{a\}\not\subseteq u(\frac{j}{i})\}$.
    \item\textbf{IH: } For $m<n$, $X_{m}$ is complete closed set of dimension $\lambda[(m+1)\times\alpha]$ s.t. $X_{m}\subseteq\rho\{a\}$; $U_{m}:=\{[u]_{i}:u\mbox{ is clopen set},i\in\alpha,X_{m}\subseteq [u]_{i}\backslash\bigcup_{j\in\lambda[(m+1)\times\alpha]}u(\frac{j}{i})\}$.
    \item\textbf{IS: } It is easy to see $\|U_{n-1}\|\le\alpha$. Lists $U_{n-1}$ as $\{[u_{j}]_{i_{j}}:j\in\alpha,i_{j}\in\alpha\}$. For $j\in\alpha$, let $j':=\lambda(n,j)$, $u'_{j}:=u_{j}(\frac{j'}{i_{j}})$. Then for any $j_{1},...,j_{k}\in\alpha$,
    
    \begin{tabular}{rl}
         &$a\in X_{n-1}$\\
         $\Leftrightarrow$&$a\in X_{n-1}\cap[u_{j_{1}}]_{i_{j_{1}}}$\\
         $\Leftrightarrow$&$a\in X_{n-1}$ and there is $b\sim_{j_{1}'}a$ s.t. $b\in u_{j_{1}}(\frac{j_{1}'}{i_{j_{1}}})$\\
         $\Leftrightarrow$&$a\in X_{n-1}$ and there is $b\sim_{j_{1}'}a$ s.t. $b\in X_{n-1}\cap u_{j_{1}}(\frac{j_{1}'}{i_{j_{1}}})$\\
         $\Leftrightarrow$& there is $b\sim_{j_{1}'}a$ s.t. $b\in X_{n-1}\cap u_{j_{1}}(\frac{j_{1}'}{i_{j_{1}}})$\\
         $\Leftrightarrow$&$a\in[X_{n-1}\cap u'_{j_{1}}]_{j_{1}'}$\\
         &$\dots\dots$\\
         $\Leftrightarrow$&$a\in[...[X_{n-1}\cap u'_{j_{1}}\cap...\cap u'_{j_{k}}]_{j_{1}'}...]_{j_{k}'}$
    \end{tabular}
    
    Since $X_{n-1}\neq\emptyset$, we know $[...[X_{n-1}\cap u'_{j_{1}}\cap...\cap u'_{j_{k}}]_{j_{1}'}...]_{j_{k}'}\neq\emptyset$, then $X_{n-1}\cap u'_{j_{1}}\cap...\cap u'_{j_{k}}\neq\emptyset$. By compactness, $X'_{n}:=X_{n-1}\cap\bigcap_{j\in\alpha}u_{j}(\frac{j'}{i_{j}})\neq\emptyset$ and easy to see $\Delta(X'_{n})\subseteq\lambda[(n+1)\times\alpha]$. Let $X_{n}$ be complete closed set in $X'_{n}$ of dimension $\lambda[(n+1)\times\alpha]$. $U_{n}:=\{[u]_{i}:u\mbox{ is clopen set},i\in\alpha,X_{n}\subseteq [u]_{i}\backslash\bigcup_{j\in\lambda[(n+1)\times\alpha]}u(\frac{j}{i})\}$.
\end{itemize}

Let $X:=\bigcap_{n\in\omega} X_{n}$, then by compactness and $T_{2}$ property, we know $X$ is a singleton. Let the point in $X$ be $b$, we know for any $[u]_{j}\ni b$, there is a $j\in\alpha$ s.t. $u(\frac{i}{j})\ni b$, Hence $b$ is a model point. Since $b\in X\subseteq\rho\{a\}$, we know $a\prec b$. Then the $\emptyset$-dimensional complete closed set containing $b$ is the same as the $\emptyset$-dimensional complete closed set containing $a$. Hence $b\in v$.

If the diagonal family has no finite subset covering $v$, then by compactness, $v\cap\bigcap\{-D_{ij}:i,j\in\alpha\}\neq\emptyset$. Choosing an arbitrary point $b$ in it, let $\rho,\rho'$ be defined as $\rho(i)=2i,\rho'(i)=2i+1$, it is easy to see $\rho\{a\}\cap\rho'\{b\}\neq\emptyset$. For any point $a'$ in it, let $a'$ replace $a$ in the BS of induction. Obviously, the last constructed model point takes $b$ as a factor and therefore is a big model point.
\end{proof}

The following theorem implies the compactness theorem of first-order logic.
\begin{theorem}[Compactness theorem of first-order logic]\label{cop}$ $

For any first-order language $\mathcal{L}$ and any $\mathcal{L}$-theory $T$, $\mathcal{C}^{T}$ is compact and therefore is an FOL space.
\end{theorem}
\begin{proof}
By Theorem \ref{alexpan}, there is a basis-preserving C-injection $f:\mathcal{C}^{T}\rightarrow(\mathcal{C}^{T})^{\omega}$. We know $(\mathcal{C}^{T})^{\omega}$ is an FOL space. Now we just need to prove $f$ is a surjection which, by Lemma \ref{bh}, means $\mathcal{C}^{T}$ is S-homemorphic to $(\mathcal{C}^{T})^{\omega}$.

Let $l:\Lambda_{\mathcal{L}}\rightarrow\tau^{X}$\footnote{$\tau^{X}$ is the topology of $(\mathcal{C}^{T})^{\omega}$.} be $\phi\mapsto X_{(1,l_{T}(\phi))}$, we know $f^{-1}[l(\phi)]=l_{T}(\phi)$. From the structure of $\omega$-expansion space, it is easy to see that $l$ is an $\mathcal{L}$-formation. Hence $f:(\mathcal{C}^{T},l_{T})\rightarrow((\mathcal{C}^{T})^{\omega},l)$.

For any $x\in(\mathcal{C}^{T})^{\omega}$, let the number of clopen sets containing $x$ be $\alpha$, we know $\alpha\ge\omega$. Denote $\mathcal{C}':=((\mathcal{C}^{T})^{\omega})^{\alpha}$, $g:\mathcal{C}'\rightarrow(\mathcal{C}^{T})^{\omega}$ be expansion mapping. By Proposition \ref{ltl}, there is a unique $\mathcal{L}$-formation $l'$ on $\mathcal{C}'$ s.t. $g:(\mathcal{C}',l')\rightarrow((\mathcal{C}^{T})^{\omega},l)$.

Choosing an arbitrary $z\in g^{-1}[x]$, the number of clopen sets in $\mathcal{C}'$ containing $z$ is $\alpha\times\alpha=\alpha$, then by Lemma \ref{mex}, there is a $\rho:\alpha^{\alpha}$ and a model point $y\in z|_{\emptyset}$ s.t. $\rho^{-1}y=z$. By Lemma \ref{mtm}, there is a $\mathcal{L}$-structure $\mathfrak{A}$ and a basis-preserving C-mapping $h:(\mathcal{C}^{\mathfrak{A}}_{\alpha},l_{\mathfrak{A}})\rightarrow(\mathcal{C}',l')$ s.t. $\|\mathfrak{A}\|\le\alpha$ and $z\in ran(h)$ i.e. there is an $a\in A^{\alpha}$, $h(a)=z$.

Clearly $a\in\bigcap\{A_{\phi}:x\in l(\phi)\}$, which means that $\{\phi\in\Lambda_{\mathcal{L}}:x\in l(\phi)\}$ is satisfied by the assignment represented by $a$. Then there exists a point in $\bigcap\{S_{\phi}:x\in l(\phi)\}$ which is mapped to $x$. By the arbitrariness of $x$, we know that $f$ is a surjection.
\end{proof}\label{4.2}
\subsection{$\alpha$-model space}
Theorem \ref{cop} tells us that the model space of a first-order logic theory must be an FOL space. In turn, for any FOL space $\mathcal{C}$, it is easy to see that we can take a formation on it (e.g., take a language $\mathcal{L}$ such that for any $n$, $\mathcal{L}$ contains $|\tau^{\mathcal{C}}|$ $n$-ary predicates, and then construct the formation such that each clopen set $u$ in $\mathcal{C}$ with $\Delta(u)=n$ is mapped by an $n$-ary atomic formula) and thus determine a first-order logic theory according to the following theorem.

\begin{theorem}
Let $\mathcal{C}_{\alpha}$ be an FOL space, $\mathcal{L}$ be first-order language, $l$ be a $\mathcal{L}$-formation on $\mathcal{C}_{\alpha}$, $T:=\{\phi\in\Lambda_{\mathcal{L}}:\phi\mbox{ is a sentence},l(\phi)=|\mathcal{C}|\}$, then $T$ is an $\mathcal{L}$-theory and there is an S-homemorphism $f:(\mathcal{C},l)\cong((\mathcal{C}^{T})^{\alpha},l^{\alpha}_{T})$ (here $l^{\alpha}_{T}$ is the $\mathcal{L}$-formation on $(\mathcal{C}^{T})^{\alpha}$ induced by $l_{T}$).
\end{theorem}
\begin{proof}$ $
\begin{itemize}
    \item\textbf{$T$ is a theory}
    
    It suffices to show that $T$ is closed under consequence.
    
    For any sentence $\phi\in\Lambda_{\mathcal{L}}$ s.t. $T\vDash\phi$, let $b\in|\mathcal{C}|$ be an arbitrary model point. By Lemma \ref{mtm}, there is an $\mathcal{L}$-structure $\mathfrak{A}$, a basis-preserving C-mapping $f:(\mathcal{C}^{\mathfrak{A}}_{\alpha},l_{\mathfrak{A}})\rightarrow(\mathcal{C},l)$ and a domain point $b'\in f^{-1}[b]$, then $b'\in \bigcap_{\psi\in T}l_{\mathfrak{A}}(\psi)$ i.e. the assignment represented by $b'$ realizes $T$, then the assignment represented by $b'$ satisfies $\phi$, then $b\in l(\phi)$. Since  $\phi$ is a sentence, we know the dimension of $l(\phi)$ is $\emptyset$, then for any $\rho:\alpha^{\alpha}$, $\rho l(\phi)=l(\phi)$, then the factor set of $b$ is subset of $l(\phi)$. By Lemma \ref{mex}, $|\mathcal{C}|=\{a\in|\mathcal{C}|:\mbox{ there is a model point }b\mbox{ }s.t.\mbox{ }a\prec b\}$, then from the arbitrariness of $b$ in the above proof, we know $|\mathcal{C}|\subseteq l(\phi)$ and therefore $\phi\in T$.
    \item\textbf{Constructing f as a bijection}
    
    For any $\phi,\psi\in\Lambda_{\mathcal{L}_{\alpha}}$,
    
    \begin{tabular}{rl}
         &$l(\phi)\cap l(\psi)\neq\emptyset$\\
         $\Leftrightarrow$&$l(\phi\wedge\psi)\neq\emptyset$\\
         $\Leftrightarrow$&$[...[l(\phi\wedge\psi)]_{i_{1}}...]_{i_{k}}\neq\emptyset$ (here $\{i_{1},...,i_{k}\}=\Delta(l(\phi\wedge\psi))$)\\
         $\Leftrightarrow$&$-[...[-l(\neg(\phi\wedge\psi))]_{i_{1}}...]_{i_{k}}\neq|\mathcal{C}|$\\
         $\Leftrightarrow$&$l(\forall v_{i_{1}}...v_{i_{k}}\neg(\phi\wedge\psi))\neq|\mathcal{C}|$\\
         $\Leftrightarrow$&$l(\forall v_{1}...v_{k}\neg(\phi\wedge\psi)(v_{1}/v_{i_{1}})...(v_{k}/v_{i_{k}}))\neq|\mathcal{C}|$\\
         $\Leftrightarrow$&$\forall v_{1}...v_{k}\neg(\phi\wedge\psi)(v_{1}/v_{i_{1}})...(v_{k}/v_{i_{k}})\not\in T$\\
         $\Leftrightarrow$&$l_{T}^{\alpha}(\forall v_{1}...v_{k}\neg(\phi\wedge\psi)(v_{1}/v_{i_{1}})...(v_{k}/v_{i_{k}}))\neq|(\mathcal{C}^{T})^{\alpha}|$\\
         &$\dots\dots$\\
         $\Leftrightarrow$&$l_{T}^{\alpha}(\phi)\cap l_{T}^{\alpha}(\psi)\neq\emptyset$
    \end{tabular}
    
    Hence for any finite $\Phi\subseteq\Lambda_{\mathcal{L}_{\alpha}}$, $\bigcap_{\phi\in\Phi}l(\phi)\neq\emptyset\Leftrightarrow\bigcap_{\phi\in\Phi}l_{T}^{\alpha}(\phi)\neq\emptyset$, then by compactness, for any $\Phi\subseteq\Lambda_{\mathcal{L}_{\alpha}}$, $\bigcap_{\phi\in\Phi}l(\phi)\neq\emptyset\Leftrightarrow\bigcap_{\phi\in\Phi}l_{T}^{\alpha}(\phi)\neq\emptyset$. Hence, defining $f$ as $f(a):=a'\in\bigcap\{l_{T}^{\alpha}(\phi):a\in l(\phi)\}$ (by $T_{2}$ property, $a'$ is unique), it is easy to see $f$ is a bijection.
    
    \item\textbf{$f$ is S-homemophism}
    
    By definition of $f$, it is easy to see for any $\phi$, $l(\phi)=f^{-1}[l_{\mathfrak{A}}^{\alpha}(\phi)]$, then for any $i\in\alpha$, $a\sim_{i}b\in|\mathcal{C}|$, we know $\{u\in\tau^{\mathcal{C}}:u$ is an $i$-saturated clopen neighborhood of $a\}=\{u\in\tau^{\mathcal{C}}:u$ is an $i$-saturated clopen neighborhood of $b\}$, then $\{\phi\in\Lambda_{\mathcal{L}_{\alpha}}:a\in l(\phi)$ and $v_{i}$ does not occur freely in $\phi\}=\{\phi\in\Lambda_{\mathcal{L}_{\alpha}}:b\in l(\phi)$ and $v_{i}$ does not occur freely in $\phi\}$, then $\{\phi\in\Lambda_{\mathcal{L}_{\alpha}}:f(a)\in l_{\mathfrak{A}}^{\alpha}(\phi)$ and $v_{i}$ does not occur freely in $\phi\}=\{\phi\in\Lambda_{\mathcal{L}_{\alpha}}:f(b)\in l_{\mathfrak{A}}^{\alpha}(\phi)$ and $v_{i}$ does not occur freely in $\phi\}$, then  $\{u:u$ is a $i$-saturated clopen neighborhood of $f(a)\}=\{u:u$ is a $i$-saturated clopen neighborhood of $f(b)\}$, then by Proposition \ref{cstc}, $f(a)\sim_{i}f(b)$, then by Proposition \ref{blbk}, $f$ is a basis-preserving C-mapping, then by Lemma \ref{bh}, $f$ is an S-homemorphism.
\end{itemize}
\end{proof}
This theorem implies that the class of all FOL spaces is exactly the class of all S-homeomorphic spaces of model spaces and their expansion spaces. Based on the discussion above, we can have a more general definition of model space.
\begin{definition}
Let $\mathcal{C}_{\alpha}$ be an FOL space, $\mathcal{L}$ be a first-order language, $l$ be an $\mathcal{L}$-formation on $\mathcal{C}_{\alpha}$, $T$ be an $\mathcal{L}$-theory, if $T=\{\phi\in\Lambda_{\mathcal{L}}:l(\phi)=|\mathcal{C}|\}$, then we call $\mathcal{C}_{\alpha}$ an \textbf{$\alpha$-dimensional model space} of $T$ and denote it as $\mathcal{C}^{T}_{\alpha}$, call $l$ a \textbf{$T$-formation} on it and denote it as $l^{\alpha}_{T}$.
\end{definition}
\label{4.3}
\subsection{The Representation of Semantics of First-Order Logic }
Lemma \ref{mtm} tells us that a model of a theory $T$ of cardinality $\kappa$ can be represented as a cylindric space that can be mapped to the $\kappa$-model space of $T$ by a basis-preserving C-mapping and therefore as a model point in this FOL space. We show this correspondence between model points and models in more detail in the first theorem of this section. Based on this, we will see that an elementary embedding between first-order models can be represented as a basis-preserving C-mapping between the corresponding topologization spaces and, therefore, as a factoring relationship between two model points in a FOL space. In this way, a systematic  topological representation is presented.

\begin{theorem}\label{su}
Let $\mathcal{L}$ be a first-order language, $T$ be an $\mathcal{L}$-theory,
\begin{enumerate}
    \item For any $T$-model $\mathfrak{A}$, any infinite ordinal $\alpha\ge\beta$, there exists a unique basis-preserving C-mapping $f:(\mathcal{C}^{\mathfrak{A}}_{\alpha},l_{\mathfrak{A}})\rightarrow(\mathcal{C}^{T}_{\beta},l_{T}^{\beta})$;
    \item For any $T$-model $\mathfrak{A}$, any infinite ordinal $\alpha$ and basis-preserving C-mapping $f:(\mathcal{C}^{\mathfrak{A}}_{\alpha},l_{\mathfrak{A}})\rightarrow(\mathcal{C}^{T}_{\alpha},l_{T}^{\alpha})$, let $U$ be the non-empty set of all domain points in $\mathcal{C}^{\mathfrak{A}}_{\alpha}$, then $f[U]$ is an $\asymp$-equivalence class and all the points in it are model point and if $\alpha=\|\mathfrak{A}\|$, then these model points are big.
    \item For any infinite ordinal $\alpha$, model point $a\in|\mathcal{C}^{T}_{\alpha}|$, up to S-homeomorphism, there exists unique $T$-model $\mathfrak{A}$ s.t. for the basis-preserving C-mapping $f:(\mathcal{C}^{\mathfrak{A}}_{\alpha},l_{\mathfrak{A}})\rightarrow(\mathcal{C}^{T}_{\alpha},l_{T}^{\alpha})$, there is domain point $b\in f^{-1}[a]$. The cardinality of $|\mathfrak{A}|$ is not greater than $\alpha$. We say $a$ \textbf{represents} $\mathfrak{A}$.
\end{enumerate}
\end{theorem}
\begin{proof}$ $

\begin{enumerate}
    \item \begin{itemize}
        \item\textbf{Existence: }For any $a\in|\mathcal{C}^{\mathfrak{A}}_{\alpha}|$, we know $\{l_{\mathfrak{A}}(\phi):\phi\in\Lambda_{\mathcal{L}_{\alpha\cap\beta}},a\in l_{\mathfrak{A}}(\phi)\}$ is closed under taking finite intersection and does not contain $\emptyset$, then it is easy to see $X:=\{l_{T}(\phi):\phi\in\Lambda_{\mathcal{L}_{\alpha\cap\beta}},a\in l_{\mathfrak{A}}(\phi)\}$ has the same property, then by compactness, $\bigcap X\neq\emptyset$. Let $a'\in\bigcap X\cap\bigcap_{i\in\beta\backslash\alpha}D^{T}_{0i}$ (by $T_{2}$ property and property of diagonal family, this is a singleton), we define $f:|\mathcal{C}^{\mathfrak{A}}_{\alpha}|\rightarrow|\mathcal{C}^{T}_{\beta}|$ as $f(a)=a'$, then for any $\psi\in\Lambda_{\mathcal{L}_{\alpha\cap\beta}}$, $b\in|\mathcal{C}^{\mathfrak{A}}_{\alpha}|$, $b\in l_{\mathfrak{A}}(\psi)\Leftrightarrow \psi\in \{\phi\in\Lambda_{\mathcal{L}_{\alpha\cap\beta}}:b\in l_{\mathfrak{A}}(\phi)\}\Leftrightarrow f(b)\in l_{T}^{\beta}(\psi)\Leftrightarrow b\in f^{-1}[l_{T}^{\beta}(\psi)]$, then $l_{\mathfrak{A}}(\psi)=f^{-1}[l_{T}^{\beta}(\psi)]$.
        
        As with the mapping in the proof of Lemma \ref{mtm}, it is easy to verify that $f$ is a structural homomorphism. Then by Proposition \ref{blbk}, $f:(\mathcal{C}^{\mathfrak{A}}_{\alpha},l_{\mathfrak{A}})\rightarrow(\mathcal{C}^{T}_{\beta},l_{T}^{\beta})$ is a basis-preserving C-mapping.
        \item\textbf{Uniqueness: }The uniqueness can be obtained immediately by Proposition \ref{ubk}.
    \end{itemize}
    \item For $a,b\in|\mathcal{C}^{\mathfrak{A}}_{\alpha}|$ s.t. $a$ is a domain point, define $\rho:\alpha^{\alpha}$ as $\rho(i):=min\{j:b(i)=a(j)\}$, we know for any $\phi\in\Lambda_{\mathcal{L}_{\alpha}}$, $b\in l_{\mathfrak{A}}(\phi)\Leftrightarrow a\in \rho l_{\mathfrak{A}}(\phi)=l_{\mathfrak{A}}((\overline{v_{\rho(i)}}/\overline{v_{i}}))$, then
    
    \begin{tabular}{rl}
         &$b$ is a domain point\\
         $\Leftrightarrow$&$b$ is a domain point and $a\in\bigcap\{l_{\mathfrak{A}}(\phi(\overline{v_{\rho(i)}}/\overline{v_{i}})):b\in l_{\mathfrak{A}}(\phi)\}$\\
         $\Leftrightarrow$&$a\in\bigcap\{l_{\mathfrak{A}}(\phi(\overline{v_{\rho(i)}}/\overline{v_{i}})):b\in l_{\mathfrak{A}}(\phi)\}$ and for any $i\not\in ran(\rho)$,\\&there is an $i'\in ran(\rho)$ s.t. $a\in l_{\mathfrak{A}}(v_{i}=v_{i'})$\\
         $\Leftrightarrow$&$f(a)\in\bigcap\{l_{T}^{\alpha}(\phi(\overline{v_{\rho(i)}}/\overline{v_{i}})):f(b)\in l_{T}^{\alpha}(\phi)\}$ and for any $i\not\in ran(\rho)$,\\&there is an $i'\in ran(\rho)$ s.t. $f(a)\in l_{T}^{\alpha}(v_{i}=v_{i'})$\\
         $\Leftrightarrow$&$f(a)\in\bigcap\{\rho(l_{T}^{\alpha}(\phi)):f(b)\in l_{T}^{\alpha}(\phi)\}$ and for any $i\not\in ran(\rho)$,\\&there is an $i'\in ran(\rho)$ s.t. $f(a)\in l_{T}^{\alpha}(v_{i}=v_{i'})$\\
         $\Leftrightarrow$&$\rho:f(b)\prec f(a)$ and for any $i\not\in ran(\rho)$,\\&there is an $i'\in ran(\rho)$ s.t. $f(a)\in D^{T}_{ii'}$ ($ran(l_{T}^{\alpha})$ is cylindrical basis)\\
         $\Leftrightarrow$&$f(b)\asymp f(a)$
    \end{tabular}
    
    Hence, the image of the non-empty set of all domain points under $f$ is an $\asymp$-equivalence class. By Lemma \ref{mtm}, all points in this class are model points.
    
    If $\alpha=\|\mathfrak{A}\|$, then there is a domain point $a\in A^{\alpha}$ s.t. for any $i\neq j,a(i)\neq a(j)$, then $f(a)\not\in\bigcup\mathcal{D}^{T}$, then $f(a)$ is a big model point. For other $b\in U$, we know $f(a)\prec b$ and therefore $b$ is also big.
    
    \item Existence is exactly Lemma \ref{mtm}. To prove uniqueness, consider two basis-preserving C-mapping $f:(\mathcal{C}^{\mathfrak{A}}_{\alpha},l_{\mathfrak{A}})\rightarrow(\mathcal{C}^{T}_{\alpha},l_{T}^{\alpha})$ and $f:(\mathcal{C}^{\mathfrak{A}'}_{\alpha},l_{\mathfrak{A}'})\rightarrow(\mathcal{C}^{T}_{\alpha},l_{T}^{\alpha})$ s.t. there are domain points $b\in f^{-1}[a],b'\in f'^{-1}[a]$. for any $\rho:\alpha^{\alpha}$, any $c\in A^{\alpha}$, we define $\rho^{+} c\in A^{\alpha}$ as $\rho^{+} c(i)=c(\rho(i))$. it is easy to see $A^{\alpha}=\{\rho^{+}b:\rho:\alpha^{\alpha}\}$ and similarly, $A'^{\alpha}=\{\rho^{+}b':\rho:\alpha^{\alpha}\}$. Define $h:A\rightarrow A'$ as $\rho^{+}b\mapsto\rho^{+}b'$, we show $h:(\mathcal{C}^{\mathfrak{A}}_{\alpha},l_{\mathfrak{A}})\rightarrow(\mathcal{C}^{\mathfrak{A}'}_{\alpha},l_{\mathfrak{A}'})$ is S-homeomorphism:
    \begin{itemize}
        \item $h$ is well-defined and is an injection:
        
        For any $\rho,\rho':\alpha^{\alpha}$,
        
        \begin{tabular}{rl}
             &$\rho^{+} b=\rho'^{+} b$\\
             $\Leftrightarrow$& For any $i$, $b(\rho(i))=b(\rho'(i))$\\
             $\Leftrightarrow$& For any $i$, $b\in l_{\mathfrak{A}}(v_{\rho(i)}=v_{\rho'(i)})$\\
             $\Leftrightarrow$& For any $i$, $b'\in l_{\mathfrak{A'}}(v_{\rho(i)}=v_{\rho'(i)})$\\
             &$\dots\dots$\\
             $\Leftrightarrow$&$h(\rho^{+} b)=\rho^{+} b'=\rho'^{+} b'=h(\rho'^{+} b')$
        \end{tabular}
        \item $h$ is a surjection: $A'^{\alpha}=\{\rho^{+}b':\rho:\alpha^{\alpha}\}=ran(h)$.
        
        \item $h$ is a basis-preserving C-mapping:
        
        For any $i\in\alpha$ and any $a,c\in A^{\alpha}$ s.t. $a\sim_{i}c$, we know $a(j)=b(j)$ for any $j\neq i$, then there are $\rho,\rho':\alpha^{\alpha}$ s.t. $\rho(j)=\rho'(j)$ for any  $j\neq i$ and $a=\rho^{+}b,c=\rho'^{+}b$, then $h(a)=\rho^{+}b',h(c)=\rho'^{+}b'$, then $h(a)(j)=h(c)(j)$ for any $j\neq i$, then $h(a)\sim_{i}h(c)$.
        
        For any $\phi\in\Lambda_{\mathcal{L}_{\alpha}}$,
        
        \begin{tabular}{rl}
             &$\rho^{+} b\in l_{\mathfrak{A}}(\phi)$\\
             $\Leftrightarrow$&$b\in l_{\mathfrak{A}}(\phi)(\frac{\rho(i_{1})}{i_{1}})...(\frac{\rho(i_{k})}{i_{k}})$ ($\Delta(l(\phi))=\{i_{1},...,i_{k}\}$)\\
             $\Leftrightarrow$&$b\in l_{\mathfrak{A}}(\phi(v_{\rho(i_{1})}/v_{i_{1}})...(v_{\rho(i_{k})}/v_{i_{k}}))$\\
             $\Leftrightarrow$&$b'\in l_{\mathfrak{A'}}(\phi(v_{\rho(i_{1})}/v_{i_{1}})...(v_{\rho(i_{k})}/v_{i_{k}}))$\\
             &$\dots\dots$\\
             $\Leftrightarrow$&$\rho^{+} b'\in l_{\mathfrak{A'}}(\phi)$
        \end{tabular}
        
        By Proposition \ref{blbk}, $h$ is a basis-preserving C-mapping.
    \end{itemize}
    By Proposition \ref{bh}, $h$ is an S-homeomorphism. By Proposition \ref{mp}, $f'\circ h$ is a basis-preserving C-mapping and obviously maps $l_{\mathfrak{A}}$ to $l_{T}^{\alpha}$, then by Proposition $\ref{ubk}$, $f'\circ h=f$.
\end{enumerate}
\end{proof}
\begin{theorem}\label{map}
Let $\mathcal{L}$ be a first-order language, $\mathfrak{A},\mathfrak{B}$ be $\mathcal{L}$-structures, the following statements are equivalent:
\begin{enumerate}
    \item There is an elementary embedding $f:\mathfrak{A}\rightarrow\mathfrak{B}$;
    \item For any infinite ordinal $\alpha$, there is a basis-preserving C-mapping $h:(\mathcal{C}^{\mathfrak{A}}_{\alpha},l_{\mathfrak{A}})\rightarrow(\mathcal{C}^{\mathfrak{B}}_{\alpha},l_{\mathfrak{B}})$;
    \item For $T\subseteq Th(\mathfrak{A})\cap Th(\mathfrak{B})$, infinite ordinal $\alpha\ge\|\mathfrak{A}\|\cup\|\mathfrak{B}\|$ and model points $a,b\in|\mathcal{C}^{T}_{\alpha}|$ s.t. $a$ represents $\mathfrak{A}$, $b$ represents $\mathfrak{B}$, we have $a\prec b$.
\end{enumerate}
\end{theorem}
\begin{proof}$ $

\begin{itemize}
    \item $1\Rightarrow 2:$ Define $h:A^{\alpha}\rightarrow B^{\alpha}$ as $h(a)(i):=f(a(i))$, then clearly the condition of structure homomorphism holds. For any $\phi\in\Lambda_{\mathcal{L}_{\alpha}}$, $c\in A^{\alpha}$,

\begin{tabular}{rl}
     &$c\in l_{\mathfrak{A}}(\phi)$\\
     $\Leftrightarrow$&$\langle c(i_{1}),...,c(i_{k})\rangle\in \phi(\mathfrak{A})$ ($v_{i_{1}},...,v_{i_{k}}$ are the free variables in $\phi$)\\
     $\Leftrightarrow$&$\langle f(c(i_{1})),...,f(c(i_{k}))\rangle\in \phi(\mathfrak{B})$\\
     $\Leftrightarrow$&$h(c)\in l_{\mathfrak{B}}(\phi)$\\
     $\Leftrightarrow$&$c\in h^{-1}[l_{\mathfrak{B}}(\phi)]$
\end{tabular}

By Proposition \ref{blbk}, $h$ is a basis-preserving mapping.

\item $2\Rightarrow 3:$ By Theorem \ref{su}.(1), there are unique basis-preserving C-mappings $g:(\mathcal{C}^{\mathfrak{A}}_{\alpha},l_{\mathfrak{A}})\rightarrow(\mathcal{C}^{T}_{\alpha},l_{T}^{\alpha}),g':(\mathcal{C}^{\mathfrak{B}}_{\alpha},l_{\mathfrak{B}})\rightarrow(\mathcal{C}^{T}_{\alpha},l_{T}^{\alpha})$. Since $\alpha\ge\|\mathfrak{A}\|\cup\|\mathfrak{B}\|$, we know both $A^{\alpha}$ and $B^{\alpha}$ contain domain points. By Theorem \ref{su}.(2), there are model points $a,b\in|\mathcal{C}^{T}_{\alpha}|$ representing $\mathfrak{A},\mathfrak{B}$ respectively.

For any domain points $a^{*}\in g^{-1}[a]$,, $b^{*}\in g'^{-1}[b]$ and any $i\in\alpha$, we know there is an $i'$ s.t. $b^{*}(i')=h(a^{*})(i)$. Define $\rho:\alpha^{\alpha}$ as $i\mapsto min\{i':b^{*}(i')=h(a^{*})(i)\}$, then for any $\phi\in\Lambda_{\mathcal{L}_{\alpha}}$ we have:

\begin{tabular}{rl}
     &$a\in l_{T}^{\alpha}(\phi)$\\
     $\Leftrightarrow$&$a^{*}\in l_{\mathfrak{A}}(\phi)$\\
     $\Leftrightarrow$&$h(a^{*})\in l_{\mathfrak{B}}(\phi)$\\
     $\Leftrightarrow$&$\langle h(a^{*})(i_{1}),...,h(a^{*})(i_{k})\rangle\in \phi(\mathfrak{B})$ ($v_{i_{1}},...,v_{i_{k}}$ are free variables in $\phi$)\\
     $\Leftrightarrow$&$\langle b^{*}(\rho(i_{1})),...,b^{*}(\rho(i_{k}))\rangle\in \phi(\mathfrak{B})$\\
     $\Leftrightarrow$&$b^{*}\in l_{\mathfrak{B}}(\phi(v_{\rho(i_{1})}/v_{i_{1}})...(v_{\rho(i_{k})}/v_{i_{k}})))$\\
     $\Leftrightarrow$&$b\in l_{T}^{\alpha}(\phi(v_{\rho(i_{1})}/v_{i_{1}})...(v_{\rho(i_{k})}/v_{i_{k}})))$\\
     $\Leftrightarrow$&$b\in \rho l_{T}^{\alpha}(\phi)$
\end{tabular}

Hence $\rho:a\prec b$.

\item $3\Rightarrow 1:$ For any $a\in|\mathcal{C}^{T}_{\alpha}|,b\in|\mathcal{C}^{T}_{\alpha}|$ representing $\mathfrak{A},\mathfrak{B}$ respectively and any $\rho:a\prec b$, by Theorem \ref{su}, we know there are unique basis-preserving C-mappings $g:(\mathcal{C}^{\mathfrak{A}}_{\alpha},l_{\mathfrak{A}})\rightarrow(\mathcal{C}^{T}_{\alpha},l_{T}^{\alpha}),g':(\mathcal{C}^{\mathfrak{B}}_{\alpha},l_{\mathfrak{B}})\rightarrow(\mathcal{C}^{T}_{\alpha},l_{T}^{\alpha})$ s.t. there are domain points $a^{*}\in g^{-1}[a],b^{*}\in g'^{-1}[b]$, let $I:=\{i\in\alpha:\mbox{ for any }j<i,a^{*}\in-D^{\mathfrak{A}}_{ij}\}$, clearly $A=\{a^{*}(i):i\in I\}$ and for any $i,j\in I$, $i\neq j\Rightarrow a^{*}(i)\neq a^{*}(j)$. We define $B$ in the same way and define $f:A\rightarrow B$ as: for any $i\in I$, $f(a^{*}(i))=b^{*}(\rho(i))$. Then for any $\phi\in\Lambda_{\mathcal{L}_{\alpha}}$, $i_{1},...,i_{k}\in I$,

\begin{tabular}{rl}
     &$\langle a^{*}(i_{1}),...,a^{*}(i_{k})\rangle\in\phi(\mathfrak{A})$\\
     $\Leftrightarrow$&$a^{*}\in l_{\mathfrak{A}}(\phi)$\\
     $\Leftrightarrow$&$a\in l_{T}^{\alpha}(\phi)$\\
     $\Leftrightarrow$&$b\in\rho l_{T}^{\alpha}(\phi)$\\
     $\Leftrightarrow$&$b^{*}\in\rho l_{\mathfrak{B}}(\phi)$\\
     $\Leftrightarrow$&$b^{*}\in l_{\mathfrak{B}}(\phi(v_{\rho(i_{1})}/v_{i_{1}}))...(v_{\rho(i_{k})}/v_{i_{k}}))$\\
     $\Leftrightarrow$&$\langle b^{*}(\rho(i_{1})),...,b^{*}(\rho(i_{k}))\rangle\in\phi(\mathfrak{B})$\\
     $\Leftrightarrow$&$\langle f(a^{*}(i_{1})),...,f(a^{*}(i_{k}))\rangle\in\phi(\mathfrak{B})$
\end{tabular}

Hence $f$ is an elementary embedding.
\end{itemize}
\end{proof}
Based on the proof of Theorem \ref{map}, we have the following corollary:
\begin{corollary}\label{ccc}
Let $\mathcal{L}$ be a first-order language, $\mathfrak{A},\mathfrak{B}$ be $\mathcal{L}$-structure, then the following statements are equivalent:
\begin{enumerate}
    \item $\mathfrak{A}\cong\mathfrak{B}$;
    \item there is an S-homeomorphism $h:(\mathcal{C}^{\mathfrak{A}\kappa},l_{\mathfrak{A}})\cong(\mathcal{C}^{\mathfrak{B}\kappa},l_{\mathfrak{B}})$ s.t. $h$ maps domain points to domain points.
    \item For $T\subseteq Th(\mathfrak{A})\cap Th(\mathfrak{B})$, infinite ordinal $\alpha\ge\|\mathfrak{A}\|\cup\|\mathfrak{B}\|$ and any model points $a,b\in|\mathcal{C}^{T}_{\alpha}|$ s.t. $a$ represents $\mathfrak{A}$, $b$ represents $\mathfrak{B}$, we have $a\asymp b$.
\end{enumerate}
\end{corollary}
\begin{proof}$ $

\begin{itemize}
    \item $1\Rightarrow 2$: Since an elementary embedding must be an injection, from the definition of $h$ in the above proof, it is easy to see that $h$ is injective. If $f$ is an isomorphism, then for any $b\in B^{\alpha}$, $i\in \alpha$, there must exist $x_{i}\in A$ s.t. $b(i)=f(x_{i})$. Let $a\in A^{\alpha}$ be $a(i)=x_{i}$, we know that $h(a)=b$. Therefore, $h$ is a bijection and thus S-homeomorphism. It is also obvious that if $a$ is a domain point, then $h(a)$ is also a domain point.
    \item $2\Rightarrow 3$: We know $h(a^{*})$ is domain point, then from the definition of $\rho$, it is easy to see $B=\{b^{*}(\rho(i)):i\in\alpha\}$, then for any $i\not\in ran(\rho)$, there is an $i'\in ran(\rho)$ s.t. $b^{*}(i)=b^{*}(i')$, then $b^{*}\in l_{\mathfrak{B}}(v_{i}=v_{i'})$, then $b\in D_{ii'}$. By definition, $\rho: a\asymp b$.
    \item $3\Rightarrow 1$: By definittion, for any $i\not\in ran(\rho)$, there is an $i'\in ran(\rho)$ s.t. $b^{*}\in D^{\mathfrak{B}}_{ii'}$ i.e. $b^{*}(i)=b^{*}(i')$, then $B=\{b(i):i\in ran(\rho)\}=ran(f)$, then $f$ is a surjection and therefore is an isomorphism.
\end{itemize}
\end{proof}

\begin{corollary}
For any first-order language $\mathcal{L}$, any $\mathcal{L}$-theory $T$ and any infinite cardinal $\kappa$, let $\mathcal{I}(T,\kappa)$ denote the number of nonisomorphic models of $T$ of cardinality $\kappa$, we have that $\mathcal{I}(T,\kappa)=\|\{[a]_{\asymp}:a$ is big model point in $\mathcal{C}^{T}_{\kappa}\}\|$.
\end{corollary}

\begin{remark}
Combined with their proofs, Theorem \ref{su} and Theorem \ref{map} tell us that for any infinite cardinality $\kappa$, up to isomorphism and relation $\asymp$, all models of a theory of cardinality$\le\kappa$ correspond one-to-one with all model points in the $\kappa$-model space of $T$, and all elementary embeddings between $T$-models correspond one-to-one with all factor relations between model points in the $\kappa$-model space of $T$. It is also easy to see that a subset of a model can be represented as a factor of the model point representing it; a partial elementary embedding can be represented as a partial factor relation. In this way, each basic object in semantics of first-order logic is represented here as a topological object on the FOL space. 

An important similarity to the topological representation of semantics of propositional logic is that, just as the number of points in a Stone space is independent of the specific theory it corresponds to and the way it corresponds to it, in an FOL space, the structure consisting of points and their factor relations is also independent of the specific $\mathcal{L}$-formation. In other words, through the FOL space, we refine the general abstract structure of semantics first-order logic independent from the formal language.
\end{remark}\label{4.4}
\section{Application in model theory}\label{yy}
In Section \ref{4.3}, we showed that the class of FOL spaces is exactly the class of S-homemorphic spaces of $\alpha$-model spaces. Then, the topological representation of semantics of first-order logic via $\alpha$-model spaces naturally leads to a method for discussing model theory using a pure set-topology language. Introducing our topological representation as such a method into model theory has at least two foreseeable benefits:
\begin{itemize}
    \item The topological representation is essentially a homogenization of first-order structures and the formal language that are heterogeneous from each other. Under which some distinct objects become the same thing (e.g., both models and Henkin sets become model points in the FOL space). Thus this treatment can obviously simplify the discussion and even generate new ideas.
    \item In the set-topology language, we do not need to consider the strict requirement of the formal language in the form of writing, which can simplify the lines.
\end{itemize}

In \ref{5.2}, we will look at this new topological method through an example. As mentioned in the introduction, we already have a complete topological proof of Morley's theorem. The example in \ref{5.2} is a critical fragment of this entire proof with some modifications in order to correspond to an intermediate theorem of the classical proof of Morley's theorem. We chose this fragment as the example because, on the one hand, this fragment reflects relatively more features and advantages of the new method;  on the other hand, it follows the classical proof idea of the theorem, thus making it easier to compare it with the classical proof.

Before getting into the example, we first explore the connection between our work and the type space in \ref{5.1}. We know that type space is a classical set-topology notion in model theory. Since many concepts and works in model theory are constructed based on the type space, not only this connection means that our new method can directly inherit many set-topology approaches that are already well established in model theory, but also we need to base on this discussion when translating concepts in the example.
\subsection{Connection with type space}\label{5.1}
Let us start with a basic concept, the formula with parameters. Saying a formula $\phi(\overline{v},\overline{a})$ with parameters $\overline{a}$ can be satisfied on a first-order structure $\mathfrak{A}$ is to say there is an assignment $\sigma$ on $\mathfrak{A}$ containing $\overline{w}\mapsto\overline{a}$ as a part satisfying $\phi(\overline{v},\overline{w})$. Then by topological language, this is to say that the maximal consistent set satisfied by $\sigma$ is in $S_{\phi(\overline{v},\overline{w})}$. In this way, as with the formula without parameters, the satisfiability of a formula with parameters can be expressed as the nonemptyness of a clopen set. Thus, the realizability of a type with parameters can be expressed as the nonemptyness of the corresponding closed set. Based on this discussion, the following theorem is more intuitive.

\begin{theorem}\label{ttm}
Let $\mathcal{L}$ be a language, $T$ be an $\mathfrak{L}$-theory, then for any $T$-model $\mathfrak{A}$, any nonempty $B\subseteq A$, any $n\in\mathbb{N}^{+}$ and any $\kappa>\|B\|+\aleph_{0}$, the type space $S^{\mathfrak{A}}_{n}(B)$ is homeomorphic to a $\kappa\backslash n$-dimensional complete closed set on $\mathcal{C}_{T}^{\kappa}$. Vice versa.
\end{theorem}
\begin{proof}$ $

\begin{itemize}
    \item For the type space $S^{\mathfrak{A}}_{n}(B)$, let $f:\{v_{i}:i\in\kappa\backslash n\}\rightarrow B$ be a surjection. We define  $g:S^{\mathfrak{A}}_{n}(B)\rightarrow\mathcal{C}_{T}^{\kappa}$ as $p\mapsto\{\phi(\overline{v},\overline{w}):\phi(\overline{v},f(\overline{w}))\in p\}$.
    \begin{itemize}
        \item \textbf{$ran(g)$ is a $\kappa\backslash n$-dimensional complete closed set:}
        
        For any $i\in n$, any $x\in ran(g)$ and any $y\sim_{i}x$, we know for any formula $\phi(\overline{w})\in y$ where $\overline{w}\in\kappa\backslash n$, 
        $S_{\phi}$ is $\sim_{i}$-saturated and then $\phi\in x$, then for $p\in g^{-1}[x]$, $\phi(f(\overline{w}))\in p\subseteq\bigcup S^{\mathfrak{A}}_{n}(B)$ is a sentence. This means $y':=\{\phi(\overline{v},f(\overline{w})):\phi(\overline{v},\overline{w})\in y\}$ is in $S^{\mathfrak{A}}_{n}(B)$. Clearly $g(y')=y$. Hence, $ran(g)$ is $\sim_{i}$-saturated.
        
        For any formula $\phi(\overline{w})$ s.t. $\Delta(S_{\phi})\cap n=\emptyset$, and $S_{\phi}\cap ran(g)=\neq\emptyset$, let $x\in S_{\phi}\cap ran(g)$. We know $\phi(f(\overline{w}))\in p$ for any $p\in g^{-1}[x]$, this means $\phi(f(\overline{w}))\in p$ for any $p\in S^{\mathfrak{A}}_{n}(B)$. Then for any $y\in ran(g)$, $\phi\in y$ i.e. $y\in S_{\phi}$.
        \item \textbf{$g$ is an injection:}
        
        For any $p\neq q\in S^{\mathfrak{A}}_{n}(B)$, we know there is $q\not\ni\phi(\overline{v},\overline{a})\in p$, then $g(q)\not\in S_{\phi(\overline{v},\overline{w})}\ni g(p)$ for any $\overline{w}\in f^{-1}[\overline{a}]$. Then $g(p)\neq g(q)$. Hence $g$ is injective.
        \item \textbf{$g$ is continuous:}
        
        For any $\phi\in\Lambda_{\mathcal{L}_{\kappa}}$, we denote $\phi$ as $\phi(\overline{v},\overline{w})$ where $\overline{v}\in\{v_{i}:i\in n\}, \overline{w}\in\{v_{i}:i\not\in n\}$. $g^{-1}[\phi]=\{p\in S^{\mathfrak{A}}_{n}(B):\phi\in g(p)\}=\{p\in S^{\mathfrak{A}}_{n}(B):\phi(\overline{v},f(\overline{w}))\in p\}=[\phi(\overline{v},f(\overline{w}))]$ is a clopen set on $S^{\mathfrak{A}}_{n}(B)$.
        \item \textbf{$g^{-1}:ran(g)\rightarrow S^{\mathfrak{A}}_{n}(B)$ is continuous:}
        
        For any $\phi(\overline{v},\overline{a})\in\bigcup S^{\mathfrak{A}}_{n}(B)$, $g[\phi(\overline{v},\overline{a})]=\{x\in ran(g):\phi(\overline{v},\overline{w})\in x$ for $\overline{w}\in f^{-1}[\overline{a}]\}=ran(g)\cap\bigcap_{\overline{w}\in f^{-1}[\overline{a}]}S_{\phi(\overline{v},\overline{w})}$ is a closed set on $ran(g)$. Then for any closed set $s$ on $S^{\mathfrak{A}}_{n}(B)$, $(g^{-1})^{-1}[s]=g[s]$ is closed. Hence $g^{-1}$ is continuous.
    \end{itemize}
    \item For any $\kappa\backslash n$-dimensional complete closed set $s$ on $\mathcal{C}_{T}^{\kappa}$, we know there is a model point $x\in\rho s$ for some $\rho$. Let $\mathfrak{A}$ be a $T$-model represented by $x$ and $x'$ be a domain point in the inverse image of $x$ in $\mathcal{C}^{\mathfrak{A}}_{\kappa}$. We define $B=\{x'(i):i\in\rho[\Delta(s)]\}$, define $g:S^{\mathfrak{A}}_{n}(B)\rightarrow\mathcal{C}_{T}^{\kappa}$ as $p\mapsto\{\phi(\overline{v},v_{i_{1}},...,v_{i_{m}}):\phi(\overline{v},x'(\rho(i_{1})),...,x'(\rho(i_{m})))\in p\}$. Then for any $y\in\mathcal{C}_{T}^{\kappa}$,
    \begin{tabular}{rl}
         &$y\in ran(g)$\\
         $\Leftrightarrow$& for formula $\phi(v_{i_{1}},...,v_{i_{m}})$ in $y$ where $i_{1},...,i_{m}\in\{v_{i}:i\not\in n\}$, $\mathfrak{A},x'\vDash\phi(v_{\rho(i_{1})},...,v_{\rho(i_{m})})$\\
         $\Leftrightarrow$& for formula $\phi(v_{i_{1}},...,v_{i_{m}})$ in $y$ where $i_{1},...,i_{m}\in\{v_{i}:i\not\in n\}$, $\phi(v_{\rho(i_{1})},...,v_{\rho(i_{m})})\in x$\\
         $\Leftrightarrow$& for formula $\phi(v_{i_{1}},...,v_{i_{m}})$ in $y$ where $i_{1},...,i_{m}\in\{v_{i}:i\not\in n\}$, $\phi(v_{i_{1}},...,v_{i_{m}})\in y'$ for any $y'\in s$\\
         $\Leftrightarrow$& $y\in s$.
    \end{tabular}
    Hence, $ran(g)=s$.
    
    Then by a similar way as the first direction, we can prove $g:S^{\mathfrak{A}}_{n}(B)\cong s$.
\end{itemize}
\end{proof}
\begin{remark}
The above theorem corresponds all cofinite-dimensional complete closed sets on $\kappa$-dimensional spaces to all type spaces over a nonempty set not more than $\kappa$. In turn, an FOL space can be seen as the merging of a family of type spaces. Then, the existing type space-based approach in model theory can be immediately transferred to the FOL space. In addition, Since the FOL space acts as a merge of type spaces, some connections between different type spaces can also be formalized as connections on the FOL space. For example, for models $\mathfrak{A}$, $\mathfrak{A'}$, $B\subseteq|\mathfrak{A}|$ and $B'\subseteq|\mathfrak{A'}|$ s.t. there is partial elementary embedding $f:B\rightarrow B'$, the relation between $S^{\mathfrak{A}}_{n}(B)$ and $S^{\mathfrak{A'}}_{m}(B')$ can be formalized by a permutation relation on the FOL space. That is, for the closed sets $s$ representing $S^{\mathfrak{A}}_{n}(B)$ and $t$ representing $S^{\mathfrak{A'}}_{m}(B')$, $t\subseteq\rho s$ for some mapping $\rho$. On the other hand, for a model $\mathfrak{A}$ and a $B\subseteq|\mathfrak{A}|$, it is easy to see that for different $m,n\in\mathbb{N}$, the set $s$ representing $S^{\mathfrak{A}}_{m}(B)$ can become a set representing $S^{\mathfrak{A}}_{n}(B)$ by taking a permutation. Then, when we need to consider $S^{\mathfrak{A}}_{n}(B)$ for all $n\in\mathbb{N}$ at the same time, we can combine all these $S^{\mathfrak{A}}_{n}(B)$ into one subspace of the FOL space, which is a coinfinite-dimensional complete closed set $s$ s.t. for a closed set $t$ representing $S^{\mathfrak{A}}_{n}(B)$ for an $n$, there is a mapping $\rho$ s.t. $\rho s=t$. The difference is that here an $n$-type is not a point, but an $\Delta(s)\cup n$-dimensional completely closed set. In particular, although $n$-type spaces over $\emptyset$ do not have a natural representation on the FOL space, such a merge space of the $n$-type spaces over $\emptyset$ do has. That is, an $\emptyset$-dimensional completely closed set.
\end{remark}
\subsection{The example}\label{5.2}
The example is the following theorem\footnote{The formulation of the theorem is referenced from Theorem 5.2.9 in \cite{marker2006model}}:
\begin{itemize}
    \item Let $\mathcal{L}$ be countable and $T$ be an $\mathcal{L}$-theory with infinite models. For all $\kappa\ge\aleph_{0}$, there is a $T$-model $\mathfrak{A}$ with $\|\mathfrak{A}\|=\kappa$ such that if $B\subseteq\mathfrak{A}$ and $n\in\mathbb{N}$, then $\mathfrak{A}$ realizes at most $\|B\|+\aleph_{0}$ types in $S^{\mathfrak{A}}_{n}(B)$.
\end{itemize}

First we consider the translation of two concepts: theory in a countable language and complete theory. Saying a language $\mathcal{L}$ is countable is equivalent to saying $\Lambda_{\mathcal{L}}$ is countable. Then let $T$ be an $\mathcal{L}$-theory, which means that on the $\alpha$-model space $\mathcal{C}_{T}^{\alpha}$ of an $\mathcal{L}$-theory $T$, clopen sets of dimension $I$ are countable for any finite $I\subseteq\alpha$. We call an FOL space with this property as a \textbf{locally-countable} FOL space. Saying a theory is complete is equivalent to saying for any sentence $\phi$, $\phi\in T$ or $\neg\phi\in T$, which means $l_{T}^{\alpha}(\phi)=\emptyset$ or $|\mathcal{C}_{T}^{\alpha}|$ for any $\alpha$. Then $\mathcal{C}_{T}^{\alpha}$ has no non-trivial $\emptyset$-dimensional clopen sets.  We call an FOL space with this property as a \textbf{complete} FOL space.

The rest of the discussion on the translation can be organized into the following proposition:

\begin{proposition}
Let $T$ be a theory and $\mathfrak{A}$ be a $T$-model of cardinality $\kappa$. For any nonempty set $B\subseteq|\mathfrak{A}|$ and $n\in\mathbb{N}$, the following statements are equivalent:
\begin{enumerate}
    \item $\mathfrak{A}$ realizes $\lambda$ many types in $S^{\mathfrak{A}}_{n}(B)$.
    \item For $x$ the model point representing $\mathfrak{A}$ in $\mathcal{C}_{T}^{\kappa}$, let $x'$ be defined as in the proof of Theorem \ref{ttm}, we define $I_{B}=\{i\in\kappa:x'(i)\in B\}$. Then for any $\kappa\backslash n$-dimensional complete closed set $s$ satisfying there is a partial mapping $\rho$ defined on $\Delta(s)$ s.t. $\rho s=a|_{I_{B}}$, there are only $\lambda$ many $b\in s$ s.t. there is mapping $\rho'\supset\rho$, $\rho:b\prec a$.
\end{enumerate}
\end{proposition}
\begin{proof} Let $f$, $g$ be defined as in the direction $2\Rightarrow 1$ of the proof of Theorem \ref{ttm}. For each $I=\{i_{0},...,i_{n-1}\}\subseteq\alpha$ ($i_{0},...,i_{n-1}$ in ascending order of subscript), defining mapping $\rho_{I}$ as $k\mapsto i_{k}$ for $k\in n$, $k\mapsto \rho(k)$ for $k\not\in n$, then for $p\in S^{\mathfrak{A}}_{n}(B)$ we have:

\begin{tabular}{rl}
     &$x'(i_{0}),...,x'(i_{n-1})$ realizes p\\
     $\Leftrightarrow$& for $\phi(\overline{v},x'(j_{1}),...,x'(j_{m}))\in p$, $\mathfrak{A}\vDash\phi(x'(i_{0}),...,x'(i_{n-1}),x'(j_{1}),...,x'(j_{m}))$\\
     $\Leftrightarrow$& for $\phi(\overline{v},,x'(j_{1}),...,x'(j_{m}))\in p$, $x\in S_{\phi(v_{i_{0}},...,v_{i_{n-1}},v_{j_{1}},...,v_{j_{m}})}$\\
     $\Leftrightarrow$&for $\phi(\overline{v},,x'(j_{1}),...,x'(j_{m}))\in p$, $\phi(\overline{v},v_{j'_{1}},...,v_{j'_{m}})\in\rho_{I}^{-1}x$ for $j'_{1},...,j'_{m}\in \rho^{-1}[j_{1}],...,\rho^{-1}[j_{m}]$\\
     $\Leftrightarrow$&$g(p)=\rho^{-1}_{I}x$ i.e. $\rho_{I} g(p)=a|_{I\cup I_{B}}$
\end{tabular}

Since $g$ is a bijection between $S^{\mathfrak{A}}_{n}(B)$ and $s$, the equivalence of statements holds.
\end{proof}

Obviously, if for any nonempty set $B\subseteq|\mathfrak{A}|$, $\mathfrak{A}$ realizes at most $|B|+\aleph_{0}$ types in $S^{\mathfrak{A}}_{n}(B)$, then $\mathfrak{A}$ realizes countable types in the type space over $\emptyset$. Based on the above discussion, we can translate the theorem into the following proposition.

\begin{proposition}
Let $\mathcal{C}_{\alpha}$ be an locally-countable complete FOL space satisfying $\bigcap_{i,j\in\alpha} -D^{\mathcal{C}}_{ij}\neq\emptyset$, there is a big model point $a\in\mathcal{C}_{\alpha}$ s.t. for any $I\subseteq\alpha$, any cofinite-dimensional complete closed set $s$ and any partial mapping $\rho$ defined on $\Delta(s)$ satisfying $\rho s=a|_{I}$, there are only $\|I\|+\aleph_{0}$ many $b\in s$ s.t. there is a mapping $\rho'\supset\rho$, $\rho':b\prec a$.
\end{proposition}
\begin{proof}
To simplify the discussion, we consider the dimension of $\mathcal{C}$ as a cartesian product $\alpha\times\omega$. 
\begin{itemize}
    \item We say a (partial) mapping $\rho:\alpha\times\omega\Rightarrow\alpha\times\omega$ is $\alpha$-order-preserving if for $(i_{1},j_{1}),(i_{2},j_{2})\in dom(\rho),(i'_{1},j'_{1})=\rho(i_{1},j_{1}),(i'_{2},j'_{2})=\rho(i_{2},j_{2})$, $i_{1}>i_{2}\Leftrightarrow i'_{1}>i'_{2}$, $j_{1}=j'_{1},j_{2}=j'_{2}$.
    \item For subset $s\subseteq|\mathcal{C}|$, we write $L(s):=\{\rho s:\rho$ is $\alpha$-order-preserving$\}$, $l(s):=\bigcap L(s)$.
\end{itemize}

\textbf{Claim: }there is a big model point $a$ s.t. $l(\{a\})=\{a\}$.

Let $a$ be the model point in this Claim, $I,\rho,s,n$ be arbitrary satisfying the relationship described in the proposition. The only difference is $\Delta(s)=(\alpha\times\omega)\backslash(1\times n)$ WLOG, we assume that for any $(i,j)\in I$, $i=\beta+m$ for some limit ordinal $\beta$ and natural number $m>n$. $I':=\{i'\in\alpha:$ there is $(i,j)\in I,i'=i-n$ for $n\in\mathbb{N}\}$, we know $\|I'\|+\aleph_{0}=\|I\|+\aleph_{0}$. We define $X:=\{a|_{t\cup I}:t\subseteq I'\times\omega$ and $\|t\|=n\}$, $Y:=\{b\in s:$ there is $\rho'\supseteq\rho,\rho'\{b\}\in X\}$. Obviously $\|X\|=\|Y\|=\|I\|+\aleph_{0}$. 

For an arbitrary $b\in s$ and a $\rho'\supseteq\rho$ satisfying $\rho' b\prec a$. If $\Delta(\rho' \{b\})\subseteq I'\times\omega$, then $b\in Y$. If $\Delta(\rho' \{b\})\backslash(I'\times\omega)\neq\emptyset$, we assume it contains only one element $(i',j')$ and for any $(i,j)\in\rho'[1\times n]\backslash\{(i',j')\}$, $i\neq i^{*}:=min\{i''\in I':i''>i'\}$. Based on the assumed property of $I$, this is no loss of generality. Then let partial mapping $\eta$ be defined as $(i',j)\mapsto (i^{*},j)$, for $j\in\omega$ and $(i,j)\mapsto(i,j)$ for $i\in I'\backslash\{i^{*}\},j\in\omega$, we have $\eta$ is $\alpha$-order-preserving and then $\eta a|_{s\cup I}=a|_{t\cup I}$ for $t=\{(i^{*},j')\}\cup\rho'[1\times n]\backslash\{(i',j')\}$. We know $a|_{t\cup I}\in X$ and $\eta\circ\rho' b=a|_{t\cup I}$. Hence, $b\in Y$. 

Hence, the proposition holds.

\textbf{Proof of Claim: }

By the following induction procedure, we construct a list of sets converging to the model point we want.
    \begin{itemize}
        \item\textbf{BS: } Let $\{w_{i}:i\in\omega\}$ list all nonempty clopen sets with dimension $1\times 1$. $U_{0}:=\{w_{i}(\frac{(0,2^{i})}{(0,0)}):i\in\omega\}$, $x_{0}:=\bigcap\{-D_{pq}:p,q\in\alpha\times 1\}\cap l(\bigcap U_{0})$. Clearly $x_{0}=l(x_{0})$ is non-empty.
        
        \item\textbf{IH: } $x_{m-1}=l(x_{m-1})$ is a nonempty set s.t. $x_{m-1}|_{\alpha\times(m-1)}$ is an $\alpha\times(m-1)$-dimensional complete closed set.
    \item\textbf{IS: } List $\{u:u$ is a clopen set s.t. $\omega\times(m-1)\subseteq\Delta(u)\subseteq \omega\times m\}$ as $\{u_{i}:i\in\omega\}$. Use recursion on $\omega$, we define a list of non-empty closed sets $x_{m-1}=y_{0}\supseteq y_{1}\supseteq...$ where $y_{i}=l(y_{i})$ for each $i$ . The $n$-th inductive step is:
    
    \begin{itemize}
        \item We consider the clopen set $u_{n}$. For each $\alpha$-order-preserving mapping $\rho$, $K^{\rho}:=\{i\in\alpha:(i,j)\in\Delta(\rho u_{n})$ for some j$\}$ and $k^{\rho}:=\|K^{\rho}\|$. Clearly $k^{\rho}=k^{\rho'}$ for any $\alpha$-order-preserving $\rho,\rho'$, so we write it as $k$. Let $H=(\alpha,[\alpha]^{k})$ be the complete $k$-hypergraph on $\alpha$. Choosing an arbitrary point $a\in y_{n-1}$, let hyperedge coloring $f:[\alpha]^{k}\rightarrow 2$ be defined as $f(I)=1\Leftrightarrow a\in \rho_{I} u_{n}$ where $\rho_{I}$ is $\alpha$-order-preserving and $K^{\rho_{I}}=I$. WLOG assume $a\not\in l(u_{n})$, then there is finite $X\subseteq L(u_{n})$ s.t. $a\not\in \bigcap X$. Since $y_{n-1}= l(y_{n-1})$, there is $h\in\mathbb{N}^{+}$ s.t. for any $I\subseteq\kappa$ with $\|I\|=h$, $a\not\in \bigcap\{\rho u_{n}: \rho$ is $\alpha$-order-preserving and $K^{\rho}\subseteq I\}$. This means there is no monochromatic clique of order $h$ colored $1$. Then by Ramsey's theorem for hypergraph, we know that for any $h'\in\alpha$, there is a monochromatic clique of order $h'$ colored $0$. Since $y_{n-1}= l(y_{n-1})$, this means for any $h'\in\alpha$, $a\in \bigcap\{\rho -u_{n}: \rho$ is $\alpha$-order-preserving and $K^{\rho}\subseteq h'\}$, and then, $a\in l(-u_{n})$. Let $y_{i}:=y_{i-1}\cap l(-u_{n})$.
    \end{itemize}
    
    $x'_{m}:=\bigcap_{i\in\omega}y_{i}$. By compactness, $x'_{m}=l(x'_{m})\neq\emptyset$. Clearly $x'_{m}|_{\alpha\times m}$ is a complete closed set of dimension $\alpha\times m$.
    
    List $\{u:u$ is a clopen set s.t. $(0,m)\in\Delta(v)\subseteq\{(0,m)\}\cup(\omega\times m)$ and $x'_{m}\subseteq[u]_{(0,m)}\}$ as $\{w_{i}:i\in\mathbb{N}^{+}\}$. Let $U_{m}:=\{w_{i}(\frac{(r_{i},p^{i}_{m+1})}{(0,m)}):p_{m+1}$ is the $(m+1)$-th prime$,i\in\mathbb{N}^{+}\}$ where $r_{i}=min\{i':$ there is a $j'$ s.t. $(i',j')\in\Delta([w_{i}]_{(0,m)})\}$. Let $x_{m}:=l(\bigcap U_{m})\cap x'_{m-1}$, by compactness, we know $x_{m}=l(x_{m})\neq\emptyset$.
    \end{itemize}
    
    Let $x:=\bigcap_{i<\omega}x_{i}$, we know $x$ is a singleton by compactness and $T_{2}$ property, know the point in it is a model point by the construction of $U_{m}$, know $x=l(x)$ by IH.
\end{proof}
\begin{remark}
The classical proof of this theorem involves many concepts, like the Skolem function, order-indiscernible, Ehrenfeucht–Mostowski model, etc., and while we do a rough count of the proof length, it can be found that the length of the complete proof (counting the proofs of all the antecedent propositions) is not less than five pages in many model theory textbooks. In contrast, our proof is slightly more than one page long in a similar typographic format and avoids defining many concepts. Even if all the translation processes are involved, the length is no more than four pages in this format. Moreover, since translation is only needed when the problem is introduced, the proportion of translation length is much lower when we use this method for more significant problems. This crude comparison of proof lengths at least illustrates our proof's lower complexity. And this is not because we have adopted a completely different proof idea. In fact, it is easy to see by comparison that the $\alpha$-order-preserving mapping plays a similar role to order-indiscernible in the proof, while the model point in the claim is a topological analog of the E–M model. And, as in the classical proof, we also use the S-theorem. The above facts highlight the distinct advantage of the new method as a homogeneous treatment.
\end{remark}

\section{Conclusion}
\textbf{The Topological Representation.} In the first part, Sections \ref{yzkj}, \ref{sbw}, and \ref{tph}, which are also the main body of this paper, a topological representation of semantics of first-order is presented. For a first-order theory $T$, all $T$-models, subsets of $T$-models, elementary embeddings, and partial elementary embeddings between $T$-models, are represented as points and purely topologically defined point-to-point relations in the model space or $\alpha$-model space of $T$, respectively. Moreover, isomorphic topological structures of these model spaces are depicted as FOL spaces. This way, the construction of semantics of first-order logic is transformed into the construction of abstract topological structures. We have discussed this representation fully in the previous sections. Here we add one more perspective on the topologization space.

By Theorems \ref{su} and \ref{map}, for a theory $T$, $T$-models can also be represented as basis-preserving C-mappings from topologization spaces to the model space of $T$. Elementary embeddings between $T$-models can also be represented as basis-preserving C-mappings between these topologization spaces. From the perspective of category theory, this representation can also be regarded as a generalization of the Stone duality. For the dual space of a Boolean algebra, a point $x$ in it can also be viewed as a continuous mapping from a single point topological space to $x$ (this is the dual morphism in the topological category of a homomorphism from Boolean space to 0-1 algebra). As a generalization of this observation, in category theory, a general element of an object is a morphism directed at it, and an object can be determined by its general elements of a certain type.\cite{awodey2010category} In this paper, a basis-preserving mapping from the topologization space of a first-order structure to an FOL space $\mathcal{C}$ can be regarded as a general element of $\mathcal{C}$.

\textbf{The application in model theory.}
As a discipline, model theory must deal not only with syntactic objects like formulas and theories but also with concrete mathematical structures like first-order structures. The objects it deals with are not homogeneous. In this paper, through the topological representation, we homogenize the heterogeneous objects treated by the model theory into components of the same kind of objects. This treatment brings with it the great advantage that we can express complex relations between heterogeneous objects in model theory by simple topological nations. This feature is amply demonstrated in the in our example. We know that many in-depth discussions in abstract model theory, including the classical proof of Morley's theorem, requires a series of complex definitions. This is evident in our example. the complexity of these definitions makes it difficult for a newcomer to have a precise and natural intuition of the connections between them. Here, all these concepts are reduced to primitive objects such as closed sets, points, and dimensions with certain properties, and their relationships are thus intuitively represented. In general, many of the not-so-fundamental, 'auxiliary line' constructions of first-order logic are transformed into some basic and natural intrinsic structures in FOL space.

In the application example, it is easy to see that our proof method is very syntactic. It is actually a Henkin construction process if translated into the language of model theory.  This phenomenon is common throughout the topological proof of Morley's theorem. What is the reason for this? It is easy to see that there is a very natural connection between cylindric space and cylindric algebra, so the link between cylindric space and first-order syntax is more direct and closed than for specific first-order structures. Topological 
treatment can therefore be seen to some extent as the 'syntacticization' of first-order semantic objects. Based on this, the inherent structure of FOL spaces brings an advantage: It significantly extends the range of applications of the syntactic approach in model theory. Furthermore, in the topological treatment, we are freed from the constraints of formal languages, which simplifies the complexity of expression when performing syntactic constructions in proofs.

While simplifying the syntactic concepts, the topological treatment does make some purely semantic concepts, such as models, subsets, elements of models, etc., less intuitive. However, these problems are solved by tools introduced in the first part of this paper.

The proof we give in the application is considerably simpler than its conventional counterpart. This simplification is a direct result of these features and tools mentioned earlier. And I think these features and tools work in essence because model theory, to a large extent, is a discipline that only studies the first-order syntactic properties of a mathematical structure and therefore does not rely as much on purely semantic concepts.

\textbf{Related works.}
We have already presented other works on the topological representation of semantics of first-order in Introduction. Here, we only discuss work related to topological methods or homogenization methods in model theory.

As mentioned above, our topological method can be seen as a homogenization approach in model theory. Of course, there are many ways of homogenization, but topological structures have unique advantages. For example, Section \uppercase\expandafter{\romannumeral2}.6.2 of \cite{picado2011frames} introduces the pointless topology method to transform a topological space into a complete lattice. Can we use this class of complete lattices as a means of homogenization? An obvious problem is that this structure lacks simple, effective means to distinguish open sets, closed sets, and points, which will cause many difficulties in the basic discussion. For example, the equivalence relations corresponding to quantifiers will be difficult to define.

Cylindric algebra is also a homogeneous tool that comes easily to mind. However, if one tries to translate the representation of semantics of first-order by the model space into the language of cylindric algebra, one finds that these concepts become very complex. Moreover, the language of algebra lacks the flexibility of the language of topology. For example, the concept of point permutation will be difficult to define.

However, some people have indeed developed an algebraic approach to model theory. The paper \cite{nemeti1988cylindric} proposed a set of methods to deal with model theoretic problems using cylindric set algebra. But its discussion is primarily based on the non-algebraic details of cylindric set algebra. For example, isomorphisms between first-order structures only correspond to base-isomorphisms involving specific $\alpha$-sequences. Then, this approach does not do anything about homogenization but simply transposes the heterogeneity in the classical approach into the algebraic approach. 

As mentioned in Introduction, the main idea of this paper is inspired by Pinter's work in \cite{pinter1980topological}. Pinter further developed a topological version of model theory based on ultrafilter space of $\omega$-dimensional local-finite dimensional cylindric algebra in \cite{pinter2016stone}. Different from us, the main work in \cite{pinter1980topological} and \cite{pinter2016stone} is algebraic, and the conclusion on the topological side is mainly the 
consequence of the algebraic side. Therefore, linking cylindric algebras of different dimensions is difficult, and relations like permutation are difficult to define. So the representation of uncountable models can only be specially constructed by a complex and unnatural method. However, some basic intuitions are similar. For example, the concept of ``model point" in \cite{pinter2016stone} can be regarded as a particular form of the concept of model point in this article.

\textbf{Future works.} 
As has been mentioned many times, in a forthcoming paper, we will prove a duality between the category of cylindric spaces with C-mappings between them and the category of cylindric algebras as a natural generalization of Stone duality.

By modifying the definition of cylindric space, we may construct duality or even topological model theory for variants of first-order logic. It will be exciting and may bring unexpected results.

Our work on the topological method in model theory is preliminary, but this method has shown considerable application potential in abstract model theory. Furthermore, it also has application potential when dealing with specific structures, for example, in algebraic model theory. This requires us to build on existing work by introducing distinguished clopen sets to represent basic predicates and functions.
\section*{Acknowledgments}
I am very grateful to Dr. Shengyang Zhong for his valuable and specific advice, assistance, and guidance in writing this paper.

\bibliographystyle{unsrt}  

\end{document}